\documentclass{amsart}
\usepackage[british,american]{babel}
\usepackage{bbm}
\usepackage{mathrsfs}
\usepackage{amsfonts, amsmath, wasysym}

\usepackage{amssymb,amsthm,
paralist
}
\usepackage{enumitem}
\usepackage{
latexsym,
nicefrac,
}
\usepackage[usenames]{color}

\usepackage{url}

\definecolor{darkgreen}{rgb}{0,0.5,0}
\definecolor{darkred}{rgb}{0.7,0,0}
\usepackage[colorlinks, 
citecolor=darkgreen, linkcolor=darkred
]{hyperref}
\usepackage{esint}
\usepackage{bibgerm}
\usepackage[normalem]{ulem}

\theoremstyle{plain}
\newtheorem{lemma}{Lemma}[section]
\newtheorem{thm}[lemma]{Theorem}

\newtheorem{cor}[lemma]{Corollary}
\newtheorem{ass}{Assumption}
\theoremstyle{definition}
\newtheorem{defn}{Definition}

\newtheorem{rmk}[lemma]{Remark}

\numberwithin{equation}{section}

\newcommand{\pn}{\partial_n}

\newcommand{\Chat}{\hat \C}
\newcommand{\al}{\alpha}

\newcommand{\ga}{\gamma}

\newcommand{\de}{\delta}
\newcommand{\om}{\omega}
\newcommand{\Om}{\Omega}

\newcommand{\La}{\Lambda}

\newcommand{\si}{\sigma}

\renewcommand{\th}{\theta}

\newcommand{\peps}{\partial_{\eps}}

\newcommand{\tu}{\modbb}

\newcommand{\R}{\ensuremath{{\mathbb R}}}
\newcommand{\N}{\ensuremath{{\mathbb N}}}

\newcommand{\C}{\ensuremath{{\mathbb C}}}

\newcommand{\norm}[1]{\Vert#1\Vert} 

\newcommand{\beq}{\begin{equation}}
\newcommand{\eeq}{\end{equation}}
\newcommand{\beqs}{\begin{equation*}}
\newcommand{\eeqs}{\end{equation*}}
\newcommand{\beqa}{\begin{equation}\begin{aligned}}
\newcommand{\eeqa}{\end{aligned}\end{equation}}
\newcommand{\beqas}{\begin{equation*}\begin{aligned}}
\newcommand{\eeqas}{\end{aligned}\end{equation*}}
\newcommand{\brmk}{\begin{rmk}}
\newcommand{\ermk}{\end{rmk}}
\newcommand{\partref}[1]{\hbox{(\csname @roman\endcsname{\ref{#1}})}}
\newcommand{\half}{\frac{1}{2}}
\newcommand{\thalf}{\tfrac{1}{2}}

\newcommand{\leqs}{\lesssim}
\newcommand{\geqs}{\gtrsim}

\newcommand{\abs}[1]{\vert#1\vert} 
\newcommand{\babs}[1]{\left\vert#1\right\vert}
 
\newcommand{\eps}{\varepsilon}
\newcommand{\na}{\nabla}

\newcommand{\TT}{\mathcal{T}}

\newcommand{\ddeps}{\tfrac{d}{d\eps}\vert_{\eps=0}}
\newcommand{\la}{\lambda}

\newcommand{\dist}{\text{dist}}

\newcommand{\Arg}{\text{Arg}}
\newcommand{\ZZ}{\mathcal{Z}}
\newcommand{\VV}{\mathcal{V}}
\newcommand{\HH}{\mathcal{H}}

\newcommand{\Id}{\mathrm{Id}}
\newcommand{\err}{\text{err}}
\newcommand{\errom}{{\err_\bb}}
\newcommand{\errga}{{\err_\ga}}

\newcommand{\Loj}{{\L}ojasiewicz }
\newcommand{\Lojns}{{\L}ojasiewicz}

\newcommand{\UU}{\mathcal{U}}
\newcommand{\DD}{\mathbb{D}}
\newcommand{\zz}{\mathfrak{z}}
\newcommand{\zzphi}{{\zz_1}}

\newcommand{\tgabm}{{\tilde \gamma_\bm}} 
\newcommand{\tgabb}{{\tilde \gamma_\bb}} 
\newcommand{\dps}{{n_0^*}} 
\newcommand{\dqs}{{n_1^*}} 
\newcommand{\bmu}{{U_0}} 

\newcommand{\bm}{{u_0}} 
\newcommand{\rmbm}{{q_0}} 

\newcommand{\rmbmc}{{q_0}} 
\newcommand{\rmbbmuc}{{q_{1,\mu}}}

\newcommand{\bbu}{{U_1}} 

\newcommand{\bb}{{u_1}} 

\newcommand{\rmbb}{{q_1}}

\newcommand{\rmmu}{{q_\mu}}

\newcommand{\modbb}{{j_1}} 
\newcommand{\modbm}{{j_0}}
\newcommand{\epssimu}{o(1)}

\newcommand{\rmbmhat}{{\hat q_0}} 
 
\newcommand{\rmmuhat}{{\hat q_\mu}} 
\newcommand{\rmbbhatmu}{{\hat q_{1,\mu}}}
\newcommand{\bbuhat}{\hat U_1}

\newcommand{\rmbbeps}{{q_{1}^{(\eps)}}}

\newcommand{\jrmbm}{{j_{0}}}
\newcommand{\hhompi}{{\beta}} 
\newcommand{\Acut}{{\hat A}}
\newcommand{\chiA}{{\mathbbm{1}_{A}}}
\newcommand{\chiAstar}{{\mathbbm{1}_{A^*}}}

\newcommand{\nutot}{{\bar\nu}} 
\newcommand{\gabb}{{\tilde \gamma_{\bbu}}}
\newcommand{\gabm}{{\tilde\gamma_{\bmu}}}
\newcommand{\mhalf}{{-\half}}

\newcommand{\ajbb}{{a_{j}(q_1)}}
\newcommand{\ajbm}{{a_{j}(q_0)}}
\newcommand{\akbb}{{a_{k}(q_1)}}

\newcommand{\uius}{{U_i^*}}

\newcommand{\etazero}{{\eta_0}} 
\newcommand{\etaz}{{\eta_0}} 

\newcommand{\etatau}{{\eta_\tau}} 
\newcommand{\etatilt}{\abs{\peps d\beta(0)}}
\newcommand{\phimul}{{\phi_{r_0,r_1}}}
\newcommand{\phimu}{{\phi_{\mu}}}
\newcommand{\cmu}{{c_{\mu}}}
\newcommand{\cmul}{{c_{r_0,r_1}}}

\newcommand{\HHz}{{\mathcal H_0}}

\newcommand{\RR}{\mathcal{R}} 

\newcommand{\tauS}{\tau_{g_{S^2}}}
\newcommand{\fmu}{{f_\mu}} 

\newcommand{\Ctilt}{{\norm{\beta}}} 
\newcommand{\sitilt}{ \abs{d\beta(0)}} 
\newcommand{\errcut}{\err_{\varphi}} 

\newcommand{\Emin}{E_{\text{min}}}
\title{\sc Low energy levels of harmonic maps into analytic manifolds}
\author{Melanie Rupflin}
\date{\today}
\begin{document}

\begin{abstract}
We consider the energy spectrum $\Xi_E(N)$ of harmonic maps from the sphere into a closed Riemannian manifold $N$. While a well known conjecture
asserts that $\Xi_E(N)$ is discrete whenever $N$ is analytic,  for most analytic targets it is only known that any potential accumulation point of the energy spectrum must be given by the sum of the energies of at least two 
harmonic spheres. The lowest energy level that could  hence potentially be an accumulation point of $\Xi_E$ is thus $2\Emin$. In the present paper we exclude this possibility for generic 3 manifolds and prove additional results that establish obstructions to the gluing of harmonic spheres and 
 \Lojns-estimates for almost harmonic maps.
\end{abstract}

\maketitle

\section{Introduction}\label{sect:intro}
Let $(N,g_N)$ be a closed Riemannian manifold which we can 
assume without loss of generality to be embedded in a suitable Euclidean space $\R^N$ using Nash's embedding theorem.
We recall that a map $u:S^2\to N$ 
is called harmonic if it is a critical point of the Dirichlet energy
$$E(u)=\half \int_{S^2} \abs{\na u}^2 dv_{g_{S^2}}.$$ 
Such harmonic maps are characterised by $\tauS(u)=0$ where the tension field is given by $\tauS(u)=-\na^{L^2}E(u)=\Delta_{g_{S^2}} u+A(u)(\na u,\na u)$, $A$ the second fundamental form of $N\hookrightarrow \R^N$. 

We recall that any harmonic map from $S^2$ is (weakly) conformal and hence that the set of harmonic maps from the sphere coincides with the set of (weakly) conformal parametrisations of minimal spheres in $N$. 
In particular, the energy spectrum
\newcommand{\ESpec}{\Xi_E(N,g_N)}
\newcommand{\ASpec}{\Xi_{Area}(N)}
$$\ESpec := \{E(u): u:S^2\to N \text{ harmonic, not constant }\}$$
of harmonic spheres agrees with the areas (counted with multiplicity) of (possibly branched) immersed minimal 
spheres.

While it is easy to construct smooth manifolds $(N,g_N)$ for which $\Xi_E(N,g_N)$ has accumulation points, a well known conjecture of Leon Simon and of Fang-Hua Lin \cite{Lin-conjecture} asserts that the energy spectrum of harmonic maps from $S^2$ into any closed analytic manifold is  discrete. 

In some very special cases this follows as a consequence of much stronger results such as an explicit characterisation of 
all minimal spheres in $N$ and the resulting explicit knowledge of $\Xi_E(N)$. E.g.~if the target is the round sphere then 
all harmonic spheres are described in stereographic coordinates $z\in \C$ by rational maps in either $z$ or $\bar z$ and hence $\Xi(S^2,g_{S^2})$ is made up by  the multiples of $4\pi$.

Conversely, for more general analytic targets very little is known about $\Xi_E$ with the few existing results following already from the compactness theory of (almost) harmonic maps developed in the 1990s in  \cite{Struwe1985, DT, QingTian, LinWang}
and the seminal result of Simon \cite{Simon} from 1983. 

To describe the known properties of $\Xi_E$ we first recall that sequences $u_j$ of maps with bounded energy which are almost harmonic in the sense that 
\beq
\label{def:almost-harmonic}
\norm{\tauS(u_j)}_{L^2(S^2,g_{S^2})}\to 0
\eeq
subconverge to a bubble tree without loss of energy or formation of necks: That is, there exists a harmonic base map $\om_0:S^2\to N$ and a finite set $S\subset S^2$  so that 
$u_j\to \om_0$ strongly in $H_{loc}^2(S^2\setminus S)$ and weakly in $H^1(S^2)$ and so that near each point $p\in S$ the $u_j$ are essentially described by a collection of highly concentrated harmonic spheres. Namely, for each such $p$ there exists a finite collection of harmonic maps $\om_k:\R^2\to N$, points $p_j^k\to p$ and scales $\mu_j^k\to \infty$ so that, working in stereographic coordinates centred at $p$, 
\beq
\label{eq:bt-convergence}
u_j(x)-\sum \big( \om_k(\mu_j^k(x-p_j^k))-\om_k(\infty)\big)-\om_0(x) \to 0 \text{ strongly in } L^\infty\cap H^1(B_r(p)).\eeq
We recall that 
 the conformal invariance of the energy and the result \cite{SU} of Sacks-Uhlenbeck imply that a finite energy map $u:\R^2\to N$ is harmonic if and only if $u\circ \pi^{-1}:S^2\to N$ is a harmonic map from the sphere, $\pi(x)=(\frac{2x}{1+\abs{x}^2},\frac{1-\abs{x}^2}{1+\abs{x}^2}):\R^2\cup \{\infty\}\to S^2$  the inverse  stereographic projection. 
In the following we can thus switch viewpoint 
and work on $\R^2$, or equivalently on $\hat\C=\C\cup\{\infty\}$, whenever this is more convenient and, except in the trivial case where $E(u_j)\to 0$, can pull-back a sequence of maps $u_j$ satisfying \eqref{eq:bt-convergence} by suitable M\"obius transforms to ensure that the base map $\om_0$ is non-constant. 
 
If no bubbles form, i.e. if the maps converge strongly in $H^2(S^2)$, and if $N$ is analytic then the seminal results \cite{Simon} of Simon are applicable. These ensure that for every 
harmonic map $u^*:S^2\to N$ there exists an $\eps=\eps(u^*)>0$ so that 
a
\Lojns-estimate
of the form 
\beq
\label{est:Simon}
\abs{E(u)-E(u^*)}\leq C\norm{\tauS(u)}_{L^2(S^2,g_{S^2})}^\gamma \text{ for some } \gamma=\gamma(u^*)\in (1,2]
\eeq
holds true for all maps in the neighbourhood
$$\UU(u^*)=\{u:S^2\to N: \norm{u-u^*}_{L^\infty}<\eps \text{ and }  \norm{u-u^*}_{H^1}<\eps \}.$$  
Conversely these results cannot be applied to study sequences of harmonic maps that undergo bubbling as such maps will not be $H^1$ close to a fixed harmonic map, or indeed to any set of harmonic maps for which the results of \cite{Simon} yield \Lojns-estimates on a \textit{uniform} $\eps$-neighbourhood, compare Remark \ref{rmk:spectral-gap}.  To the authors knowledge the only properties of the energy spectrum of harmonic maps into general analytic manifolds that are currently known are the ones that follow from the above mentioned compactness theory and from the work of Simon, i.e.~that 
\begin{itemize}[topsep=0pt]
\item The lowest non-trivial energy level 
$\Emin:=\inf \Xi_E(N)$
is achieved and isolated in the energy spectrum.
\item Any potential accumulation point of $\Xi_E(N)$ must be of the form 
$\bar E=\sum_{i=0}^n E(\om_i)$ for some non-trivial harmonic spheres $\om_i$, $i=0,\ldots,n\geq 1$, and such an accumulation point must be due to the existence of harmonic maps $u_j$  with energies $\bar E\neq E(u_j)\to \bar E$ that converge to a non-trivial bubble tree. 
\end{itemize}

The lowest level which might potentially be an accumulation point of the energy spectrum is hence $2\Emin$ and so a natural starting point to develop a better understanding of $\Xi_E$, and to gain new insight into the above mentioned conjecture,
is to investigate whether energies of harmonic maps might accumulate at  $2\Emin$. 
In the present paper we exclude this possibility for $3$ manifolds under some natural assumptions on the minimal energy harmonic spheres

\begin{thm}\label{thm:E}
Let $(N,g_N)$ be any analytic $3$ manifold  and let $\Emin=\min \Xi_E(N,g_N)$ be the minimal energy of a non-trivial harmonic sphere in $N$. 
Suppose that harmonic spheres $\om:S^2\to N$ with this energy $\Emin$ are unbranched, non-degenerate critical points and that any intersection or self-intersection  between such minimal spheres is transversal in the sense that the corresponding tangent spaces do not coincide. \\
Then $2\Emin$ cannot be an accumulation point of the energy spectrum $\Xi_E(N,g_N)$. 
\end{thm}

As the energy is invariant under conformal changes we know that any vector field of the form $Y=\ddeps \,\om\circ M_\eps$, $M_\eps$ a family of M\"obius transforms with $M_{\eps=0}=\Id$, is 
 a Jacobi field along $\om$, i.e.~so that $d^2E(\om)(Y,\cdot)=0$. 
 We hence say that $\om$ is non-degenerate if all Jacobi-fields along $\om$ are generated in this way.

We recall that the results of Gulliver, Osserman and Royden \cite{GOR} ensure that minimial surfaces in $3$-manifolds that are locally area minimising cannot have any true branch points. Any harmonic sphere with energy $\Emin$ that is a local area minimiser will hence automatically be unbranched. Of course the example of an equatorial sphere in $S^3$ shows that harmonic spheres of minimal energy might not be local minimisers of the Area and to the author's knowledge it is not known whether there any minimal energy harmonic spheres in a  $3$ manifold that are branched.

The above theorem is valid also for manifolds which are not analytic but for which it is known that prime harmonic spheres of energy $2\Emin$ are non-degenerate critical points. 

As the bumpy metric theorems obtained by White in \cite{White-bumpy-1, White-bumpy-2, White-transversal} and by Moore in \cite{Moore} and in Theorems 5.1.1 and 5.1.2 of \cite{Moore-book}, ensure that for generic metrics on manifolds of dimension at least $3$ all prime harmonic maps are non-degenerate critical points, not branched and that all potential intersections and self-intersections are transverse we hence obtain

\begin{cor}
For generic $3$ manifolds the possibility that
$2\Emin$ is an accumulation point of the energy spectrum $\Xi_E(N)$ is excluded and \Lojns-estimates as stated in Theorems \ref{thm:1}, \ref{thm:2} and \ref{thm:3} below hold true for any sequence of almost harmonic maps that converges to a bubble tree with a single bubble. 
\end{cor}
We remark that while standard arguments show that the \Lojns-estimate \eqref{est:Simon} holds true (with $\gamma=2$) in an $\eps=\eps(u^*)>0$ neighbourhood of any non-degenerate harmonic map $u^*$ into a smooth manifold, these arguments cannot be used to obtain the above result as sequences of harmonic maps undergoing bubbling will not be in a neighbourhood of a fixed such map.


%

To prove Theorem \ref{thm:E} and the above corollary we need to exclude the possibility that there exists a sequence of harmonic maps $u_j:S^2\to N$ with energy $2\Emin\neq E(u_j)\to 2\Emin$. We know that such a sequence cannot converge strongly as that would contradict \cite{Simon}, so (after pull-back by suitable M\"obius tranforms) must converge to a bubble tree with a base map $\om_0$ and a bubble $\om_1$ of energy $E(\om_0)=E(\om_1)=\Emin$.

We hence need to ask whether it is possible to glue increasingly concentrated harmonic spheres onto harmonic base maps in a way that results in a harmonic map whose energy is close, but not equal, to $E(\om_0)+E(\om_1)$. 
We will not only exclude this for harmonic spheres as considered in Theorem \ref{thm:E}, but will establish obstructions to gluing harmonic spheres that apply in more general settings, including situations where the involved harmonic spheres are branched and have non-trivial, and even non-integrable, Jacobi-fields. 

It is natural to distinguish between bubble tree for which the bubble $\om_1$ and base $\om_0$ 
\begin{enumerate}[topsep=0pt]
\item parametrise the same minimal surface with the same orientation \label{case:I}
\item  parametrise the same minimal surface with the opposite orientation\label{case:II}
\item \label{case:III} parametrise transversally intersecting minimal surfaces.
\end{enumerate}
We note that the no-neck property of the convergence to a bubble tree ensures that we only have to consider maps $\om_0$ and $\om_1$ for which $\om_0(p)=\om_1(\infty)$, $p$ the point at which the bubble forms. After rotating the domain, we can assume without loss of generality that the bubble is attached at the north pole,  so in stereographic coordinates at $z=0$. 
In the following we  hence only have to consider maps $u_j:\Chat\to N$ for which there exist $\mu_j\to \infty$ and $z_j\to 0$ for which 
\beq
\label{eq:bt-conv-2}
u_j(z)-\om_0(z)-[\om_1(\mu_j(z-z_j))-\om_1(\infty)]\to 0 \text{ in }\dot H^1\cap L^\infty(\C).
\eeq 
Our analysis is not restricted to bubble trees for which $\om_0$ and $\om_1$ are unbranched  but we can more generally  consider bubble trees with base maps and bubbles which are obtained as a composition of a harmonic sphere $U^*:\Chat\to N$ with $dU^*(q^*(0))\neq 0$ and a rational map $q^*:\Chat\to \Chat$ of arbitrary degree, i.e.~a map of the form $q^*(z)=\frac{p_1(z)}{p_2(z)}$, $p_{1,2}$ polynomials. As we can precompose $U^*$ with a rotation we can restrict our attention to rational maps $q^*(z)$ which map $0$ to $0$ and so will always ask that $dU^*(0)\neq 0$. 

We note that the conformal invariance of the energy 
ensures that any map of the form $U\circ q$, $U$ harmonic and $q$ rational, will be a harmonic map and that any vector field of the form $Y=\ddeps (U\circ q_\eps)$, $q_\eps$ rational maps,  is a Jacobi field along $U\circ q_{\eps=0}$. For higher order coverings of harmonic spheres it is hence natural to say that 
\begin{defn}
\label{def:non-deg}
The second variation of the energy is non-degenerate along a harmonic map 
$\om$ of the form $\om=U\circ q$ if every Jacobi-field $Y$ along $\om$ is generated by a variation of rational maps $q_{\eps}$, i.e.~given by $Y=\ddeps (U\circ q_\eps)$.
\end{defn}
We first consider the case where the base and the bubble parametrise the same minimal surface with the same orientation.
Here we include settings in which $U^*(0)$ is a point of higher multiplicity of an immersed minimal sphere 
$ U^*(\Chat)$ provided $\om_0$ and $\om_1(\frac{1}{z})$ parametrise the same leaf of $U^*(\Chat)$ near $z=0$, but will instead include settings in which they parametrise transversally intersecting leafs in the third case, compare Theorem \ref{thm:3} below. 

So suppose that $\om_0$ and  $\om_1$ are obtained by composing the same harmonic  $U^*:\Chat\to N$ with rational maps $q_0^*(z)$ and $q_1^*(\frac{1}{z})$ for which $q_0^*(0)= q_1^*(0)=0$. 
We want to show that the only sequences of harmonic maps $u_j$ that converge to such a bubble tree are higher degree coverings of  $U^*:S^2\to N$.

This does not follow from the existing theory since 
the size of the neighbourhoods on which \Lojns-estimates are known to hold true shrink to zero as the maps undergo bubbling, compare Remark \ref{rmk:spectral-gap}.
If $u_j$ is a sequence of harmonic maps that converges to such a bubble tree then the compactness theory from  \cite{Struwe1985, DT, QingTian, LinWang} ensures that 
 $\norm{u_j-U^*\circ r_j}_{L^\infty \cap \dot H^1}\to 0$ for some rational 
  maps $r_j$. However the existing theory does not provide the quantitative estimates on the rate of this convergence which would allow one to know that the maps  $u_j$ are in the smaller and smaller neighbourhoods $\UU(U^*\circ r_j)$ for which the results of Simon apply.

Conversely, our method allows us to prove that \Lojns-estimates indeed hold true on a \textit{uniform} $\eps$ \text{neighbourhood} around such a non-compact set of harmonic maps $\{U^*\circ r\}$ and hence that such sequences of harmonic maps cannot be responsible for an accumulation point of the energy spectrum. To be more precise, we show

\begin{thm}\label{thm:1}
Let $ (N,g_N)$ be a smooth Riemannian manifold of any dimension. Suppose that $U^*:\Chat \to N$ is a harmonic map with $dU^*(0)\neq 0$ and ${q_0^*},{q_1^*}:\Chat \to \Chat$ are rational maps with $q_1^*(0)=q_0^*(0)=0$ so that 
$d^2E$  is non-degenerate at 
$\om_0(z):=U^*\circ \rmbm^*$ and $\om_1(z):= U^*\circ \rmbb^*(\frac1{ z})$ in the sense of Definition \ref{def:non-deg}.
\\
Let  $(u_j)$ be any sequence of almost harmonic maps which converges to a bubble tree with base map $\om_0$ and bubble $\om_1$ as described in \eqref{def:almost-harmonic} and \eqref{eq:bt-conv-2}. Then there exists a constant $C$ so that
\beq
\label{claim:Loj-1-rmk}
\abs{E(u_j)-[E(\om_0)+E(\om_1)] }\leq C \min(\norm{dE(u_j)}_*,\log\mu_j \norm{dE(u_j)}_*^2),\eeq
 for all sufficiently large $j$  and for $\mu_j\to \infty$ as in \eqref{eq:bt-conv-2}. 
 \\
Furthermore, if the maps $u_j$ are harmonic, 
then, after precomposing with a suitable rotation $R_j$ of the domain, they can be written as
$$u_j(z)=U^* \big(q_0^j(z)+q_1^j(\tfrac{1}{\mu_j z})\big)$$
for rational maps $q_i^j$  which converge smoothly to $q_i^*$ when viewed as maps from $S^2$ to $S^2$.
\end{thm}

Here and in the following $\norm{\cdot}_*$ denotes a weighted $H^{-1}$ norm which is defined in Remark \ref{rmk:def-dual-norm} and which is chosen so that the above estimate \eqref{claim:Loj-1-rmk} is invariant under pull-back by M\"obius transforms. We note that  
$\norm{\tauS(u)}_{L^2(S^2,g_{S^2})}$, which dominates $\norm{dE(u)}_*$, does not have this property. 
However while it is natural work with $\norm{dE(u)}_*$ 
to analyse the energy spectrum,
it also of interest to obtain $L^2$-\Lojns-estimates e.g. to analyse the asymptotic behaviour of the harmonic map heat flow 
\beq\label{def:HMF}
\partial_t u=\tau_{g_{S^2}}(u).
\eeq
We shall hence also prove
\begin{cor}
\label{cor:thm1}
For sequences of maps $u_j$ as in Theorem \ref{thm:1} we can bound 
\beq
\label{claim:thm1-Loj}
\abs{E(u_j)-E^*}\leq C \norm{\tau_{g_{S^2}}(u)}_{L^2(S^2,g_{S^2})}^2. \eeq
\end{cor}

A consequence of this corollary is that solutions $u(t)$ of the flow \eqref{def:HMF} for which sequences $u(t_j)$, $t_j\to \infty$, are known to subconverge to a bubble tree as considered 
 in Theorem \ref{thm:1}
must indeed converge exponentially fast to a unique base map $u_\infty:S^2\to N$ as $t\to \infty$ away from a unique point $p_\infty\in S^2$ in the sense of both $L^2(S^2,g_{S^2})$ and $C^k_{loc}(S^2\setminus \{ p_\infty\})$.

In the one bubble case this extends results that were previously obtained by Topping \cite{Topping-aligned} in the case where $N$ is the round $2$-sphere and that were later extended by Liu and Yang in \cite{Kaehler} to compact Kähler manifolds with nonnegative holomorphic bisectional curvature. We note that the results from \cite{Topping-aligned, Kaehler} are not restricted to the one bubble case but apply to sequences of maps and solutions of the flow which converge to any bubble tree for which the base map and the bubbles are all parametrisations of the same minimal spheres with the same orientation. 

The above corollary suggest that while the special geometric structure of the target is crucially used in the proofs of \cite{Topping-aligned, Kaehler}, analogue results might be valid for more general targets and 
 one would expect that the main property needed  is  the non-degeneracy of the second variation of the energy at the underlying harmonic spheres $U^*\circ q_i^*$, which for $N=S^2$ follows from \cite{Gulliver-White}.

Next we want to show that it is impossible to glue a highly concentrated bubble onto a base map which parametrises the same minimal surface, but with opposite orientation. Indeed, we will prove more generally that for almost harmonic maps that converge to such a bubble tree the energy defect and the bubble scale are controlled in terms of $\norm{dE(u)}_*$.

\begin{thm}\label{thm:2}
Let $ (N,g_N)$ be a smooth Riemannian manifold of any dimension, let $U^*:\Chat \to N$ be any harmonic map with $d U^*(0)\neq 0$ and let $q_0^*,q_1^*:\Chat \to \Chat$ be any rational maps with ${q_1^*}(0)={q_0^*}(0)=0$ so that 
so that 
$d^2E$  is non-degenerate at 
$\om_0(z):=U^*\circ \rmbm^*$ and $\om_1(z):= U^*\circ \rmbb^*(\frac1{\bar z})$ in the sense of Definition \ref{def:non-deg}.

Let  $(u_j)$ be any sequence of almost harmonic maps  which converges to a bubble tree with base map $\om_0$ and 
 bubble $\om_1$ as described in \eqref{def:almost-harmonic} and \eqref{eq:bt-conv-2}. Then the bubble scale must satisfy 
$$\mu_j^{-\min(\dps,\dqs)}\leq C\norm{dE(u_j)}_*,$$
 $n_i^*\in \N$ the order of the zero of $q_i^*$ at $z=0$, while the energy defect is controlled by 
 $$\abs{E(u_j)-[E(\om_0)+E(\om_1)]}\leq C\norm{dE(u_j)}_*.
 $$
In particular, no sequence of harmonic maps can converge to such a bubble tree. 
\end{thm}

For maps to the round sphere $L^2$-\Lojns-estimates with exponent $\gamma=2$ and exponential bounds on the bubble scale were proven by Topping in \cite{Topping-quantisation}. In this setting there are only two types of  harmonic spheres, namely holomorphic or antiholomorphic maps. The results of \cite{Topping-quantisation} apply for maps that converge to any bubble tree which is so that for each singular point $p$ the bubbles forming at $p$ are all of the same type, where for points $p$ at which bubbles form that have a different type than the base map one needs to impose  additionally  that the base map is not branched.

Topping already observed in \cite{Topping-zero-density}
that even for maps into $(S^2,g_{S^2})$ one cannot expect to prove $L^2$-\Lojns-estimates with exponent $\gamma=2$ if the base map is branched at a point where a bubble of a different type forms.  

At the same time, any $L^2$-\Lojns-estimate with an exponent $\gamma>1$ is sufficient to prove convergence results for solutions of harmonic map flow. As this is our main motivation to also prove \Lojns-estimates that involve $\norm{\tauS(u)}_{L^2(S^2,g_{S^2})}$ rather than the weaker, and scaling invariant $\norm{dE(u)}_*$, we shall hence simply show

\begin{cor}
\label{cor:thm2}
For any sequence $(u_j)$ of almost harmonic maps as in Theorem \ref{thm:2} we can bound 
\beq
\label{claim:cor2}
\abs{E(u_j)-[E(\om_0)+E(\om_1)]}\leq C \norm{\tau_{g_{S^2}}(u_j)}_{L^2(S^2,g_{S^2})}^\gamma
\eeq
for an exponent $\gamma>1$ that only depends on $U^*$ and $q_{0,1}^*$. 
\end{cor}

We finally want to show that for maps into $3$ manifolds it is impossible to glue two harmonic spheres $\om_0$ and $\om_1$ which are so that $\om_0$ and $\om_1(\frac{1}{z})$ parametrise transversally intersecting minimal surfaces, respectively transversally intersecting leafs of the same minimal surface, near $z=0$.  

Our result in this setting 
also applies for harmonic spheres 
which are degenerate critical points, i.e.~which have Jacobi-fields that are not generated by variations of rational maps. Indeed, we can  
even deal with harmonic spheres $\om_i$ which have  non-integrable Jacobi-fields, i.e. for which $d^2E(\om_i)$ has null directions that are not tangential to the set of harmonic maps near $\om_i$. 

The one thing we want to ask is that for maps $\om_i$ that are obtained as higher order coverings of a map $U_i^*$ this structure is reflected also at the level of Jacobi-fields. Namely, if $q_i^*$ has degree strictly greater than $1$ then we ask that 
 all Jacobi-fields along $\om_i=U_i^*\circ q^*$ are generated by a combination of Jacobi-fields along $U_i^*$ and Jacobi-fields that are generated by variations of the  rational maps, i.e.~
\beq
\label{ass:eq-Jabobi}
\ker(L_{\om_i})=\{W\circ q^*, \, W\in \ker(L_{U_i^*})\} +T_{\om_i}\{U_i^* \circ q: q:S^2\to S^2 \text{ rational } \},
\eeq
$L_u$ the Jacobi operator along $u$.
Here we note that the fact that $W\circ q^*$ is a Jacobi-field along $U^*\circ q^*$ whenever $W$ is a Jacobi-field along $U^*$ follows from the conformal invariance of the energy and the resulting formula
\beq
\label{eq:trafo-tension}
\tau(U\circ q)=\thalf \abs{\na q}^2 \tau(U\circ q)
\eeq
for the transformation of the tension under conformal changes. 

As our final main result we prove

\begin{thm}\label{thm:3}
Let $ (N,g_N)$ be any analytic $3$ manifold and let $U_{0,1}^*:\Chat\to N$ be any harmonic spheres with $U_1^*(0)=U_0^*(0)$ and $d U_{0,1}^*(0)\neq 0$ which are so that the tangent spaces 
 $T_{U_i^*(0)} U_i^*(\Chat)$ to the corresponding minimal surfaces $U_i(\Chat)$ do not coincide.
  Let $q_0^*$, $q_1^*$ be 
  any rational maps 
with $q_0^*(0)=q_1^*(0)=0$, assumed to satisfy \eqref{ass:eq-Jabobi} if their degree is at least $2$.

 Then  for any sequence $(u_j)$ of almost harmonic maps that converges to a bubble tree with base map $\om_0=U_0^*\circ q_0^*$ and bubble $\om_1(z)=U_1^*(q_1^*(\frac{1}{z}))$ as described in \eqref{def:almost-harmonic} and \eqref{eq:bt-conv-2} we can bound 
the bubble scale by 
$$\mu_j^{-\min(\dps,\dqs)}\leq C\norm{dE(u_j)}_*$$
and we have a 
 \Lojns-estimate of the form 
 $$\abs{E(u)-[E(\om_1)+E(\om_0)]}\leq C\norm{dE(u_j)}_*.
 $$
In particular, no sequence of harmonic maps can converge to such a bubble tree. 
\end{thm}

\begin{rmk}
If $U_{0,1}^*$ are non-degenerate critical points, then the above theorem remains valid also without the assumption that $(N,g_N)$ is analytic. 
\end{rmk}

We note that in contrast to Theorems \ref{thm:1} and \ref{thm:2} which are valid for target manifolds of arbitrary dimension, here we have to impose that the target is $3$ dimensional. We cannot expect to obtain the same repulsion effect if we were to consider maps into higher dimensional manifolds such as 
 $N=S^2\times S^2\hookrightarrow \R^6$ where e.g. $u_j(\cdot):=(\pi(\cdot), \pi(\frac1{\mu_j \cdot})):\R^2\to N$, $\mu_j\to \infty$, are harmonic maps which converge to a bubble tree for which the tangent spaces of the base $\om_0=(\pi,\pi(\infty))$ and the bubble $\om_1=(\pi(\infty),\pi(\frac{1}{\cdot}))$ intersect transversally in a point. 

\textbf{Outline of the paper:}\\
To prove our main results we use a general method that was first developed in the joint work \cite{MRS} of Malchiodi, Sharp and the author on $H$-surfaces. 
A crucial aspect of this method is that it allows us to prove \Loj estimates
for sequences of (almost) critical points of an energy $E$ that converge to a singular limit, here a bubble tree, without having to analyse the properties of such general (almost) critical points. 

 Instead, if we can construct a suitable finite dimensional (non-compact) manifold $\ZZ$ of \textit{singularity models} so that the restriction of the energy to $\ZZ$ has the right properties, then this method allows us to obtain \Lojns-estimates that are valid on a \textit{uniform} $\eps$ \text{neighbourhood} of $\ZZ$ based on the properties of the energy and its variation on $\ZZ$.

These singularity models will in general not be critical points of the energy, but will always be so that $\norm{dE(\zz)}_* $ is small as they serve as models for the behaviour of almost critical points. The key properties $\ZZ$ needs to satisfy are that
\begin{itemize}[topsep=0pt]
\item[(i)] The second variation of the energy is uniformly definite orthogonal to $\ZZ$
\item[(ii)] For each $\zz\in \ZZ$ which is not a critical point of $E$  we can identify a direction on $\ZZ$ in which the variation of $E$ has a sign and suitable scaling. To be more precise, we need that 
 for each such $\zz\in \ZZ$ there exists a unit element $y_\zz\in T_\zz\ZZ$  so that 
\beq \label{est:key-property}
Q(\zz,y_\zz):= \frac{  \norm{dE(u)}_*\big(\norm{dE(u)}_*+\norm{d^2E(u)(y_\zz,\cdot)}_*\big)}{dE(u)(y_\zz)} \text{ is small .}\eeq
\end{itemize}

Our proofs hence consist of three main steps: 
the construction of the manifold $\ZZ$ of singularity models,  the analysis of the energy and its variations on $\ZZ$
and finally the arguments of how these properties of $E$ on $\ZZ$ yield our main results.
After explaining the construction of the singularity models in Section \ref{sect:Z} we state the relevant properties of $E$ on $\ZZ$ 
in a series of lemmas in the subsequent Section \ref{sect:key-lemmas}, but postpone the rather technical proofs until Sections \ref{section:proof-of-lemmas} and \ref{sect:variations} as an understanding of these proofs is not required to complete the proofs of the main results which are carried out in Section \ref{sect:main-proofs}. 


\section{Construction of the singularity models}\label{sect:Z}

In this section we construct the singularity models 
which are maps $\zz:S^2\to N$ that model the behaviour of sequences of almost harmonic maps that converge to a bubble tree 
with base  $\om_0=U_0^*\circ q_0^*$  and bubble  $\om_1(z)=U_1^*\circ q_1^*(\frac{1}{z})$
where here and in the following  
 $U_i^*:\Chat\to N$ always denote harmonic maps with $d U_i^*(0)\neq 0$ and  $q_i^*$ are elements of 
\beq
\label{def:R}
\RR:=\{q:\Chat\to\Chat \text{ rational with } q(0)=0\}.
\eeq
It suffices to construct singularity models $\zz:S^2\to N$ for which the bubble is attached at the north pole and we denote the corresponding set of maps by $\ZZ_0$. Once we have constructed $\ZZ_0$ we can then obtain the full manifold of singularity models 
as 
\beq\label{def:ZZ-full}
\ZZ:=\{\zz_0\circ R_p: p\in S^2, \, \zz_0\in \ZZ_0\},
\eeq  where $R_p\in SO(3)$ denotes the rotation which maps the plane containing $0$, $p$ and the north pole $\pi(0)$ to itself and which is so that $R_p(p)=\pi(0)$ (with the convention that  $R_p=\Id$ if $p=\pi(0)$).

To construct these maps $\zz\in \ZZ_0$ we work in fixed stereographic coordinates where the base and bubble will be represented by harmonic maps $U_i^*:\Chat\to N$ and rational maps $q_i^*:\Chat\to \Chat$. 
As discussed in the introduction we need to analyse the three cases where either $U_0^*=U_1^*$ or $U_0^*(\bar z)=U_1^*(z)$  or where $U_{0,1}^*$ parametrise minimal spheres in a 3-manifold that intersect transversally in $U_1^*(0)=U_0^*(0)$. 

We can assume without loss of generality that the rational maps $q_{0,1}^*$ are so that the leading order coefficients in $q_i^*(z)= a_{n_i^*}(q_i^*) z^{n_i^*}+\sum_{j>n_i^*} a_{j}(q_i^*) z^j$ are so that $\abs{a_{n_i^*}(q_i^*)}=1$. 
Indeed, to ensure this for $q_0^*$ we can 
simply replace $U_0^*$ by $U_0^*(c \cdot)$ for a suitable $c>0$, while
replacing the chosen scales $\mu_i$ in the bubble tree convergence 
 \eqref{eq:bt-conv-2}
by $c\mu_i$ allows us to assume that this holds true for $q_1^*$ while still preserving the relations $U_0^*=U_1^*$ respectively $U_0^*(z)=U_1^*(\bar z)$.

To obtain a manifold $\ZZ$ which is so that $d^2E$ is definite orthogonal to $\ZZ$ we first
need to 
 determine suitable sets $\HH_1^\si(U_i^*)$ of maps which are so that the tangent space to 
\beq\label{def:H}
\HH^\si(\om_i):=\{U\circ q: U\in \HH_1^\si(U_i^*), \quad q\in \RR^\si(q_i^*)\}\eeq
at $\om_i=U_i^*\circ q_i^*$  coincides with the space $\ker(L_{\om_i})$ of Jacobi-fields along $\om_i$. 
Here and in the following $\si>0$ denotes a small constant that is chosen later and that is in particular small enough so that all maps in 
\beq
\label{def:RR-si}
\RR^\si(q_i^*):\{ q\in \RR: \norm{\pi\circ q\circ \pi^{-1}-\pi \circ  q_i^*\circ \pi^{-1} }_{C^k(S^2)}<\si\}\eeq
have the same degree  as $q_i^*$ and are so that $\abs{a_{n_i^*}(q_i)}\in (\half ,2)$.   Here and in the following we can choose $k\geq 2$ to be any fixed number.

If $\om_i$ is non-degenerate, i.e.~if all Jacobi fields are generated by variations of rational maps, then we can simply choose  
$$\HH_1^\si(U_i^*):=\{ U_i^*(\cdot +c): c\in \C, \, \abs{c}<\si \}.
$$ 
More generally if all Jacobifields along $U_i^*$ are integrable we first  choose a manifold $\HH_0(U_i^*)$ of harmonic maps of dimension $\dim(\ker L_{U_i^*})-6$ which is so that each harmonic map close to $U_i^*$ can be written uniquely in the form $U\circ M$ for a M\"obiustransform $M$ and an element $U$ of $\HH_0(U_i^*)$. We then set 
\beq
\label{def:HH1} 
\HH_1^\si(U_i^*):=\{ U(\cdot +c): c\in \C,\, \abs{c}< \si, \, U\in \HH_0(U_i^*) \text{ with } \norm{(U-U_i^*)\circ \pi^{-1}}_{C^k(S^2)}<\si \} \eeq
and note that 
 \eqref{ass:eq-Jabobi} ensures that the resulting $\HH^\si(\om_i)$ is indeed so that $T_{\om_i}\HH^\si(\om_i)=\ker(L_{\om_i})$.

In the more involved case where $U_i^*$ has non-integrable Jacobi-fields we need to additionally include certain non-harmonic maps in  $\HH^\si(\om_i)$, 
 and hence in $\HH_1^\si(U_i^*)$ and $\HH_0(U_i^*)$, to ensure that also these non-integrable Jacobi-fields correspond to directions that are tangent to 
 $\HH^\si(\om_i)$. 
 In this case, we first choose $\HH_0(U_i^*)$ as
described in detail in Appendix A of \cite{R} as a manifold of smooth maps with
\beq
\label{est:key-HH-0}
T_{U_i^*}\{U\circ M: U\in \HH_0(U_i^*), M \text{ M\"obius transform}\}=\ker(L_{U_i^*})\eeq
which has the key properties that 
for every $U\in \HH_0(U_i^*)$
we can bound \beq
\label{est:key-HH-Ck}
\norm{\tau_{g_{S^2}}(U\circ \pi^{-1})}_{C^k(S^2)}\leq C \norm{\tau_{g_{S^2}}(U\circ \pi^{-1})}_{L^2(S^2,g_{S^2})}
\eeq and find a variation $U_\eps$ in $\HH_0(U_i^*)$ of $U$  with 
 $\norm{\peps (U_\eps\circ \pi^{-1})}_{C^k}\leq C$ for which
\beq
\label{est:key-HH}
\ddeps E(U_{\eps})\geq \norm{\tau_{g_{S^2}}(U\circ \pi^{-1})}_{L^2(S^2,g_{S^2})},
\eeq
see Lemma 2.1 of \cite{R}.

We then 
 again define $\HH_1^\si(U_i^*)$ by \eqref{def:HH1} and note that 
 this manifold inherit the properties \eqref{est:key-HH-Ck} and \eqref{est:key-HH}, while \eqref{ass:eq-Jabobi} ensures that $T_{\om_i}\HH^\si(\om_i)=\ker(L_{\om_i})$ for $\HH^\si(\om_i)$ defined by \eqref{def:H}.

Having chosen $\HH_1^\si(U_i^*)$ in this way we now obtain our singularity models by carefully gluing a highly concentrated copy of an element of $\HH_1^{\si}(U_1^*)$ onto an element of $\HH_1^{\si}(U_0^*)$. 
That is we want to construct the elements $\zz$ of $\ZZ_0$ in a way that  $\zz\approx U_0(q_0(z))$ 
near $z=0$ while  $\zz(z)\approx U_1\circ q_1(\frac{1}{\mu z})$ away from zero  for  maps $U_{i}\in \HH_{1}^\si(U_{i}^*)$, $ q_{i}\in \RR^{\si}(q^*_{i})$, $i=0,1$ and a large scaling factor $\mu$.  

 \begin{rmk}
 \label{rmk:mu-large-si-small}
We will always consider elements that are obtained 
from maps in $\HH_1^\si(\om_i)$ and scaling factors $\mu> \bar \mu$ for  a sufficiently large number  $\bar \mu$
and a sufficiently small number $\si>0$ (depending only on $U_i^*$, $q_i^*$)
and in the following all claims and estimates
 are to be understood to hold true after increasing $\bar \mu$ and decreasing $\si$ if necessary.
 \end{rmk}

To describe this construction we 
 first assign to each $q_0\in \RR^{\si}(q_0^*)$ and each 
  $q_1(\frac{1}{\mu z})$ that we obtain from an element $q_1\in \RR^\si(q_1^*)$ and a scaling factor $\mu> \bar \mu$ the numbers 
\beq\label{def:mui}
\mu_0:= \abs{a_{n_0^*}(q_0)}^{-1/n_0^*}\text{ and }
\mu_1:= \mu \abs{a_{n_1^*}(q_1)}^{-1/n_1^*}.
\eeq
We note that while there are multiple ways of representing the same rational map $q_1(\frac{1}{\mu z})$ using an element of $\RR^\si(q_1^*)$ and a factor $\mu> \bar \mu$, the number $\mu_1$ is uniquely determined by $q_1(\frac{1}{\mu z})$ and while $\mu_0$ is of order $1$, $\mu_1$ will be of order $\mu$. 

Some of the gluing construction below will be carried out  on annuli 
\beq
\label{def:annuli}
A=\DD_{r_0}\setminus \DD_{r_1} , \quad  A^*:=(\DD_{2r_1}\setminus \DD_{r_1})\cup (\DD_{r_0}\setminus \DD_{\half r_0}) \text{ and } 
\hat A= \DD_{2\hat r}\setminus \DD_{\half \hat r}
\eeq
 whose radii
$\mu^{-1}\ll r_1\ll \hat r\ll r_0\ll 1$ are determined by 
\beq
\label{def:ri}
r_0:= \mu_0f_{\mu_{0}/\mu_1}^{-1} , \quad 
r_1:= \mu_1^{-1} f_{\mu_{0}/\mu_1} \text{ and } \hat r:= \mu_0^\half \mu_1^\mhalf \quad. 
\eeq
Here $\mu\mapsto f_\mu$ is a fixed non-decreasing  function which is so that
\beq
\label{ass:fmu}
\mu \abs{\partial_\mu f_\mu}\leq\fmu \text{ while } 
\tfrac{\log f_{\mu}}{\log\mu}\leq \si_1 \text{ and } f_{\mu}^{-1}\leq \si_1 
\eeq
for $\mu>\bar \mu$ and a small constant  $\si_1\in(0,\frac14]$. We use the same convention for $\si_1$ as for $\si$, i.e. ask in the following that all claims  hold provided $\fmu$ satisfies this assumption for a sufficiently small number $\si_1>0$ that only depends on $U_{0,1}^*$ and $q_{0,1}^*$.

The precise choice of $\fmu$ is not important in the construction and it will be useful later on that we can fix $\fmu$ according to our needs as for the proof of Theorems  \ref{thm:3} it will be convenient to work with $\fmu=\log\mu$ while in the proof of Corollary \ref{cor:thm2} we will want to choose $\fmu=\mu^{\si_1}$.

\begin{rmk}\label{rmk:symm}
We note that the choice of the above radii and annuli is so that the construction is invariant under a rescaling $z\mapsto \la z$ of the stereographic coordinates and 
symmetric with respect to interchanging the roles of the base and the bubble as we change the coordinates according to $z\mapsto \frac{\mu_0}{\mu_1 z}$.
\end{rmk}

To construct our 
singularity models we will first modify the maps 
\beq\label{def:ui}
\bb:=\bbu\circ \rmmu \text{ and } \bm:=\bmu\circ \rmmu \text{ for }q_\mu(z):= q_0(z)+q_1(\tfrac{1}{\mu z})
\eeq
on $A$ so that they agree upto first order errors in  $\abs{z}+\abs{\mu z}^{-1}$, then further modify the resulting maps on all of $\R^2$ so that they agree 
upto second order before gluing these maps together on the central annulus $\hat A$ and projecting onto $N$.

To this end we let
 $\hat \gamma=\hat \gamma_{\bbu(0),\bmu(0)}:[0,1]\to N$ be the shortest geodesic from $\hat \gamma(0)=U_1(0)$ to 
$\hat \gamma(1)=U_0(0)$. We note that $\hat \gamma$ is well defined as 
$\delta(\bmu,\bbu):= \dist_N( \bmu(0),\bbu(0))$
is small 
for all elements of  $\HH^\si(U_i^*)$ as such maps are 
 $C^k$ close to $U_i^*$ and 
$U_0^*(0)=U_1^*(0)$.  

To transition between 
$U_1(0)$ and $U_0(0)$ on $A$ we use
 \beq 
\label{def:gamma}
\gamma:=\hat \gamma\circ \phi_{r_0,r_1} 
\eeq
where $\phi_{r_0,r_1}=\phi_{r_0,r_1}(\abs{x}):\R^2\to [0,1]$ 
is defined by 
\beq
\label{def:phi-mu-la}
\phimul(r):=\begin{cases} 
(1-\phi(\tfrac{r}{r_1}))\cmul\log(\tfrac{r}{r_1}) &\text{ for } r\leq  2r_1 \\
\cmul\log(\tfrac{r}{r_1} )=1-\cmul\log(\tfrac{r_0}{r}) &\text{ for } r\in (2 r_1,\thalf r_0)
\\
1-\phi(\tfrac{2r}{r_0})\cmul \log(\tfrac{r_0}{r})  &\text{ for } r\geq  \half r_0
\end{cases} 
\eeq 
for $\cmul:=(\log(\frac{r_0}{r_1}))^{-1}$
and a fixed 
cut-off function $\phi\in C_c^\infty([0,2),[0,1])$  with $\phi\equiv 1$ on $[0,1]$.

We note that $\phi_{r_0,r_1}$ is chosen so that  $\Delta \phi_{r_0,r_1}\equiv 0$ outside of  $A^*$  
and so that  $\phi_{r_0,r_1}\vert_{\DD_{r_1}}\equiv 0$ while $\phi_{r_0,r_1}\vert_{(\DD_{r_0})^c}\equiv 1$. Hence $\gamma:\R^2\to N$ is given 
by a harmonic map outside of $A^*$ and 
 transitions between 
 $\gamma\vert _{\DD_{r_1}}\equiv U_1(0)$ and $\gamma\vert_{\DD_{r_0}^c}\equiv U_0(0)$.
Setting
\beq
\label{def:gamma-tilde}
 \gabb:= \gamma-\bbu(0) \text{ and } \gabm:=\gamma-\bmu(0)
\eeq
we can hence change our original maps $u_{0,1}$ into maps $u_1(z)+\gabb(z)$ and $u_0(z)+\gabm(z)$ that agree upto error terms of order $O(\abs{z}+\frac{1}{\mu \abs{z}})$ on $A$. 
Next we set  
\beq
\label{def:omu}
\hhompi:=\bmu-\bbu, \quad \modbm:= -d\hhompi(0)(q_1(\tfrac{1}{\mu z} )) \text{ and }\modbb:= d\hhompi(0)(\rmbmc(z))
\eeq
to obtain maps 
\beqa\label{def:v1}
v_1(z):= \bb(z)+\gabb(z)+d\hhompi(0)(\rmbmc(z))=\bb(z)+\gabb(z)+\modbb(z)
\eeqa
and 
\beqa \label{def:v2} 
v_0(z):=\bm(z)+\gabm(z) -d\hhompi(0)(\rmbbmuc(z))=\bm(z)+\gabm(z)+\modbm(z),
\eeqa
which now agree upto a second order error term 
$v_0-v_1
=\hhompi(\rmmu)-\hhompi(0)-d\hhompi(0)(\rmmu).
$

We finally interpolate between these maps on the central annulus 
$\hat A$
 and project onto $N$. 
That is we define our singularity model as 
\beqa \label{def:zz}
\zz:=\pi_N((1-\varphi_{\hat r}) v_1+\varphi_{\hat r} v_0),
\eeqa
where we set $\varphi_{\hat r}(z):=\bar \varphi(\hat r^{-1}\abs{z})$ for a fixed smooth function $\bar \varphi:[0,\infty)\to [0,1]$ which is so that 
$\bar \varphi(t)= 0$ for $t\leq \half$ while $\bar \varphi(t)=1$ for $t\geq 2$.

We note that these maps are well defined for $U_{0,1}\in \HH^\si(U_{0,1}^*)$ and  $q_{0,1}\in \RR^\si(q_{0,1}^*)$ as the image of 
$(1-\varphi_{\hat r}) v_1+\varphi_{\hat r} v_0$ will be contained in a small tubular neighbourhood of $N$ where the nearest point projection is well defined.

\begin{rmk}
\label{rmk:gluing}
An important feature of our construction is that the modifications to the maps $u_i$ are not just supported on an annulus, but rather on all of $\R^2$. 
In contrast to maps we would get from a more ad hoc gluing construction this means that $dE(\zz)$ has a significant part that is supported on $\DD_{r_1}$ and $\DD_{r_0}^c$. This will be crucial in the proof of our main results later on as  variations of elements of $\ZZ_0$ induced by variations of one of the rational maps are mainly concentrated on  $\DD_{r_1}$ respectively $\DD_{r_0}^c$ and hence can only be used to obtain directions $y_\zz\in T_\zz\ZZ$ satisfying \eqref{est:key-property} if $dE(\zz)$ is so that it interacts suitably with such directions. 
\end{rmk}

  \section{Key lemmas}\label{sect:key-lemmas}
Here we collect the key properties on the behaviour of the energy and its variations on $\ZZ$ that we need to prove our main results. We postpone the proofs of these lemmas to Sections \ref{section:proof-of-lemmas} and \ref{sect:variations}, as an understanding of these technical proofs is not required for the proofs of our main results.

As these estimates will all be invariant under precomposition with a rotation, it suffices to state and prove them for elements of the set $\ZZ_0$ of singularity models for which the bubble is attached at the north pole. We can furthermore assume that $q_1(\frac{1}{\mu z})$ is represented by a map $q_1$ for which $\abs{a_{n_1^*}(q_1)}=1$ and hence use that $\mu_1$ and $\mu_0$ defined in \eqref{def:mui} are so that  $\mu_1=\mu$ while $\mu_0\sim 1$.

Given $\zz=\zz(U_0\circ q_0, U_1\circ q_{1}(\frac{1}{\mu\cdot}))\in \ZZ_0$  we 
will work with respect to the
weighted $H^1$-norm $\norm{\cdot}_\zz$ and the corresponding inner product $\langle\cdot,\cdot\rangle_\zz$
which are given in stereographic coordinates by 
 \beq
 \label{def:inner-product}
 \norm{w}_{\zz}^2:=\int_{\R^2} \abs{\na w}^2 +\rho_\zz^2 \abs{w}^2 dx
 \text{ for } \rho_{\zz}^2:=\abs{\na \pi_{\mu_0(\zz)^{-1}}}^2+\abs{\na \pi_{\mu_1(\zz)}}^2,
 \eeq
  $\pi_\la(z):=\pi(\la z)$.
We note that this definition respects the symmetry of the construction, compare Remark \ref{rmk:symm}, and that we define $\norm{\cdot}_\zz$  for general elements $\zz=\zz_0\circ R_p$ of $\ZZ$ by the analogue formula in stereographic coordinates centred at $p$. 
 
 We recall that the second variation of the energy is given by 
\beq
\label{eq:second-var}
d^2E(\zz)(v,w)=\int \na v\na w-A(u)(\na u,\na u)A(u)(v,w) \text{ for } v,w\in\Gamma(\zz^* TN),
\eeq
compare e.g. \cite[Lemma 3.1]{R}, and note that $d^2E$ is uniformly bounded with respect to this norm, i.e.~so that 
\beq
\label{est:sec-var-general}
\abs{d^2E(\zz)(w,v)}\leqs \norm{w}_\zz \norm{v}_\zz \text{ for all } v,w\in  \Gamma(\zz^* TN),
\eeq
 compare Section \ref{sect:variations}. 
 Here and in the following
 we write $T_1\leqs T_2$ if we can bound  $T_1\leq C T_2$  for a constant $C$ that only depends on the limiting configuration, i.e.~on $U_{0,1}^*$ and $q_{0,1}^*$ while we write $T_1\sim T_2$ if both $T_1\leqs T_2$ and $T_2\leqs T_1$.

The first key property we need is that the choice of $\HH_1^\si(U^*_{0,1})$ in the construction of $\ZZ_0$ ensures that the second variation is uniformly definite orthogonal to $\ZZ$, i.e.~that

\begin{lemma}
\label{lemma:positive-def-2ndvar}
There exist $c_0>0$ depending only on $U_{0,1}^*$ and $q_{0,1}^*$  so that
for every $\zz\in \ZZ_0$ we can split $T_\zz^\perp\ZZ$ orthogonally into subspaces $\VV_\zz^\pm$ which are so that 
\beq
\label{est:2ndvar-pos-def}
\pm d^2E(\zz)(w^\pm,w^\pm)\geq c_0\norm{w^\pm}_\zz^2 \text{ and } d^2E(\zz)(w^+,w^-)=0 \text{ for all }w^\pm \in \VV_\zz^\pm.
\eeq 
\end{lemma}
 
 To state all other properties of the energy $E$ on $\ZZ$ we now 
associate to each 
 $\zz=\zz(U_i,q_i, \mu)$ the quantities 
\beqa 
\label{def:quantities-z}
\de:=\dist_{N}(\bbu(0),\bmu(0)),\quad 
\TT:= \max(\TT_0,\TT_1) \text{ and } \nutot := \max(\nu_1,\nu_0)
\eeqa
where we set
$ \TT_i :=\norm{\tauS(\bbu\circ \pi^{-1})}_{L^2(S^2,g_{S^2})}$ and 
write $q_i(z)=\sum_ja_j(q_i)z^j$ near $z=0$ to define  
\beq
\label{def:nu}
\nu_1:= \max_{j=1,\ldots,\dqs} \abs{a_j(q_1)}\mu^{-j}, \quad \nu_0:= \max_{j=1,\ldots,\dps} \abs{a_j(q_0)}\mu_1^{-j},
\eeq
 $n_{0,1}^*\geq 1$ the order of the zero of the fixed rational maps $q_{0,1}^*$ at $z=0$. 
 We remark that  these quantities respect the symmetry  discussed in Remark \ref{rmk:symm} and recall furthermore that $\TT_i$ also controls the 
 $C^k$-norm of the tension, compare \eqref{est:key-HH-Ck}.

The first quantitative estimate we need is a bound on the dual norm $\norm{dE(\zz)}_*:=\sup\{dE(\zz)(w): \norm{w}_\zz=1\}$ of the first variation and in Section \ref{sect:variations} we shall prove 

\begin{lemma}\label{lemma:first-variation}
For any $\zz\in \ZZ_0$  we have
$$\norm{dE(\zz)}_* \leqs \Ctilt \nutot (\log\mu)^{\half}
+\tfrac{\de(\log\fmu)^\half}{\log\mu}+\TT.
$$
\end{lemma}

Here and in the following we compute the norms of $\beta=U_0-U_1$ and $\peps\beta$ as norms of the corresponding maps from $S^2$ i.e.~set 
$$\norm{\beta}:=\norm{\beta\circ \pi^{-1}}_{C^2(S^2)} \text{ and } \norm{\peps \beta}:=\norm{\peps \beta\circ \pi^{-1}}_{C^2(S^2)}.$$

In the following we consider variations 
 $\zz_\eps=\zz(U_{0,1}^{(\eps)},q_{0,1}^{(\eps)},\mu)$ in $\ZZ$ which are induced by variations  of
$U_i^{(\eps)}$ in  $\HH_1^\si(U_i^*)$ 
and of $q_i^{(\eps)}$  in $\RR^\si(q_i^*)$. We do not impose the 
assumption that $\abs{a_{n_i}^*(q_i^{(\eps)})}=1$ for $\eps\neq 0$ but allow for general variations of $q_i$ in $\RR^\si(q_i)$ as this allows us to  represent variations of $\mu$ such as $\mu_\eps=\mu(1+\eps)$ as variations of $q_1$ of the form $q_1^{(\eps)}(z)=q_1(\frac{1}{1+\eps} z)$. In the following we hence do not need to consider variations of $\mu$ separately.

To simplify the notation we will drop the index $\eps$ whenever there is no room for confusion and use the convention that $\peps $ is evaluated at $\eps=0$ unless specified otherwise. We  
will only ever consider variations for which 
\beq
\label{ass:variations}
\norm{\peps (\pi\circ q_i\circ\pi^{-1})}_{C^{k}(S^2)}\leqs 1
\text{ and }
\norm{\peps U_i}_{C^{k}(S^2)}
\leqs 1, \quad i=0,1\eeq
as these corresponds to variations in $\ZZ_0$ for which
$\norm{\peps \zz}_\zz\leqs 1$. We note that for such variations also $\norm{\peps \zz}_{L^\infty}\leqs 1$ while $\abs{\peps \mu_i}\leqs \mu_i$.

We associate to each such variation the quantities 
\beqa
\label{def:quantities-var}
\etazero:=\abs{\peps \bbu(0)}+\abs{\peps \bmu(0)} \text{ and }
\etatau:=\norm{\peps \tauS(\bbu\circ \pi^{-1})}_{C^0(S^2)}+\norm{\peps\tauS(\bmu\circ \pi^{-1})}_{C^0(S^2)}
\eeqa
and 
furthermore set $\eta_{rat}=0$ whenever $\peps q_0=\peps q_1=0$ while otherwise we let $\eta_{rat}:=1$ and choose $j^*\geq 1$ so that both $\peps q_0$ and $\peps q_1$ are of order $O(z^{j^*})$ near $z=0$. 

The dual norm $\norm{d^2E(\zz)(\peps \zz, \cdot)}_*:=\sup\{\peps dE(\zz_\eps)(w): w\in \Gamma(\zz^*TN), \norm{w}_\zz=1\}$ of the second variation of $E$ along such variations is controlled by the following lemma which we will prove in Section \ref{sect:variations}.

\begin{lemma}
\label{lemma:second-variation}
For any variation $\zz_\eps$ for which \eqref{ass:variations} holds we can bound  
\beqa \label{est:second-var-case1}
\norm{d^2E(\zz)(\peps \zz, \cdot)}_*
\leqs \big[( \Ctilt+\norm{\peps \beta})\nutot+ \eta_{rat} \mu^{-j^*}\big]
\log(\mu)^{\half} +\tfrac{(\de+\eta_0)(\log\fmu)^\half}{\log\mu}
+\TT+\eta_\TT.
\eeqa

\end{lemma}

We note that for changes of $U_i$ that are induced by translations on the domain we can bound $\eta_\tau\leqs \TT$ since 
$\TT_i$ controls $C^k$ norm of the tension  of $U_i$, compare \eqref{est:key-HH-Ck}. For more general 
variations in $\HH_1^\si(U_i^*)$  we can instead use 
the continuity of the Jacobi operator and the fact that  
elements of 
$T_\uius\HHz(\uius)$ are Jacobifields along $\uius$ to see that 
\beq
\label{est:eta-tau}
\eta_\TT\leqs \TT+\norm{P_{\bbu}(\peps\tauS(U_{1,\eps}\circ \pi^{-1}))}_{C^0(S^2)}+
\norm{P_{\bbu}(\peps\tauS(U_{0,\eps}\circ \pi^{-1}))}_{C^0(S^2)}
\eeq
will be smaller than any given positive constant if $\si>0$ is chosen sufficiently small.

In addition we need to understand the behaviour of $E$ on $\ZZ$ and we will show  
\begin{lemma}\label{lemma:general-energy-expansion}
Let $U_{0,1}^*$ be any harmonic maps from $S^2$ into a smooth Riemannian manifold $N$ of any dimension, let $q_{0,1}^*:\Chat\to \Chat $ be any rational maps with $q_{0,1}^*(0)=0$ and let $\ZZ_0$ be the corresponding set of singularity models defined in Section \ref{sect:Z}. Then we have 
\beqa 
\label{eq:en-exp-nice}
dE(\zz)(\peps \zz)&=\sum_{i=0,1} \int_{\Om_i}j_i\peps \Delta u_{i}
+\pi c_{r_0,r_1} \tfrac{d}{d\eps}\de(U_{0},U_{1})^2+\sum_{i=0,1} \deg(q_i) \tfrac{d}{d\eps}E(U_i)  +\err
\eeqa 
for any variation 
$\zz_\eps$ in $\ZZ_0$ which satisfies \eqref{ass:variations}. Here 
for $\Om_1:=\DD_{\hat r}=\Om_0^c$, $c_{r_0,r_1}= (\log(\frac{r_0}{r_1}))^{-1}$ and $\err$ denotes an error term that is bounded by $\abs{\err}\leqs R_0$ for 
\beqa
\label{est:en-exp-nice-error}
R_0:=&
 \nutot  \Ctilt\big[
 \mu^\mhalf
+\etazero +\eta_\tau \big] 
+\tfrac{\de}{\log\mu}\big[\de+ \fmu^{-2}+\eta_0 [(\log\mu)^{-1}+\fmu^{-1}]\big] \\
&+\TT\big[\mu^{-1}\log\mu+\etaz (\log\mu)^{-1}\fmu^{-1}\big] .
\eeqa
\end{lemma} 

As $c_{r_0,r_1} \sim (\log\mu)^{-1}$ and as we shall see that 
$\abs{\int_{\Om_i} j_i\peps \Delta u_i}\leqs \norm{\beta} (\nu_i^*+\mu^\mhalf\nutot)$ for  $\nu_{1^*}=\nu_0$ and $\nu_{0^*}=\nu_1$
this  ensures in particular that
\beq
\label{est:var-E-always}
\abs{\peps E(\zz_\eps)}\leqs \norm{\beta} \nutot+\tfrac{\de}{\log\mu}+\TT
\eeq
for any variation $\zz_\eps$ in $\ZZ$ for which $\norm{\peps \zz}\leqs 1$. 
For variations induced by only a change of $q_1$ we obtain an improved bound of 
\beq
\label{est:var-E-always-rational}
\abs{\peps E(\zz_\eps)}\leqs \norm{\beta}\nu_0+
\Ctilt\nu_1\mu^{\mhalf} +\tfrac{\de^2}{\log\mu}+\tfrac{\de}{f_\mu^{2}\log\mu} +\TT \mu^{-1}\log\mu.
\eeq
The same bound also applies to   
$
\mu\abs{\partial_\mu E(\zz(U_{0,1},q_{0,1},\mu)} 
$
since variations of $\mu$ can be represented by variations of $q_1$.
On the other hand, for variations of $U_i$ that are induced by translations on the domain we obtain that
\beq\label{est:var-E-always-Efixed}
\abs{\peps E(\zz_\eps)}\leqs \nutot+\tfrac{\de}{\log\mu}+\tfrac{\TT}{\log\mu}.
\eeq
To prove our main results we now want to identify variations 
for which $dE(\zz)(\peps \zz)$ scales like the dominating term in the above estimates. 

To state these lemmas in a way that makes them applicable to all our main results we recall that we always deal with settings for which the following assumption is satisfied. 
\begin{ass}
\label{ass:Ui}
We ask that  $U_{0,1}^*:\Chat\to N$ are harmonic maps with $U_0^*(0)=U_1^*(0)$ and 
 $d U_{0,1}^*(0)\neq 0$ for which one of the following holds
\begin{enumerate}[topsep=0pt]
\item[(i)]\label{ass:Ui-1} We have $U_0^*=U_1^*$ or $U_0^*(z)=U_1^*(\bar z)$  and these maps are non-degenerate critical points  \\
or
\item[(ii)]\label{ass:Ui-2} The target manifold $N$ is $3$ dimensional and the tangent spaces 
 $T_{U_i^*(0)}U_i^*(\Chat)$ do not coincide.
\end{enumerate}
\end{ass}

For elements $\zz\in\ZZ_0$ for which $\norm{\beta}\nu_0$, respectively $\norm{\beta}\nu_1$, is large compared to the error terms that involve $\de$ and $\TT$, we will want to consider variations that are induced by a suitable variations of the rational map $q_1$, respectively $q_0$.
This case only occurs if $U_1^*\neq U_0^*$ and  in 
Section \ref{sec:main-term} we shall prove

\begin{lemma}\label{lemma:energy-expansion-rational}
Let $U_{0,1}^*$ be harmonic maps with $U_0^*\neq U_1^*$ for which Assumption \ref{ass:Ui} holds, let $q_{0,1}^*\in \RR$ and let $\ZZ_0$ be the corresponding set of singularity models defined in Section \ref{sect:Z}.

Then there exists a constant $c_1>0$ so that  
for every $\zz\in \ZZ_0$ and 
each $i\in\{0,1\}$
there is  a variation $\zz_\eps$ that is induced by a variation of $q_i$ which satisfies \eqref{ass:variations} for which  
\beq 
\label{claim:energy-expansion-case1}
dE(\zz)(\peps \zz_\eps)\geq c_1 \sitilt^2\nu_{i^*}  -C\big[ \Ctilt\nutot\mu^{\mhalf} +\tfrac{\de^2}{\log\mu}+ 
\tfrac{\de}{f_\mu^{2}\log\mu} +\TT \mu^{-1}\log\mu\big],
\eeq
where we write for short 
$\nu_{0^*}=\nu_1$ and $\nu_{1^*}=\nu_0$.
\end{lemma}

Conversely, if $\frac{\de}{\log\mu}$ is large compared to $\norm{\beta}\nutot$ and $\TT$ we want to consider variations which reduce $\de$. To this end we can use 

\begin{lemma}
\label{rmk:de}
Let $U_{i}^*$ be as in Assumption \ref{ass:Ui} and let $U_{0,1}\in \HH_1^\si(U_{0,1}^*)$. Then there exists a variation of the form  $U_{i,\eps}=U_i(\cdot +\eps a)$, $a\in \R^2$ 
with $\abs{a}\leqs 1$ so that 
\beq
\tfrac{d}{d\eps}\de(U_{0},U_{1})^2=\de(U_{0},U_{1}),
\eeq
where for settings for which the first part of Assumption \ref{ass:Ui} is satisfied, we can freely choose which of $U_0$ or $U_1$ we want to vary, while  for settings satisfying the second part of Assumption \ref{ass:Ui} the above holds true for at least one of $i\in\{0,1\}$.
\end{lemma}

For such variations we then obtain

\begin{lemma}\label{lemma:energy-expansion-rotation}
Let $U_{0,1}^*$ be as in Assumption \ref{ass:Ui}, let $q_{0,1}^*$ be any maps in $\RR$ and let $\ZZ_0$ be the corresponding set of singularity models defined in Section \ref{sect:Z}. Then 
\beqa
dE(\zz)(\peps \zz)\geq \tfrac{\de}{\log\mu}- C   \big(\Ctilt
\nutot+\TT
(\log\mu)^{-1}f_\mu^{-1}\big)
\eeqa
for 
 the 
 variation $\zz_\eps$ in $\ZZ_0$ that is induced by changing $U_i$ as described in Lemma \ref{rmk:de}.
\end{lemma}

Finally, in non-integrable settings where we have to deal with maps $U_i$ which are not harmonic, we also want to consider variations which reduce $E(U_{i})$ and hence show

\begin{lemma}\label{lemma:energy-expansion-tension}
Let $U_{0,1}^*:S^2\to N$ be harmonic maps
 into an analytic Riemannian manifold of any dimension  with $U_{0}^*(0)=U_1^*(0)$. Then for any $\zz\in \ZZ_0$ there exists a variation $\zz_\eps$ in $\ZZ_0$ induced by a variation of $U_i$ in $\HH_1^\si(U_i^*)$ which satisfies \eqref{ass:variations} so that
\beqa
dE(\zz)(\peps\zz)\geq \TT_i-C \big[\Ctilt\nutot +
 \tfrac{\de}{\log\mu}+f_\mu^{-1}(\log\mu)^{-1}
\TT\big].
\eeqa
\end{lemma}

\section{Proof of the main theorems}\label{sect:main-proofs}
Here we  give the proof of our main results based on  the lemmas stated in the previous section.
So let $(u_j)$ be a sequence of almost harmonic maps from $S^2$ to $N$ that converges to a bubble tree with bubble $\om_1(z)=U_1^*( q_1^*(\frac{1}{z}))$ and base $\om_0=U_0^*\circ q_0^*$ and let $\ZZ=\{\zz:S^2\to N\}$ be the corresponding set of singularity models constructed in Section \ref{sect:Z}.

For sufficiently large $j$ our construction 
and
 \eqref{eq:bt-conv-2} ensure that the function $\zz\mapsto \norm{u_j-\zz}_\zz$ achieves its minimum on $\ZZ$ and that 
the corresponding $\zz_j\in \ZZ$ with $\norm{u_j-\zz_j}_{\zz_j}=\dist(u_j,\ZZ):=\inf_{z\in \ZZ} \norm{u_j-\zz}_\zz$ are obtained from maps $U_{0,1}^{(j)}\to U_{0,1}^*$, rational functions $q_{0,1}^{(j)}\to q_{0,1}^*$ and 
scales $\mu(\zz_j)$ that are comparable to the numbers $\mu_j$ from \eqref{eq:bt-conv-2}. 
Furthermore, \eqref{eq:bt-conv-2} can be used to see that not only $\norm{u_j-\zz_j}_{\zz_j}\to 0$ but also $\norm{u_j-\zz_j}_{L^\infty(S^2)}\to 0$.

It hence suffices to prove 
that there exists a constant $\eps>0$ so that the claims of our main results hold true for all maps $u$ for which 
 \beq
\label{est:norm-less-eps}
\norm{u-\zz}_{\zz}=\dist(u,\ZZ)<\eps \text{ and } \norm{u-\zz}_{L^\infty(S^2)}<\eps \text{ for some }  \zz\in \ZZ,
\eeq
where after a rotation of the domain we can assume without loss of generality that $u$ is so that the above holds for a $\zz\in \ZZ_0$. 

Here and in the following we continue to use the convention that all claims are to be understood to hold true provided the numbers $\si, \si_1>0$ and $\bar \mu$ in the construction of $\ZZ$ are chosen sufficiently small respectively large. 

\begin{rmk}
\label{rmk:def-dual-norm}
Given any $u$ for which 
\eqref{est:norm-less-eps} holds we use the convention that 
\beq
\label{def:dE(u)-norm}
\norm{dE(u)}_{*}:=\sup\{dE(u)(w): \norm{w}_\zz=1\}
\eeq
is to be computed with respect to the weighted norm $\norm{\cdot}_\zz$ defined in \eqref{def:inner-product} for a $\zz\in \ZZ$ for which \eqref{est:norm-less-eps} holds.
\\
We do not claim that $\zz$ is uniquely determined by this relation but observe that $\norm{\cdot}_{\tilde \zz}\sim \norm{\cdot}_\zz$ 
if \eqref{est:norm-less-eps} is satisfied for both $\zz$ and $\tilde \zz$ and hence that it does not matter which such element of $\ZZ$ we use to define $\norm{dE(u)}_*$ in our main results.
\end{rmk}

We will combine the results from the previous section  with the following lemma that describes the basic estimates that we can obtain from our method of proof, compare also \cite[Section 2]{MRS}.

\begin{lemma} \label{lemma:first step}
Let $U_{0,1}^*:S^2\to N$ be any given harmonic maps into a smooth Riemannian manifold, let $q_{0,1}^*\in \RR$ and let $\ZZ$ be the corresponding set of singularity models.  Then there exist constants $\eps>0$, $\bar \La<\infty$ and $C$ so that the following holds true for any 
$u\in H^1(S^2, N)$ with $\norm{\tauS(u)}_{L^2(S^2,g_{S^2})}\leq 1$ 
for which there exists $\zz\in\ZZ$ so that
\eqref{est:norm-less-eps} is satisfied. 

We can bound $w:= u-\zz$ by 
\beq
\label{est:w}
\norm{w}_{\zz}\leq C \norm{dE(u)-dE(\zz)}_* 
\eeq
and furthermore get that 
\beq \label{est:first-step}
dE(\zz)(y_\zz)\leq C \abs{dE(u)(y_\zz)}+C\norm{dE(u)}_*^2
\eeq
for any 
$u, \zz$ for which there is a unit element 
$y_\zz\in T_\zz\ZZ$ with
\beq
\label{ass:quotient}
Q(\zz,y_\zz):= \frac{\norm{dE(\zz)}_*^2+\norm{d^2E(\zz)(y_\zz,\cdot )}_* \norm{dE(\zz)}_*}{ dE(\zz)(y_\zz)}
\leq  \bar \La^{-1} . \eeq


\end{lemma}

For $C^2$-energies that are defined on a fixed Hilbert-space, the analogue of the above lemma can be obtained from a simple argument that is based on the fundamental theorem of calculus, compare \cite{MRS}. In more involved settings such as the present problem where we consider an energy on a suitable Hilbert-manifold of maps and where we work with distances that are calculated using a weighted norm that depends on the corresponding element of $\ZZ$ one needs to proceed with more care.
For almost harmonic maps from higher genus surfaces which are close to a simple bubble tree a similar result was obtained by the author in Section 4 of \cite{R}, and most of the arguments from \cite{R} are applicable also in the present situation. 
 For the convenience of the reader we hence include a sketch of the proof of this lemma, but omit the technical details except to explain why the estimate 
 \eqref{est:w} is simpler than the one obtained in \cite[Lemma 4.4]{R}.  
 
\begin{proof}[Sketch of the proof of Lemma \ref{lemma:first step}]
Given $u$, $\zz$ and $w=u-\zz$ as in the lemma we set  
 $w_\zz:= P_\zz w$, $P_p$ the projection onto $T_pN$. We first note that \eqref{est:norm-less-eps} ensures that  $\norm{w-w_\zz}_{\zz}\leq C\norm{w}_{L^\infty}\norm{w}_\zz$ where and in the following all $L^p$ norms are computed over $(S^2,g_{S^2})$ unless indicated otherwise.

We can easily check
that the variation of the weight $\rho_\zz$ in \eqref{def:inner-product} is controlled by 
$\norm{\peps \rho_\zz}_{L^2}\leqs 1$ whenever $\norm{\peps \zz}_{\zz}\leq 1$. As $\zz$ minimises the distance to $u$ and thus
$$0=\ddeps \norm{u-\zz}_{\zz}^2=-2\langle u-\zz,\peps \zz\rangle_\zz+2\int \rho_\zz\peps \rho_\zz \abs{w}^2$$ we hence get that
$\norm{P^{T_\zz\ZZ} w_\zz}_\zz\leq C\norm{w}_{L^\infty}\norm{w}_\zz$, i.e. that
 $w_\zz$ is nearly orthogonal to $T_\zz\ZZ$.

We then set $\tilde w_\zz:=
P^{\VV_\zz^+}w_\zz-P^{\VV_\zz^-}w_\zz ,$ where 
 $P^{\VV_\zz^\pm}:\Gamma(\zz^*TN)\to \VV_\zz^{\pm}$ denote the $\langle\cdot,\cdot\rangle_\zz$ orthogonal projections onto the subspaces $\VV_\zz^\pm$ of $T_\zz^\perp \ZZ$ from Lemma \ref{lemma:second-variation}. The uniform definiteness of $d^2E$ on these spaces ensures that
\beq
\label{est:first-step-proof-lambda}
(c_0-C\norm{w}_{L^\infty})\norm{w}_{\zz}^2\leq d^2E(\zz)(w_\zz,\tilde w_\zz) \leq (dE( u)-dE(\zz))(\tilde w_\zz) +\err_1,\eeq
where we can analyse the resulting error term as explained in detail 
in the proof of Lemma 4.4 of \cite{R}, resulting in an estimate of
\beq\label{est:err-first-step}
\abs{\err_1}\leqs \norm{w}_{L^\infty}\norm{w}_{\zz}^2+ \norm{w}_{L^\infty} \norm{\tauS(u)}_{L^2} \norm{\tilde w_\zz}_{L^4}^2.\eeq
 The only difference to the proof in \cite{R} is that in the present situation
 $$\norm{\tilde w_\zz}_{L^4}\leqs \norm{\tilde w_\zz}_{H^1}\leqs \norm{\tilde w_\zz}_\zz\leq \norm{w_\zz}_\zz\leqs \norm{w}_\zz$$ 
and that we can
hence bound the second term on the right hand side of \eqref{est:err-first-step} by a multiple of the first term. This was not possible in \cite{R}, as the correct weighted norm in that context does not control the $H^1$ norm in a uniform way. 

For $\eps>0$ sufficiently small we hence conclude that 
\beq 
\label{est:norm-w-squared}
\norm{w}_\zz^2\leq C (dE( u)-dE(\zz))(\tilde w_\zz)
\eeq
which yields the first claim of the lemma.

If there exists a unit element $y_\zz$ for which  \eqref{ass:quotient} is satisfied then we can use that
$$dE(\zz)(y_\zz)=dE(u)(y_\zz)+d^2E(\zz)(y_\zz,w_\zz)+\err_2$$
for an error of order $\abs{\err_2}\leqs \norm{w}_\zz^2$. We hence conclude that 
\beqa
\label{est:second-step}
dE(\zz)(y_\zz)&\leq \abs{dE(u)(y_\zz)}+ \abs{d^2E(\zz)(y_\zz,w_\zz)}+ C\norm{dE(u)}_*^2+C\norm{dE(\zz)}_*^2\\
&\leq C  \abs{dE(u)(y_\zz)}+C \norm{dE(u)}_*^2+C\bar \La^{-1} dE(\zz)(y_\zz)
\eeqa
for constants $C$ that only depend on $U_i^*$ and $q_i^*$. 
For $\bar \La=\bar\La(U_i^*,q_i^*) $ sufficiently large this yields the second claim of the lemma.   
\end{proof}

To prove our main results we will now apply the above lemma in directions $y_\zz$ that correspond to variations $\zz_\eps$ as considered in one of the three Lemmas \ref{lemma:energy-expansion-rational}, \ref{lemma:energy-expansion-rotation}, \ref{lemma:energy-expansion-tension}.

\subsection{Proofs of Theorem \ref{thm:1} and Corollary \ref{cor:thm1}}\label{sec:setting1}
$ $\\
We first consider settings in which the base and bubble are given by parametrisations of the same non-degenerate minimal sphere with the same orientation.

So let $U^*:\R^2\to N$, $q_{0,1}^*\in\RR$ and $\om_{0,1}$ be as in Theorem \ref{thm:1}, let $\ZZ$ be the corresponding set of singularity models and let 
$u$ be a map for which \eqref{est:norm-less-eps} holds true for $\eps>0$ chosen as in Lemma \ref{lemma:first step} and $\zz\in \ZZ_0$. 

We then consider the variation $\zz_\eps$ of $\zz$ that we obtain by fixing $U_0$, $q_{0,1}$ and $\mu$ and varying $U_1$ according to $U_{1,\eps}=U_1(\cdot+\eps a)$ 
for $\abs{a}\sim 1$ chosen as in Lemma \ref{rmk:de}.  
As   $\norm{\peps \zz_\eps}_\zz\sim 1$ we obtain from
Lemma \ref{lemma:energy-expansion-rotation} that 
 \beq
\label{est:case1-0}
dE(\zz)(y_\zz) \geq 2c\tfrac{\de}{\log\mu}- C   \Ctilt \nutot
\eeq
for the corresponding unit element 
$y_\zz=\frac{\peps \zz_\eps}{\norm{\peps \zz_\eps}_\zz}$ and 
for constants $C,c>0$ that only depend on $U^*$ and $q_{0,1}^*$. 
Here we use that $\TT=0$ as  the maps $U_i$ are given by  $U_i(\cdot)=U^*(\cdot+c_i)$ for some $\abs{c_i}<\si$ since
$U^*$ is assumed to be non-degenerate.  
 As $dU^*(0)\neq 0$ this furthermore implies that $\de\sim \norm{\beta}$ 
as both of these quantities scale like $\abs{c_0-c_1}$.
As $\nutot\leqs \mu^{-1}\ll (\log\mu)^{-1}$ we hence deduce from \eqref{est:case1-0} that 
 \beqs
dE(\zz)(y_\zz) \geq c\tfrac{\de}{\log\mu}
\eeqs
while Lemmas \ref{lemma:first-variation} and  \ref{lemma:second-variation} ensure that 
\beq 
\label{est:proof-1-1}
\norm{dE(\zz)}_* \leq C\tfrac{\de(\log\fmu)^\half}{\log\mu} \text{ and } 
\norm{d^2E(\zz)(y_\zz,\cdot)}_*\leq C \tfrac{(\log\fmu)^\half}{\log\mu},
\eeq
again for constants that only depend on the limiting configuration. 

We hence deduce that 
$$Q(\zz,y_\zz)\leq C \tfrac{\log\fmu}{\log\mu}\leq C \si_1\leq \bar \La^{-1},$$
$\bar \La$ as in Lemma \ref{lemma:first step}
since $\fmu$ satisfies \eqref{ass:fmu} for a constant $\si_1=\si(U_i^*,q_i^*)>0$ that can still be reduced if necessary. 
This lemma hence implies that
\beq\label{est:proof-1-3-1}
\norm{w}_\zz\leqs \norm{dE(u)}_{*}+\norm{dE(\zz)}_{*}
\eeq
and
\beq \label{est:proof-1-3-2}
\tfrac{\de}{\log\mu}\leq dE(u)(y_\zz)+C\norm{dE(u)}_*^2.\eeq
If $u$ is harmonic then \eqref{est:proof-1-3-2} 
ensures that $\de(\zz)=0$ and thus that $\zz(z)=U^*( q_0(z)+q_1(\frac{1}{\mu z}) +c_0)$. Thus $dE(\zz)=0$ so   \eqref{est:proof-1-3-1} ensures that $u=\zz$ and we obtain the claim about \textit{harmonic} maps made in the theorem. 

To prove the claims about almost harmonic maps made in 
the theorem and the subsequent Corollary \ref{cor:thm1}
 we now show that  
\beq \label{est:proof-1-2}
\abs{E(u)-E^*}\leqs \tfrac{\de^2}{\log\mu} +\norm{dE(u)}_*^2,
\eeq
where we write for short $E^*=E(\om_0)+E(\om_1)=[\deg(q_0^*)+\deg(q_1^*)]E(U^*)$. 

To see this
we first use \eqref{est:proof-1-1} and \eqref{est:proof-1-3-1} to bound 
\beqas
\abs{E(u)-E(\zz)}&\leqs \abs{dE(\zz)(w)}+\norm{w}_\zz^2
\leqs \norm{dE(\zz)}_*^2+\norm{w}_\zz^2
\leqs \norm{dE(\zz)}_*^2+\norm{dE(u)}_*^2\\
&\leqs \tfrac{\de^2 \log\fmu}{(\log\mu)^2}+\norm{dE(u)}_{*}^2\leqs \tfrac{\de^2}{\log\mu}+\norm{dE(u)}_{*}^2.
\eeqas
To obtain \eqref{est:proof-1-2} it hence suffices to check that  
\beq
\label{est:proof-1-5}
\abs{E(\zz)-E^*}\leqs  \tfrac{\de^2}{\log \mu} \text{ for all } \zz\in \ZZ_0.
\eeq
If $\de=0$, i.e.~if $c_0=c_1$, this is trivially true. 
Conversely if $c_0\neq c_1$ then
we can interpolate between 
$\zz$ and $\hat\zz =U^*\circ (q_\mu+c_0)$, which satisfies $E(\hat \zz)=E^*$, 
by considering the family $(\zz_t)_{t\in  [0,\abs{c_1-c_0}]}$ that is generated by $U_{1,t}=U^*(\cdot +c_1+t \tfrac{c_0-c_1}{\abs{c_1-c_0}})$. 
 As $\norm{\partial_t \zz_t}\leqs 1$ and $\norm{\beta(\zz_t)}\leqs \de(\zz_t)\leqs \de:=\de(\zz)$ 
we can use \eqref{est:var-E-always-Efixed} to bound 
 $\abs{\partial_t E(\zz_t)}\leqs \tfrac{\de}{\log \mu}$ which, when integrated over $0\leq t\leq \abs{c_0-c_1}\sim \de$, yields \eqref{est:proof-1-5} and hence \eqref{est:proof-1-2}.

From
 \eqref{est:proof-1-3-2} and  \eqref{est:proof-1-2} we now immediately deduce that
$$\abs{E(u)-E^*}\leqs \min(\norm{dE(u)}_*,\log\mu \norm{dE(u)}_*^2)$$
which completes the proof of Theorem \ref{thm:1}. 

It remains to prove the $L^2$-\Lojns-estimate \eqref{claim:thm1-Loj}  claimed in Corollary \ref{cor:thm1}. To this end we first combine
\eqref{est:proof-1-3-2} with the fact that $\norm{\cdot}_{L^2(S^2,g_{S^2})}\leqs \norm{\cdot}_\zz$, and hence $\norm{dE(u)}_*\leqs \norm{\tauS(u)}_{L^2(S^2,g_{S^2})}$, to bound   
\beqs
\label{est:proof-1-l2} 
\tfrac{\de}{\log\mu}\leqs \norm{\tauS(u)}_{L^2(S^2,g_{S^2})}\norm{y_\zz}_{L^2(S^2,g_{S^2})} + \norm{\tauS(u)}_{L^2(S^2,g_{S^2})}^2.
\eeqs
As $y_\zz$ is obtained from a variation of just the bubble $U_1$, it is 
mainly concentrated on a small ball around the point where the bubble is attached and hence its $L^2$ norm is small compared to $\norm{y_\zz}_\zz=1$.
 To be more precise, for variations satisfying \eqref{ass:variations} for which $q_0$ and $U_0$  are fixed
 we shall see in \eqref{est:norm-yz-proof} that
 \beq \label{est:norm-yz}
 \norm{\peps \zz}_{L^2(S^2,g_{S^2})}\leqs \mu^{-1}(\log\mu)^\half+\tfrac{\de+\eta_0}{f_\mu\log\mu }.
 \eeq
In the present situation where $\eta_0\sim 1$ we hence get that 
 $\norm{y_\zz}_{L^2(S^2,g_{S^2})}\leqs (\log\mu)^{-1} f_\mu^{-1}$ 
so 
\eqref{est:proof-1-3-2} allows us to conclude that
\beq
\label{est:proof-1-cor} 
\tfrac{\de^2}{\log\mu}\leq \norm{\tauS(u)}_{L^2(S^2,g_{S^2})}^2.
\eeq
Combined with
\eqref{est:proof-1-2} this immediately yields the claim  $\abs{E(u)-E^*}\leqs\norm{\tauS(u)}_{L^2(S^2,g_{S^2})}^2 $ made in Corollary \ref{cor:thm1}.

\begin{rmk} \label{rmk:spectral-gap}
As noted in the introduction, the reason we cannot obtain these results from the existing theory is that the spectral gap at zero for the Jacobi-operator at $U^*(q_\mu)$ tends to zero as $\mu\to \infty$ and that the size of the neighbourhoods on which the existing theory yields \Lojns-estimates scales like this spectral gap. 
To be more precise, as the energy defect scales like $\frac{\de^2}{\log\mu}$ we can see that
the Jacobi-operator must have an eigenvalue that scales like $\frac{1}{\log\mu}$.
 If we did not include these directions in our set of singularity models $\ZZ$ this would mean that the argument of Lemma \ref{lemma:first step} would break down even for harmonic maps $u$ unless we knew a priori that  $\norm{w}_\zz$ is bounded by a small enough multiple of $\frac{1}{\log\mu}$, compare \eqref{est:first-step-proof-lambda}.
\end{rmk}

 \subsection{Proofs of Theorems  \ref{thm:2} and \ref{thm:3} and of Corollary \ref{cor:thm2}} $ $\\
 In this section we always consider settings for which the assumptions of either Theorem \ref{thm:2} or of Theorem \ref{thm:3} are satisfied and let $\ZZ$ be the corresponding set of singularity models. We note that in these settings $\abs{d\beta(0)}$ and $\norm{\beta}$ are of order $1$ and that we have
 \begin{lemma}
 \label{lemma:energy-defect}
 The energy defect of elements $\zz\in\ZZ$ is bounded by 
 \beq
\label{claim:energy-defect}
 \abs{E(\zz)-E^*}\leqs \tfrac{\de^2}{\log \mu}+\min(\nu_0,\nu_1)+\mu^\mhalf \nutot + \TT_1^{\gamma_1}+ \TT_0^{\gamma_0}+ \TT \mu^{-1}\log\mu,
 \eeq 
where $\gamma_{0,1}\in (1,2] $ are  the exponents for which the classical \Lojns-Simon estimate \eqref{est:Simon} is valid near $U_{0,1}^*$ while $E^*:=E(\om_0)+E(\om_1)=\deg(q_0)E(U_0^*)+\deg(q_1)E(U_1^*)$.
 \end{lemma}

\begin{proof}[Proof of Lemma \ref{lemma:energy-defect}]
It suffices to carry out the proof for elements of $\ZZ_0$ and using the symmetry described in Remark \ref{rmk:symm} we can assume without loss of generality that $\nu_0\leq \nu_1$.

We will first prove the lemma in the special case where $\de=0$. To this end, we consider  the family $\zz(t)$, $t>0$, which we obtain by varying the bubble scale according to $\mu(t)=\mu(\zz)e^t$ while keeping the maps $U_i$ and $q_i$ fixed and use that \eqref{est:var-E-always-rational} ensures that
$$\abs{\tfrac{d}{dt}E(\zz(t))}=\mu(t)\abs{\partial_\mu E(\zz(t))}\leqs \nu_0(t)+\mu(t)^\mhalf \nu_1(t)+ \TT \mu(t)^{-1}\log\mu(t) \text{ for } t\in [0,\infty).$$
As $\nu_i(t)=\max_{j\leq n_i^*}\abs{a_j(q_i)}\mu(t)^{-j}\leqs e^{-t} \nu_i(0)$ and as $E(\zz(t))\to \deg(q_0)E(U_0)+\deg(q_1)E(U_1)$ for $t\to \infty$ we hence deduce that 
\beqas
\abs{E(\zz)-E^*}&\leq \sum_i\deg(q_i)\abs{E(U_i)-E(U_i^*)}+C\big[\nu_0+\mu^\mhalf \nu_1+\TT \mu^{-1}\log\mu \big]
\eeqas
which, combined with \eqref{est:Simon}, yields the claim in this special case where $\de=0$. 

So suppose that $\de=\de(U_1,U_0)\neq 0$. In this case we can choose
$U_{i,t}=U_i(\cdot+c_i(t))$,  $t\in [0,\de]$ with  $\abs{\dot c_i(t)}\leqs 1$ in a way that $\de(U_{0,t},U_{1,t})=\de-t$, compare Lemma \ref{rmk:de}. 
The resulting family $\zz_t\in \ZZ $ satisfies 
\beqs
\abs{\tfrac{d}{dt} E(z_t)}\leqs \nutot+\tfrac{\de}{\log\mu}+\tfrac{\TT}{\log\mu},
\eeqs
compare 
\eqref{est:var-E-always-Efixed}, so turns $\zz$ into an element $\hat\zz=\zz_{t=\de}$ with 
$\de(\hat \zz)=0$ and
 \beqs
\abs{E(\zz)-E(\hat \zz)}\leqs \tfrac{\de^2}{\log\mu}+\de \nutot+\tfrac{\de}{\log\mu} \TT\leqs \tfrac{\de^2}{\log\mu} +\nutot^2\log\mu+\TT^2.
\eeqs 
As $\nutot\log\mu\leqs \mu^{-1}\log\mu\leqs \mu^\mhalf$ and as we have already shown that  
 \eqref{claim:energy-defect} holds  for $\hat\zz$ we hence obtain the claimed bound on the energy defect also for general elements of $\ZZ_0$. 
\end{proof}

Let now $u$ be a map with $\norm{\tauS(u)}_{L^2(S^2,g_{S^2})}\leq 1$  which satisfies  \eqref{est:norm-less-eps} for  
$\eps>0$ as in Lemma \ref{lemma:first step}. 

We can combine the above Lemma \ref{lemma:energy-defect} with 
the estimate 
 $$\abs{E(u)-E(\zz)}\leqs  \norm{dE(\zz)}_*^2+\norm{w}_\zz^2\leqs \nutot^2 \log\mu +\tfrac{\de^2 (\log\fmu)^2}{(\log\mu)^2}+\TT^2+\norm{dE(u)}_*^2
 $$
that we obtain from \eqref{est:w} and Lemma \ref{lemma:first-variation} to see that
 \beq
\label{est:energy-defect-general}
 \abs{E(u)-E^*}\leqs \tfrac{\de^2}{\log \mu}+\min(\nu_0,\nu_1)+\mu^\mhalf \nutot+ \TT^\gamma + \TT\mu^{-1}\log\mu+
 \norm{dE(u)}_*^2
 \eeq 
for any such $u$ and for $\gamma:=\min(\gamma_1,\gamma_2)\in(1,2]$. 

In order to prove the bounds on $\abs{E(u)-E^*}$ claimed in Theorems \ref{thm:2} and \ref{thm:3} we hence need to show that all these quantities are controlled by $\norm{dE(u)}_*$. For this it is convenient to choose $\fmu:=\log\mu$ which of course satisfies \eqref{ass:fmu} for $\bar \mu$ sufficiently large. 

We distinguish between three different cases, depending on whether terms involving $\de$, $\TT$ or $\nutot$ dominate, beginning with

\textbf{Case 1: } Suppose that $\zz$ is so that 
\beq
\label{ass:de-dominates}
\nutot (\log\mu)^\half\leq \tfrac{\de}{\log\mu} \text{ and } 
 \TT (\log\log\mu)^{-1/2}\leq \tfrac{\de}{\log\mu}.
\eeq
In this case we let $y_\zz$ be the unit element in the direction of a variation $\peps \zz_\eps$ as obtained in Lemma \ref{lemma:energy-expansion-rotation}.
 Inserting \eqref{ass:de-dominates} into the estimates obtained in Lemmas
 \ref{lemma:first-variation}, \ref{lemma:second-variation} and   \ref{lemma:energy-expansion-rotation}
 gives
 \beq
 \norm{dE(\zz)}_*\leqs \tfrac{\de(\log\log\mu)^\half}{\log\mu}, 
 \norm{d^2E(\zz)(y_\zz)}_*\leqs \tfrac{(\log\log\mu)^\half}{\log\mu}
  \text{ and } 
  dE(\zz)(y_\zz)\geq c \tfrac{\de}{\log\mu}.
 \eeq
 This ensures that $Q(\zz,y_\zz)\leqs \frac{\log\log\mu}{\log\mu}$ is small so  
 Lemma \ref{lemma:first step} yields
\beqs
\tfrac{\de}{\log \mu}\leqs \abs{dE(u)(y_\zz)}+\norm{dE(u)}_*^2\leqs  \norm{dE(u)}_*.
\eeqs
Combined with \eqref{ass:de-dominates} this 
immediately implies that $\mu^{-\min(n_0^*,n_1^*)}\leqs \nutot\leqs \norm{dE(u)}_*$  and inserting these estimates into 
 \eqref{est:energy-defect-general} gives
 $$\abs{E(u)-E^*}\leqs \tfrac{\de}{\log\mu} +\big(\tfrac{\de}{\log\mu} \big)^{\gamma} (\log\log\mu)^{\gamma/2}+\norm{dE(u)}_*^2\leqs \tfrac{\de}{\log\mu}+\norm{dE(u)}_*^2\leqs \norm{dE(u)}_*.
 $$

Having thus seen that the claims of Theorems \ref{thm:2} and \ref{thm:3} hold in this case where $\de$ dominates we now consider

\textbf{Case 2:} Suppose that $\zz$  is so that
\beq
\label{ass:tau-dominates}
\tfrac{\de (\log\log\mu)^{1/2}}{\log \mu}\leq \TT \text{ and } \bar \nu (\log\mu)^\half\leq \TT  .
\eeq 
In this situation  \eqref{est:energy-defect-general} 
ensure that $\abs{E(u)-E^*}\leqs \TT$ and we construct $y_\zz$ using a variation as considered in Lemma \ref{lemma:energy-expansion-tension} for $i$ so that $\TT_i=\TT$. Lemmas \ref{lemma:energy-expansion-tension} and  \ref{lemma:first-variation}  yield
\beq
\label{est:case2-1}
dE(\zz)(y_\zz)\geqs \TT \text{ and } \norm{dE(\zz)}_*\leqs \TT 
\eeq
while Lemma \ref{lemma:second-variation} and \eqref{est:eta-tau} ensure that we can make $\norm{d^2E(\zz)(y_\zz,\cdot)}_*$ smaller than any given constant by reducing $\si$ accordingly. Hence we can ensure that $Q(\zz,y_\zz)\leqs \bar\La^{-1}$ and 
apply Lemma \ref{lemma:first step} to get
$$\TT\leqs \norm{dE(u)(y_\zz)}+\norm{dE(u)}_{*}^2\leqs \norm{dE(u)}_{*}.
$$
Combined with \eqref{est:energy-defect-general} and  \eqref{ass:tau-dominates} this immediately yields the claims of the theorems.

Finally, if $\zz$ is so that neither of the above cases applies then we must have
 
\textbf{Case 3:}
 \beq
 \label{ass:nu-dominates}
 \tfrac{\de}{\log\mu}\leq \nutot (\log\mu)^\half \text{ and } \TT\leq \nutot (\log\mu).
 \eeq
%
 
In this case we can choose a variation as in Lemma \ref{lemma:energy-expansion-rational} to obtain $y_\zz$ with
\beqs
dE(\zz)(y_\zz)\geqs \nutot \text{ and } \norm{d^2E(\zz)(y_\zz,\cdot)}_*\leqs \mu^{-1}\log\mu. \eeqs
At the same time Lemma \ref{lemma:first-variation} and \eqref{ass:nu-dominates} ensure that  $\norm{dE(\zz)}_*\leq \nutot \log\mu$ and hence that $Q(\zz,y_\zz)\leqs \mu^{-1}(\log\mu)^2$ will be small. 

Lemma \ref{lemma:first step} hence implies that 
$\nutot\leqs \norm{dE(u)}_{*}$. This immediately yields the claimed bound on the bubble scale and, as \eqref{ass:nu-dominates} ensures that all terms that involve $\de$ or $\TT$ in \eqref{est:energy-defect-general} are small compared to $\nutot$, also that 
$\abs{E(u)-E^*}\leqs \norm{dE(u)}_*$. 

Having thus completed the proofs of both Theorems \ref{thm:2} and \ref{thm:3} we finally 
explain how the above arguments have to be modified to obtain the

\begin{proof}[Proof of Corollary \ref{cor:thm2}] $ $\\
Here it is convenient to choose $\fmu:=\mu^{\si_1}$ where we fix $\si_1>0$ small enough so that all results in Section \ref{sect:key-lemmas} hold for this $\si_1$ and for all sufficiently small $\si$ and large $\bar\mu$. 
 Compared to the previous proof this choice of $\fmu$ gives us more flexibility with regards to which type of variation we can choose  as we obtain better bounds on the error terms in the energy expansions. This is useful as variations of the bubble are preferable to variations of the base when proving $L^2$-\Lojns-estimates since they result in elements $y_\zz$ with smaller $L^2$-norm. 

To prove this corollary we first note that in this setting
 \eqref{est:energy-defect-general} reduces to
 \beq
\label{est:energy-defect-special}
 \abs{E(u)-E^*}\leqs \tfrac{\de^2}{\log \mu}+\min(\nu_0,\nu_1)+\mu^\mhalf \nutot
+\norm{\tauS(u)}_{L^2(S^2,g_{S^2})}^2
 \eeq 
as all elements of $\HH_1^\si(U_i^*)$ are harmonic and as $\norm{dE(u)}_*\leqs \norm{\tauS(u)}_{L^2(S^2,g_{S^2})}$.  

We fix $c_1=\frac{\si_1}4 \in (0,\frac1{16})$ and distinguish between cases dependent on whether
 $
 \mu^{c_1}\nutot\leq \tfrac{\de}{\log\mu}. 
$

If  $
 \mu^{c_1}\nutot\leq \tfrac{\de}{\log\mu} 
$  then we can argue exactly as in the proof of Corollary \ref{cor:thm1} to see that 
$$\tfrac{\de}{(\log\mu)^2}\leqs \norm{\tauS(u)}_{L^2(S^2,g_{S^2})}^2,$$ 
compare \eqref{est:proof-1-cor}.
As $\nutot\geqs \mu^{-n^*}$, $n^*:=\min(n_1^*,n_0^*)$, we furthermore know that in this case
$$\nutot^{1-\tfrac{c_1}{n^*}}\leqs \mu^{c_1}\nutot\leq \tfrac{\de}{\log\mu}\leqs \norm{\tauS(u)}_{L^2(S^2,g_{S^2})} $$
 and 
hence immediately obtain from \eqref{est:energy-defect-special} that
$$\abs{E(u)-E^*}\leqs \norm{\tauS(u)}_{L^2(S^2,g_{S^2})}^\al \text{ for } \al:= (1-\tfrac{c_1}{n^*})^{-1}>1.$$

It hence remains to consider the case
$\tfrac{\de}{\log\mu}\leq \mu^{c_1}\nutot$ where 
 \eqref{est:energy-defect-special} and Lemma \ref{lemma:first-variation} give
\beq
\label{est:proof_cor-2-1}
\abs{E(u)-E^*}\leqs \min(\nu_0,\nu_1)+\mu^\mhalf \nutot +\norm{\tauS(u)}_{L^2(S^2,g_{S^2})}^2 \text{ and } 
\norm{dE(z)}_{*}\leqs \mu^{c_1}(\log\mu)^\half \nutot,
\eeq
while Lemmas \ref{lemma:energy-expansion-rational} and \ref{lemma:second-variation} give
\beq \label{est:proof_cor-2-2}
dE(\zz)(y_\zz)\geq c\nu_{i^*}-C\mu^{-2c_1}\nutot \text{ and } \norm{d^2E(\zz)(y_\zz,\cdot)}_*\leqs \mu^{-1+c_1}(\log\mu)^\half 
\eeq
for the unit element $y_\zz$ that is generated by a variation of $q_i$ as considered in Lemma \ref{lemma:energy-expansion-rational}.

If 
 $\mu^{c_1}\nu_0\geq \nu_1$ we can vary the rational map of the bubble to get a  $y_\zz$ with
 $$\norm{y_\zz}_{L^2(S^2,g_{S^2})}\leqs \mu^{-1}(\log\mu)^{\half} 
 \text{ and } dE(\zz)(y_\zz)\geqs \nu_{0}\geqs \mu^{-c_1}\nutot,
  $$
 compare \eqref{est:norm-yz} and \eqref{est:proof_cor-2-2}. Hence $Q(\zz)(y_\zz)\leqs \mu^{-1+3c_1}\log\mu$ is small and we can apply the estimate \eqref{est:first-step} from  Lemma \ref{lemma:first step} to see that 
$$\nu_0\leqs \mu^{-1}(\log \mu)^\half\norm{\tauS(u)}_{L^2(S^2,g_{S^2})} + \norm{dE(u)}_*^2.
$$
We can thus bound 
$$\mu^{-\half}\nutot+\nu_0\leqs \nu_0\leqs \norm{\tauS(u)}_{L^2(S^2,g_{S^2})}^\al \text{ for any } \al\in (1, (1-\tfrac{1}{n^*})^{-1})$$ and 
\eqref{est:proof_cor-2-1} implies that 
the claim of the corollary holds true for any such exponent.

In the final case where 
$\nu_0\leq \mu^{-c_1} \nu_1$ and 
$\tfrac{\de}{\log\mu}\leq \mu^{c_1}\nutot$
we instead consider a variation of $q_0$ and use the above estimates and Lemma \ref{lemma:first-variation} to see that 
$\nutot\leqs \nu_1\leqs \norm{dE(u)}_*\leqs \norm{\tauS(u)}_{L^2(S^2,g_{S^2})}$. From \eqref{est:proof_cor-2-1} we hence obtain that 
$$\abs{E(u)-E^*}\leqs 
\mu^{-c_1} \nutot+\norm{\tauS(u)}_{L^2(S^2,g_{S^2})}^2 \leqs  \nutot^\al +\norm{\tauS(u)}_{L^2(S^2,g_{S^2})}^2
\leqs \norm{\tauS(u)}_{L^2(S^2,g_{S^2})}^\al .
$$
for any $ \al\in (1,1+ \frac{c_1}{n^*})$.
\end{proof}

\section{Variations of the energy along $\ZZ_0$}\label{section:proof-of-lemmas}
In this section we carry out the proofs of the key lemmas on the behaviour of the energy  on the manifold $\ZZ_0$ of singularity models, i.e.~of Lemmas \ref{lemma:general-energy-expansion}, \ref{lemma:energy-expansion-rational}, \ref{lemma:energy-expansion-rotation} and \ref{lemma:energy-expansion-tension}, that we stated in Section \ref{sect:key-lemmas}. 

Before we turn to these proofs we
collect a number of estimates, which will be used both throughout this section as well as in the subsequent Section \ref{sect:variations} where we will prove the properties of $dE(\zz)$ and $d^2E(\zz)$ claimed in Lemmas \ref{lemma:positive-def-2ndvar}, \ref{lemma:first-variation} and  \ref{lemma:second-variation}.

 \subsection{Technical estimates}
 \label{subsec:technical} $ $\\
Given any $\zz\in \ZZ_0$ we work in fixed stereographic coordinates which we can assume to be scaled in a way that the corresponding function $q_0$ satisfies $\abs{q_{n_0^*}(q_0)}=1$. As we represent $q_1(\frac{1}{\mu\cdot})$ for any fixed $\zz$ by a function 
$q_1$ satisfying $\abs{a_{n_1^*}(q_1)}=1$ we hence get $\mu_1(\zz)=\mu$, $\mu_0(\zz)=1$ and
\beq
\label{eq:nice-radii}
r_1(\zz)=\fmu \mu^{-1}\ll \hat r(\zz)=\mu^\mhalf\ll r_0(\zz)=\fmu^{-1}.\eeq
To simplify the notation we will in the following write for short  $\phi_\mu=\phimul$
for the cut-off function \eqref{def:phi-mu-la} used in the definition of $\zz$ and also use the  slight abuse of notation of writing 
 $\peps \phi_\mu:= \peps\vert_{\eps=0} \phi_{r_0(\zz_\eps),r_1(\zz_\eps)}$ even though \eqref{eq:nice-radii} is only applicable at $\eps=0$.

For most of the proofs it will be more convenient to 
 view the rational maps $q_i$ as maps into the sphere $S^2_{p^*}:=S^2-(0,0,1)$, which is shifted in a way that elements of $\RR$ map $0$ to the origin in $\R^3$.  
We hence associate to any given $U:\Chat\to N$ and any rational map $q:\Chat\to \Chat$ the maps $\hat U:= U\circ \pi_{p^*}^{-1}:S^2_{p^*}\to N$ and $\hat q= \pi_{p^*} \circ q :\Chat\to S^2_{p^*}$ for
  $\pi_{p^*}:=\pi-(0,0,1):\R^2\to S^2_{p^*}$. 

As at least one of $q_0(z)$ or $q_{1,\mu}(z):= q_1(\frac{1}{\mu z})$ will be small for every $z\in\C$ we can easily check that the map $\hat q_\mu= \pi_{p^*} \circ (\hat q_0\circ \pi_{p^*}^{-1}+\hat q_{1,\mu}\circ \pi_{p^*}^{-1})$ that represents $q_\mu=q_0+q_{1,\mu}$ is so that 
\beq
\label{est:qmu-by-q}
 \abs{\rmmuhat}\leqs \abs{\rmbbhatmu}+\abs{\rmbmhat} \text{ and }
\abs{\na \hat q_\mu}\leqs  \abs{\na \rmbmhat}+\abs{\na \rmbbhatmu}.
 \eeq

Here and in the following we use the convention that unless specified otherwise all quantities are computed with respect to the euclidean metric on $\R^2$, rather than with respect to the metric that is induced by $g_{S^2}$.

We can hence estimate the maps 
$\tilde u_i:=
u_i-U_i(0)=\hat U_i\circ \hat q_\mu-U_i(0)$ by
 \beq\label{est:ui-tilde}
  \abs{\tilde u_i}
  \leqs \abs{\rmbbhatmu}+\abs{\rmbmhat},\quad 
 \abs{\na u_i} 
 \leqs \abs{\na \rmbmhat}+\abs{\na \rmbbhatmu} 
 \text{ and } \abs{\Delta u_i}\leqs \abs{ \na \rmbbhatmu}^2+\abs{\na \rmbmhat}^2,
 \eeq
where the last estimate follows as the maps $\hat q_i$ are harmonic maps into the sphere.

Since the maps $\hat q_i\circ \pi_{p^*}^{-1}:S_{p^*}^2\to S_{p^*}^2$ and their variations are uniformly bounded in $C^k$, we have 
\beq\label{est:base-simple}
\abs{\rmbmhat(z)} +\abs{\peps\rmbmhat}
 \leqs \abs{\pi(z)-\pi(0)}\leqs 
\tfrac{\abs{z}}{1+\abs{z}}, \quad \abs{\na \rmbmhat}+ \abs{\peps\na \rmbmhat}\leqs\abs{\na \pi(z)}\leqs 
 \tfrac{ 1}{1+\abs{z}^2}
 \eeq
 and 
 \beq
 \label{est:bubble-simple}
 \abs{\rmbbhatmu}+\abs{\peps\na \rmbmhat}\leqs \abs{\pi(\tfrac{1}{\mu z})-\pi(0)}\leqs 
 \tfrac{1}{1+\abs{\mu z}}, \quad \abs{\na \rmbbhatmu}+\abs{\peps \na \rmbbhatmu}\leqs \abs{\na (\pi(\tfrac{1}{\mu z}))}\leqs 
 \tfrac{\mu}{1+\abs{\mu z}^2}.
\eeq
In particular we can always estimate $\abs{\na \hat q_\mu}\leqs \rho_\zz$, though will later need more refined bounds in settings  where we deal with  rational maps $q_{0,1}^*$ that are branched at $z=0$, compare   \eqref{est:rmbm-expansion} and \eqref{est:tilde-om-expansion} below.

The above estimates imply in particular that 
\beq
\label{est:standard-est-annulus}
\norm{\abs{\rmbbhatmu}+\abs{\rmbmhat}+\abs{\peps \rmbbhatmu}+\abs{\peps\rmbmhat}}_{L^\infty(A)}+
\norm{\abs{\na \rmbbhatmu}+\abs{\na \rmbmhat}}_{L^2(A)}\leqs \fmu^{-1},
\eeq
where we recall that the annuli $A$, $\hat A$ and $A^*$ are as defined in \eqref{def:annuli}. 
Similarly we have  
\beq
\label{est:I-phimut} 
 \int_{\Om_1} \phimu^2\rho_\zz^2\leqs \fmu^{-2}(\log\mu)^{-2},\,  \int_{\Om_1} \phimu\rho_\zz^2\leqs \fmu^{-2}(\log\mu)^{-1},\, \int_{A}\abs{\na\phimu} \rho_\zz\leqs \fmu^{-1}(\log\mu)^{-1} \eeq
for the cut-off function $\phi_\mu$ that is defined in \eqref{def:phi-mu-la} and that satisfies 
 \beqa
 \label{est:phi}
 \abs{\na \phi_\mu}\leqs c_\mu r^{-1} \mathbbm{1}_A, \quad 
\abs{\Delta \phi_\mu} \leqs c_\mu r^{-2} \mathbbm{1}_{A^*} 
\text{ for } c_\mu:=\log(\tfrac{r_0}{r_1}).
\eeqa
As $\hat \gamma$ is a geodesic with $\hat \gamma(0)=U_1(0)$ and $\abs{\hat\gamma'}=\de$ we have
 \beq
 \label{est:gamma-tilde}
 \abs{ \tgabb} \leq \de \phi_\mu 
 \text{ with }\abs{\na \gamma}\leqs \de \abs{\na \phimu}\leqs c_\mu \de r^{-1} \mathbbm{1}_A
 \eeq 
and
\beqa
\label{eq:tension-gamma}
\tau(\gamma)=\Delta \phimu\cdot\hat\gamma'\circ \phimu \text{ while }
P_{\gamma}^\perp(\Delta \gamma)= \abs{\na \phimu}^2 \cdot A(\hat\gamma)(\hat \gamma',\hat\gamma')\circ \phimu
\eeqa
so
\beqa
\label{est:der-gamma}
    \abs{\tau(\gamma)}\leqs \de \abs{\Delta \phimu}\leqs \de c_\mu r^{-2}  \mathbbm{1}_{A^*}\quad
  \text{ and } 
  \quad \abs{P_{\gamma}^\perp(\Delta \gamma)}\leqs \de^2 \abs{\na \phimu}^2\leq c_\mu^2r^{-2}\de^2 \mathbbm{1}_A.
\eeqa
As $\cmu\sim (\log\mu)^{-1}$ we hence get
\beq
\label{est:norm-tension-gamma}
\norm{\tau(\gamma)}_{L^1}\leqs \de (\log\mu)^{-1},\quad 
\norm{P_{\gamma}^\perp(\Delta \gamma)}_{L^1}\leqs \de^2 (\log\mu)^{-1}
\text{ and } \norm{\Delta \gamma}_{L^1}\leqs  \de (\log\mu)^{-1}.
\eeq
To obtain suitable bounds on the variations of $\phimu$ we recall that $\abs{\peps \mu_i}\leqs 
\mu_i$ and that $\abs{\partial_\mu \fmu}\mu\leq \fmu$. This ensures that 
\beq
\label{est:var- radii}
\abs{\peps r_i}\leqs r_i, \quad \abs{\peps \hat r}\leqs \hat r \text{ and } \abs{\peps c_{r_0,r_1}}\leqs \cmu^2\leqs (\log\mu)^{-2}
\eeq 
which in turn imply that
\beqa \label{est:peps-phi}
\abs{\peps \phi_\mu} &\leqslant c_\mu  \mathbbm{1}_{A} ,\quad 
\abs{\peps \na \phi_\mu}\leqslant c_\mu^2r^{-1}\mathbbm{1}_A+c_\mu r^{-1}\mathbbm{1}_{A^*},\quad 
\abs{\peps \Delta \phi_\mu}\leqslant c_\mu r^{-2}\mathbbm{1}_{A^*}.
\eeqa
We also observe that the definition of $\phimu$ and the above estimates ensure that  
\beq
\label{est:norm-phi}
\norm{\na \phimu}_{L^2}^2=
2\pi \cmu +O((\log\mu)^{-2}) \text{ and that }
\abs{\peps \norm{\na \phi_\mu}_{L^2(\R^2)}^2} \leqs (\log\mu)^{-2}.
\eeq
As $\abs{\peps \hat \gamma }+\abs{\peps \na \hat \gamma} \leqs \etazero$, $\eta_0$ defined in \eqref{def:quantities-var},
we can furthermore  bound
\beq 
\label{est:var-gamma}
\abs{\peps \gamma}\leqs \etazero+c_\mu \delta\mathbbm{1}_{A}, \quad
\abs{\peps \tgabb}
 \leqs \etazero \phimu+c_\mu \delta \chiA, \quad \abs{\peps \tgabm}\leqs \etazero(1-\phimu) + c_\mu \delta\chiA
 \eeq
and
 \beq
 \label{est:var-gamma-der}
\abs{\peps \na \gamma}\leqs (\etazero+\de\cmu)\cmu r^{-1}\chiA + \de\cmu r^{-1}\chiAstar
\text{ so } \norm{\peps \na \gamma}_{L^2}\leqs \etazero(\log\mu)^\mhalf +\de (\log\mu)^{-1}.
\eeq
We also note that 
 $\abs{q_0(z)}\leqs \abs{z}$ is small 
on the set $\{z: \abs{z}\leq 2\mu^\mhalf\}$ where we consider $v_1$ and hence $j_1=d\beta(0)(q_0(z))$. On this set we can hence bound 
 $\abs{q_0(z)}\leqs \abs{\hat q_0(z)}$ and thus have
 \beqs 
  \abs{\modbb}\leqs \sitilt \abs{\rmbmhat},\quad  \abs{\na \modbb}\leqs \sitilt \abs{\na \rmbmhat}, \quad \Delta\modbb=0 
\eeqs
as well as 
\beq
\label{est:var-maps-rat}
\abs{\peps \modbb}\leqs \sitilt \abs{\peps \rmbmhat} +
\etatilt\abs{\rmbmhat},  \,
 \abs{\peps \na \modbb}\leqs 
\sitilt \abs{\peps\na  \rmbmhat} + 
\etatilt\abs{\na \rmbmhat}. 
\eeq
Analogue estimates hold for 
$\modbm$ on the set $\abs{z}\geq \half \mu^\mhalf$ where $v_0$, and hence $j_0$, are considered.

 We also observe that we can combine the above estimates with  
\eqref{est:var-phicut} to see that if  $\zz_\eps$ is a variation along which 
$q_0$ and $U_0$ are fixed then
\beqa
\norm{\peps \zz}_{L^2(S^2,g_{S^2})}&\leqs \norm{\peps v_1 \abs{\na \pi}}_{L^2(\DD_{2\hat r})}+
 \norm{\peps v_0 \abs{\na \pi}}_{L^2(\R^2\setminus \DD_{\frac{\hat r}2})}+\abs{\hat A}^\half \norm{v_1-v_0}_{L^\infty}\\
 &\leqs \norm{z}_{L^2(\DD_{2\hat r})}+\norm{\tfrac{1}{
 1+\mu\abs{z}}
 \abs{\na \pi}}_{L^2(\R^2)}+\tfrac{\de+\eta_0}{\log\mu}(\norm{\phi_\mu}_{L^2(\R^2)}+\abs{\DD_{2\hat r}}^\half)+\mu^{-1}\\
 &\leqs \mu^{-1}(\log\mu)^\half+\tfrac{\de+\eta_0}{f_\mu\log\mu }
 \label{est:norm-yz-proof}
 \eeqa
as claimed in the proofs of Corollaries \ref{cor:thm1} and \ref{cor:thm2}.

\subsection{General variations of the energy along $\ZZ$} $ $

The purpose of this section is to derive the expression 
 for $\peps E(\zz_\eps)$ claimed in Lemma 
\ref{lemma:general-energy-expansion}.

Thanks to the symmetry of the construction and the fact that on  $\Om_1=\{z:\abs{z}\leq \mu^\mhalf\}$ our map
 $\zz$ is essentially described by $\zz_1=\pi_N(v_1)$
the main step in this proof is to show

\begin{lemma}\label{lemma:general-energy-expansion-proof}
If $\zz_\eps$ is a variation in $\ZZ_0$ so that \eqref{ass:variations} is satisfied then  
\beqa 
\label{eq:en-exp-main-lemma}
\int_{\Om_1} \peps \zz_1\Delta \zz_1 &=-\int_{\Om_1}\modbb\peps \Delta \bb - \peps \bb \tau( \bb)
+\int_{A\cap \Om_1}  \Delta \gamma\peps \gamma
+\err_1
\eeqa 
for 
an error term which is bounded by 
$\abs{\err_1}\leq R_1$ for 
\beqa
\label{est:main-error}
 R_1:=& \,\delta (\log\mu)^{-1} \fmu^{-2}+\de^2(\log \mu)^{-1} + \etazero [\sitilt I_{3} +\de (\log\mu)^{-1}\fmu^{-1}]\\
&+
\sitilt^2 (I_1+I_2)+ \sitilt\etatilt I_1
+\sitilt I_{\partial \Om_1}\\
&+\TT [\sitilt (I_3+I_4)+\eta_0(\log\mu)^{-1}\fmu^{-2}+\abs{\peps d\beta(0)}I_3]+\eta_\tau \sitilt I_3.
\eeqa

\end{lemma}
Here $\de,  \tau, \eta_0$ and $\eta_\tau$ are as in \eqref{def:quantities-z} and \eqref{def:quantities-var} and the integrals 
$I_\cdot=I_{\cdot}(q_0,q_{1,\mu})$ are as in

\begin{lemma} \label{lemma:errors-energy-expansion}
For any variations $q_{0,1}^{(\eps)}$ in $\RR^\si(q_{0,1}^*)$ satisfying \eqref{ass:variations} we can bound 
\begin{eqnarray}
\label{est:I1}
& &I_1:= \int_{\Om_1} 
\abs{\rmbmhat}^2\rho_q^2+\abs{\rmbbhatmu}^2\abs{\na \rmbmhat}^2 \leq \nutot^2\log\mu \\
\label{est:I2}
& &I_2:=\int_{\Om_1} 
\abs{\rmbmhat}\abs{\peps \rmbmhat} \rho_q^2+\abs{\peps \hat q_0}\abs{\hat q_{1,\mu}}\abs{\na \hat q_0}\abs{\na q_{1,\mu}}
+\abs{\rmbbhatmu}\abs{\na \rmbmhat}^2+\abs{\rmbmhat}\rho_q \abs{\na \rmbmhat}
\leq \mu^\mhalf \nutot\\
\label{est:I3}
& &I_3:= \int_{\Om_1} \abs{\rmbmhat}\rho_q^2+ \abs{\na \rmbmhat}\abs{\rmbbhatmu}\rho_q \leq \nutot\\
\label{est:I4}
 & &I_4:= \int_{\Om_1} \abs{\peps \rmbmhat}\rho_q^2+\abs{\rmbmhat}\abs{\peps\na \rmmuhat}\rho_q\leqs \mu^{-1}\\
 \label{est:Ibdry}
& &I_{\partial \Om}:=
\int_{\partial \Om_1}\abs{\rmbmhat}^2+\abs{\rmbbhatmu}^2  dS
\leq \mu^\mhalf \nutot
\end{eqnarray}
where we write for short 
for $\rho_q:=\abs{\na \rmbmhat}+\abs{\na \rmbbhatmu}.$
\end{lemma}

This lemma ensures that $R_1$ defined by \eqref{est:main-error} is bounded by the quantity $R_0$ defined in \eqref{est:en-exp-nice-error} and hence that $\err_1$ can be included in the error term in Lemma \ref{lemma:general-energy-expansion}. 

Similarly the contributions of $\zz-\zz_1$ to $\peps E(\zz_\eps)$ are of lower order, namely

\begin{lemma}
\label{lemma:errphi}
For any variation in $\ZZ_0$ as considered above we can bound 
\beqa \label{claim:errcut}
\babs{  \int_{\Om_1} \peps \zz \Delta \zz-\peps \zz_1\Delta \zz_1 }
\leq \Ctilt \nutot(\mu^\mhalf+\eta_0 )+\de \mu^{-1}+\TT\nutot\mu^{-1}(\norm{\beta}+\norm{\peps \beta}).
\eeqa
\end{lemma}

Combining these three lemmas with their analogues on $\Om_0$, which follow by symmetry, we get 
\beqa
\label{est:hallo-0}
\peps E(\zz)&=-\int_{\Om_1} \peps \zz\Delta \zz-\int_{\Om_0} \peps \zz\Delta \zz\\
&=\int_{\Om_1}\modbb\peps \Delta \bb - \peps \bb \tau( \bb) +\int_{\Om_0}\modbm\peps \Delta \bm - \peps \bm \tau( \bm) 
-\int_{A}  \Delta \gamma\peps \gamma+\err_2
\eeqa
for an error term that satisfies the required bound $\abs{\err_2}\leqs R_0$. 

As 
 $\abs{\hat\gamma'}=\de$ we can use \eqref{est:norm-phi} to see that
\beq
\label{est:hallo-1}
-\int_A\peps \gamma \Delta \gamma =\tfrac12 \tfrac{d}{d\eps} \norm{\na \gamma}_{L^2}^2= \tfrac12 \tfrac{d}{d\eps} 
\big( 
\abs{\hat \gamma'}^2 \norm{\na \phimu}_{L^2}^2\big) = \pi \cmu \tfrac{d}{d\eps}\delta(U_0,U_1)^2 +\err_3 
\eeq
for an error term that is bounded by $\abs{\err_3}\leqs \tfrac{\de (\de+\etazero)}{(\log\mu)^{2}}\leqs R_0$. 

Finally, we  
can use that 
 \beqa
 \label{est:hallo-2}
 \text{deg}(q_1) 
 \tfrac{d}{d\eps}E(\hat U_1)&= 
\tfrac{d}{d\eps} E(\hat U_1\circ \rmbbhatmu)=-\int_{\R^2} \tau(\hat U_1\circ \rmbbhatmu) \cdot \peps (\hat U_1\circ \rmbbhatmu)\\
&=-\int_{\Om_1} \tau(u_1) \cdot \peps u_1+\err_3+\err_4.
\eeqa
where we use \eqref{est:bubble-simple} to bound
$$\abs{\err_3}\leqs \abs{\int_{\Om_0} \tau(\hat U_1\circ \rmbbhatmu) \cdot \peps (\hat U_1\circ \rmbbhatmu)}\leqs\TT\int_{\Om_0}\abs{\na \rmbbhatmu}^2\leqs \mu^{-1}\TT\leqs R_0
$$
and
 set $\err_4:= \int_{\Om_1}\peps \bb\tau(\bb)-\peps (\hat U_1\circ \rmbbhatmu)  \tau(\hat U_1\circ \rmbbhatmu)$. To bound this term we can use that 
 $\abs{\peps (\bb-\hat U_1\circ \rmbbhatmu)}\leqs \abs{\rmbmhat}+\abs{\peps\rmbmhat}$, while  \eqref{eq:trafo-tension} ensures that 
\beqas
\abs{\tau(\bb)-\tau(\bbu\circ \rmbbhatmu)}&=\abs{\tau(\hat U_1)\circ \rmmuhat\abs{\na \rmmuhat}^2-\tau(\hat U_1)\circ \rmbbhatmu \abs{\na \rmbbhatmu}^2}
\leqs \TT\cdot (\abs{\rmbmhat}\rho_q^2
+\abs{\na \rmbmhat}\rho_q).\eeqas
Combined with \eqref{est:base-simple} and \eqref{est:bubble-simple} this yields
\beqas
\abs{\err_4}\leqs 
\TT \int_{\Om_1}(\abs{\rmbmhat}+\abs{\peps \rmbmhat})\rho_q^2+\abs{\na \rmbmhat}\rho_q\leqs \mu^{-1}\log\mu \TT\leqs R_0.
\eeqas
This reduces the proof of 
Lemma \ref{lemma:general-energy-expansion} to the proofs of the three lemmas stated above, which we carry out in the remainder of this section.

\newcommand{\rhoz}{\rho_\zz}

To prove Lemma \ref{lemma:general-energy-expansion-proof}, and later on also Lemmas \ref{lemma:first-variation} and \ref{lemma:second-variation}, it is useful to expand 
\beq \label{eq:expansion-z-Om1}
\zz_1=\pi_N(v_1)
=\bb+P_{\bb}(\tilde \gamma_\bb+\modbb)+\errom=\bb+\tgabb +\modbb-P_{\bb}^\perp( \tgabb+\modbb)+\errom
 \eeq
 and note that the resulting error term can be written as 
 \beq
 \label{eq:err}
 \errom=\int_0^1 \big(d\pi_N(\bb+t(\tgabb+\modbb))-d\pi_N(\bb)\big)(\tgabb+\modbb) dt\eeq
and hence bounded by
\beqa
\label{est:err}
\abs{\errom}\leqs 
\abs{\modbb}^2+\abs{\tgabb}^2\leqs \sitilt^2\abs{\rmbmhat}^2+\de^2\phimu
\eeqa
since $d\pi_N(p)$ agrees with the projection $P_p$ onto the tangent space $T_pN$ whenever $p\in N$.

 \begin{proof}[Proof of Lemma \ref{lemma:general-energy-expansion-proof}]
Writing $\zz_1$ as in \eqref{eq:expansion-z-Om1} and using that $\Delta \modbb=0$ 
we find that 
 \beqa \label{eq:main-expansion-Om1}
 \int_{\Om_1}\peps \zzphi\Delta \zzphi 
 &=\int_{\Om_1}\peps \bb \,\Delta \bb+\peps(P_{\bb}\modbb)\,\Delta \bb  -\peps \bb\,\Delta(P_{\bb}^\perp\modbb) +\peps \gamma \,\Delta \gamma +T_1\\
 &= \int_{\Om_1}\peps \bb \,\tau(\bb)- P_\bb\tu \peps \Delta \bb -\peps \bb\,\Delta(P_{\bb}^\perp\modbb) +\peps \gamma \,\Delta \gamma +T_1+T_2
 \eeqa
 for $T_2:=  \int_{\Om_1}\peps(P_\bb \modbb \Delta \bb) =\int_{\Om_1}\peps(\modbb\tau(\bb))$ and
 \beqa \label{est:T1}
\abs{ T_1}&\leq \int_{\Om_1} \abs{\peps (P_{\bb}\tgabb) \,\Delta \bb}+\abs{\peps \errom}\, \abs{\Delta\bb}+\abs{\peps (\zzphi-\bb)\,\Delta(P_\bb^\perp\modbb)}
 \\& \qquad +\int_{\Om_1} \abs{\peps \zzphi\, \Delta \errom}+
 \abs{\peps \zzphi\Delta(P_\bb^\perp\tgabb)}+\abs{ \int_{\Om_1}\peps(\zzphi-\gamma)\Delta \gamma } \\
 &=: T_1^{(1)}+...+T_1^{(6)}.
 \eeqa
We can rewrite the third term in \eqref{eq:main-expansion-Om1} using integration by parts, which results in a boundary term of the form 
\beq
\label{def:B1}
B_1:=-\int_{\partial \Om_1} \pn (P_\bb^\perp\modbb)\cdot\peps \bb- P_\bb^\perp\modbb\cdot\pn\peps \bb dS,\eeq
to obtain the desired expression
 \beqa \label{eq:main-expansion-Om1-2}
 \int_{\Om_1}\peps \zzphi\Delta \zzphi =
\int_{\Om_1}\peps \bb \,\tau(\bb)- \int_{\Om_1} \modbb \peps \Delta \bb +\int_{\Om_1} \peps \gamma\Delta \gamma +T_1+T_2+B_1 .
 \eeqa
 This reduces the proof of the lemma to showing that all  
 error terms $T_1^{(i)}$, $T_2$ and $B_1$ obtained above are controlled by the quantity $R_1$ defined in \eqref{est:main-error}.

To bound $T_1^{(1)}$ we consider $A(p)$ as a bilinear form 
 $A(p):\R^N\times \R^N\to T_p^\perp N$ on the whole space which vanishes in  normal directions and write 
 \beqas
 \peps (P_\bb\tgabb ) A(\bb)(\na \bb,\na \bb)&=- P_\bb(\tgabb) \peps \big(A(\bb)(\na \bb,\na \bb)\big)\\
 &=- P_\bb(\tgabb)\cdot (\peps A(\bb))(\na \bb,\na\bb).\eeqas
 Combined with \eqref{eq:trafo-tension}, \eqref{est:gamma-tilde} and \eqref{est:var-gamma} we
can hence bound 
\beqas
T_1^{(1)} &\leqs \int \abs{\peps(P_\bb(\tgabb)}\abs{\tau(\bb)}+\int \abs{P_\bb(\tgabb)}\abs{(\peps A(\bb))(\na \bb,\na\bb)}\\
  &\leqs \big[(\de +\etazero)\TT +\de\big] \int_{A\cap \Om_1}(\phimu+c_\mu)\abs{\na \rmmuhat}^2\leqs (\etaz \TT+\de)(\log\mu)^{-1}\fmu^{-2},
  \eeqas  where we use \eqref{est:standard-est-annulus} and \eqref{est:I-phimut}  in the last step. 
 Hence $T_1^{(1)}\leqs R_1$ as required. 
     
Next we 
use that 
that the formula \eqref{eq:err} for $\err_{u_1}$, combined with  \eqref{est:var-gamma} and 
 \eqref{est:var-maps-rat},  gives
\beqa
\label{est:peps-err-1}
\abs{\peps \errom}&\leqs \abs{\modbb}^2+\abs{\tgabb}^2+(\abs{\modbb}+\abs{\tgabb})(\abs{\peps\modbb}+\abs{\peps \tgabb})\\
&\leqs \sitilt^2\abs{\rmbmhat}^2+\de^2\phimu+(\sitilt \abs{\rmbmhat}+\de\phimu)(\sitilt\abs{\peps \rmbmhat}+\etatilt\abs{\rmbmhat}+\etaz\phimu+c_\mu\de\chiA).
\eeqa
Using additionally 
\eqref{est:ui-tilde}, \eqref{est:standard-est-annulus} and \eqref{est:I-phimut} we hence see that
\beqas
T_1^{(2)}\leqs& \sitilt^2(I_1+I_2)+ \sitilt\etatilt I_1 + 
(\log\mu)^{-1}\fmu^{-3} \de +\eta_0 \abs{d\beta(0)} I_3
\\&+\eta_0\de (\log\mu)^{-2}\fmu^{-2}+\de^2 (\log\mu)^{-1}\fmu^{-2}
\eeqas
 and thus that $T_1^{(2)}\leqs R_1$. 
 
We then split
\beqs
T_1^{(3)}
 \leq \int_{\Om_1}\abs{P_\bb^\perp(\peps(\zzphi-\bb))}\abs{\Delta P_\bb^\perp\modbb}+\int_{\Om_1} \abs{\peps (\zzphi-\bb)}\abs{P_\bb(\Delta P_\bb^\perp\modbb)}.
\eeqs
As $\abs{\Delta \bb}+\abs{\na \bb}^2\leqs \abs{\na \rmmuhat}^2\leqs \rho_q^2$ and as $\Delta j_1=0$ we can certainly estimate
\beqas 
\abs{\Delta(P_{\bb}^\perp\tu)}&\leqs \abs{\tu}(\abs{\na \bb}^2+\abs{\Delta \bb})+\abs{\na \modbb}\abs{\na \bb}\leqs \sitilt (\abs{\rmbmhat}\rho_q+
\abs{\na \rmbmhat}) \rho_q.
\eeqas 
To obtain an improved bound on the tangential part of this quantity we write 
 $P_{\bb}^\perp\tu =\sum_j 
\langle \tu ,\nu_\bb^j\rangle \nu_\bb^j$ to see that  $
\abs{P_{\bb}(\Delta(P_{\bb}^\perp \tu)} \leqs \abs{\tu}\rho_q^2+\rho_q \abs{P_{u_1}^\perp( \na \tu)},
$ and write
\beq
\label{eq:writing-om}
\modbb=d\bmu(0)(\rmbm)-d\bbu(0)(\rmbm)=P_{\bbu(0)}( \modbb)+(P_{\bmu(0)}-P_{\bbu(0)})(d\bmu(0)(\rmbm))
\eeq
to see that
\beqa
\label{est:Pu1-j}
\abs{P_{\bb}^\perp(\na \modbb)}&\leqs \abs{u_1-U_1(0)}\abs{\na \modbb}+\abs{\bmu(0)-\bbu(0)}\abs{\na \rmbmhat}
\leqs \sitilt(\abs{\rmbmhat}+\abs{\rmbbhatmu})\abs{\na \rmbmhat}+\de \abs{\na \rmbmhat}.
\eeqa
Combined we hence get
\beqas
\abs{P_{\bb}(\Delta(P_{\bb}^\perp \tu))}&\leqs\sitilt \abs{\rmbmhat}\rho_q^2 
+\sitilt \abs{\rmbbhatmu}\abs{\na \rmbmhat} \rho_q+\de \abs{\na \rmbmhat}\rho_q.
\eeqas 
On the other hand, while
we can only bound 
\beqas
\abs{\peps(\zzphi-\bb)}&=\abs{\peps(P_{\bb}(\tgabb+\modbb)+\errom)}
\leqs \abs{\modbb}+\abs{\peps \modbb}+\abs{\tgabb}+\abs{\peps \tgabb}\\
& \leqs (\sitilt+\etatilt)\abs{\rmbmhat}+\sitilt \abs{\peps \rmbmhat}+ \de(\phimu+c_\mu \mathbbm{1}_A)+\eta_0\phimu,
\eeqas
compare \eqref{est:peps-err-1} and \eqref{est:var-gamma}, 
we get a stronger bound of
\beqas
\abs{P_\bb^\perp(\peps(\zz_1-\bb)}=\abs{(P_\bb^\perp-P_\zzphi^\perp)(\peps \zzphi)}\leq \abs{\bb-\zzphi}\abs{\peps \zzphi}\leqs \abs{\tu}+\abs{\tgabb}\leqs  \sitilt\abs{\rmbmhat}+\de \phimu.
\eeqas
Combined we hence get  
\beqas
T_1^{(3)}&
\leqs
\int_{\Om_1}\sitilt\abs{\rmbmhat}
\abs{\Delta P_\bb^\perp\tu}+\de \int_{A\cap \Om1}
(\phimu+c_\mu)\abs{\Delta P_\bb^\perp\tu}\\
&+\int_{\Om_1} (\etatilt\abs{\rmbmhat}+\sitilt \abs{\peps \rmbmhat} )
\abs{P_\bb(\Delta P_\bb^\perp\tu)}+\etaz \int_{A \cap \Om_1}\phimu \abs{P_\bb(\Delta P_\bb^\perp\tu)}
\\
&\leqs  \sitilt^2 (I_1+I_2)+
\sitilt\etatilt I_1 +\de (\etatilt +\sitilt+\etaz)\mu^\mhalf 
+\etaz\sitilt I_3
\eeqas
and thus $T_1^{(3)}\leqs R_1$. 
Here we simply use that $\norm{\rmbmhat}_{L^\infty(\Om_1)}+\norm{\na \rmbmhat}_{L^2(\Om_1)}\leqs \mu^\mhalf$ to deal with the terms involving $\de$ as these only give lower order contributions.

Next we estimate 
\beqas
T_1^{(4)}&\leq 
\thalf\int_{\Om_1} \abs{P_{\zz_1} (d^2\pi_N(\bb)(\na j_1,\na j_1))}+
\int_{\Om_1}\abs{\Delta\errom-\thalf d^2\pi_N(\bb)(\na j_1,\na j_1)}.
\eeqas
To deal with the first term we exploit that 
 $A(x)=-dP_x$ and thus that
\beq 
\label{eq:d2pi-A} 
d^2\pi(x)(v,w)=- A(x)(v,w)\in T_x^\perp N \text{ for all } x\in N, v,w\in T_xN.
\eeq
Combined with  \eqref{est:Pu1-j} this allows us to bound 
$$\abs{P_{\zz_1}(d^2\pi_N(\bb)(\na j_1,\na j_1))}\leqs \abs{\zzphi-\bb}\abs{\na \modbb}^2+\abs{P_{\bb}^\perp \na \modbb}\abs{\na \modbb}\leqs \abs{d\beta(0)}^2 (\abs{\hat q_0}+\abs{\hat q_{1,\mu}} + \abs{\tgabb}))\abs{\na \hat q_0}^2$$
while the formula \eqref{eq:err} for $\err_{u_1}$ allows us to check that 
\beqas
\abs{\Delta\errom-\half d^2\pi_N (\bb)(\na \modbb,\na\modbb)}&\leqs 
(\abs{\modbb}+\abs{\tgabb})
\big[\abs{\na \bb}(\abs{\na \modbb}+\abs{\na \gamma})+\abs{\na \modbb}^2 +\abs{\na\gamma}^2 +\abs{\tau(\gamma)}\big]\\
&\quad +(\abs{\modbb}^2+\abs{\tgabb}^2)\rho_q^2+\abs{\na\gamma}^2+\abs{\na \gamma} \abs{\na \modbb}.
\eeqas
Combined 
we hence get
\beqas
T_1^{(4)}
&\leqs \sitilt^2 (I_1+I_2)+\de\sitilt \mu^\mhalf
 +\de^2(\norm{\phimu^2\rho_q^2}_{L^1(\Om_1)}+\norm{\na \phi_\mu}_{L^2}^2)+
 (\norm{j_1}_{L^\infty(\Om_1)} + \de) \norm{\tau(\gamma)}_{L^1}\\
 &\leqs 
  \sitilt^2 (I_1+I_2)+\de\sitilt \mu^\mhalf +\de^2(\log\mu)^{-1}
\eeqas
where we also use \eqref{est:standard-est-annulus}, \eqref{est:I-phimut}, \eqref{est:phi} and \eqref{est:norm-tension-gamma} in the second step. Thus also $T_1^{(4)}\leqs R_1$.

Next we estimate 
\beqas 
T_1^{(5)}&\leqs \int_{A\cap \Om_1}\abs{P_\zzphi(\Delta (P_\bb^\perp \tgabb))}\\
&\leqs \int_{A\cap \Om_1} \abs{u_1-\zz_1} (\abs{\Delta \gamma}+ \abs{\na \gamma}\rho_q) +
 \abs{P_{u_1}^\perp(\na\gamma)}\rho_q+\abs{\tgabb}\rho_q^2\\
 &\leqs (\mu^\mhalf+\de) \de (\log\mu)^{-1} +\int_{A\cap \Om_1}\abs{u_1-\gamma} \abs{\na\gamma}\rho_q+\de(\log\mu)^{-1} \fmu^{-2}
\eeqas
where we used \eqref{est:norm-tension-gamma}, $\abs{u_1-\zz_1}\leq \mu^\mhalf+\de$ and 
\eqref{est:I-phimut} in the last step.
Using furthermore that $\abs{u_1-\gamma}\leqs \abs{u_1-U_1(0)}+\abs{\tgabb}\leq \fmu^{-1} +\de \phimu$ on $A$ 
we hence get 
\beqa 
T_1^{(5)}
&\leqs \mu^\mhalf\de +\de^2(\log\mu)^{-1} + \de(\log\mu)^{-1} \fmu^{-2}\leqs R_1.
\eeqa
Finally, to bound $T_1^{(6)}$ we  write 
$\zzphi$ instead  as 
\beqas
\label{exp:z-by-gamma}
\zzphi=\gamma+P_{\gamma}(\tilde u_1+\modbb)+\errga \text{ for } \errga=\int_0^1 \big(d\pi_N(\gamma+t(\modbb+\tilde \bb)\big)-d\pi_N(\gamma))(\tilde \bb+\modbb)
\eeqas
and
use \eqref{est:standard-est-annulus} and \eqref{est:var-gamma}
to see that 
$$\abs{\peps(\zz_1-\gamma)}\leqs \abs{\peps \gamma} \fmu^{-1}+\fmu^{-2}\leqs (\eta_0+\de(\log\mu)^{-1})\fmu^{-1}+\fmu^{-2} \text{ on } A.$$ 
As $\norm{\Delta \gamma}_{L^1}\leqs \de(\log\mu)^{-1}$ we hence immediately deduce that  
\beqas
T_1^{(6)}&\leq \de (\log\mu)^{-1} \fmu^{-2}+\de^2(\log\mu)^{-2}\fmu^{-1}+\eta_0\de (\log\mu)^{-1}\fmu^{-1}\leqs R_1.
\eeqas
Since  $T_2=\int_{\Om_1}\peps (\modbb \tau(\bbuhat)\circ \rmmuhat \abs{\na \rmmuhat}^2)$, compare \eqref{eq:trafo-tension}, we furthermore have 
\beqas 
\abs{T_2}&
\leqs \TT \int\abs{\peps \modbb}\rho_q^2+\abs{\modbb}(\abs{\peps \rmmuhat}\rho_q^2+\rho_q\abs{\peps\na \rmmuhat} )
+\eta_\TT
\int\abs{\modbb} \abs{\na \rmbbhatmu}^2\\
&\leqs \TT\sitilt (I_3+I_4) +\TT\abs{\peps d\beta(0)} I_3+
\eta_\TT\sitilt  I_3\leqs R_1.
\eeqas
Finally, to estimate $B_1$ we use 
\eqref{eq:writing-om} to see that
\beqa
\label{est:Pperp-u}
\abs{P_{\bb}^\perp\modbb}&
\leqs \abs{\rmmuhat}\abs{\modbb}+\de \abs{\rmbmhat}\leqs \sitilt \abs{\rmmuhat}\abs{\rmbmhat}+\de \abs{\rmbmhat}. 
\eeqa 
Since 
$\abs{\na u_1}+\abs{\peps \na u_1}\leqs 1$ on  $\partial \Om_1$ and hence 
$\abs{P_\bb (\pn P_\bb^\perp\modbb)}\leqs \abs{\pn \bb}\abs{P_\bb^\perp \modbb}\leqs \abs{P_\bb^\perp \modbb}$
we can hence bound 
\beqa
\abs{B_1}\leqs \int_{\partial \Om_1}\abs{P_\bb^\perp\modbb}dS
\leqs \abs{d\beta(0)}I_{\partial \Om}+\mu^{-1}\de,
\eeqa
where the last step follows since 
$\abs{\rmbmhat}$ and $\abs{\partial \Om_1}$ are of order $O(\mu^{\mhalf})$. 

We have thus shown that all error terms are controlled by $R_1$ which completes the proof of Lemma \ref{lemma:general-energy-expansion-proof}.
 \end{proof}

\begin{proof}[Proof of Lemma \ref{lemma:errors-energy-expansion}]
To prove this lemma we will use that 
\beq \label{est:rmbm-expansion}
\abs{\hat q_0(z)}+\abs{z}\abs{\na \hat q_0(z)}\leqs  \sum_{j=1}^{\dps} \abs{a_j(q_0)} \abs{z}^j  \text{ for every } z\in \C
\eeq
and for $q_0 \in \RR^\si(q_0^*)$, $\hat q_0=\pi_{p^*}^{-1}\circ q_0$.
To see this we note that our choice of $\si$ ensures that the right hand side is at least $\half \abs{z}^{n_0^*}$ and that \eqref{est:rmbm-expansion} is  hence trivially satisfied outside of any fixed size disc $\DD_{\bar r}(0)$ since the maps $\hat q_0$ satisfy uniform $C^k$ bounds. 
As $q_0^*(0)=0$ and as the maps $\hat q_0$ are uniformly close to $\hat q_0^*$ we can fix $\bar r=\bar r(q_0^*)>0$ small enough so that $\abs{q_0(z)}\leq 1$ for all $\abs{z}\leq 2\bar r$. Standard estimates from complex analysis then ensure that $\abs{a_j(q_0)}\leqs (2\bar r)^{-j}$ for all $j$ and hence that 
$\sum_{j\geq n_0^*+1} j \abs{a_j(q_0)} \abs{z}^j\leqs \abs{z}^{n_0^*}$ for all 
 $\abs{z}\leq \bar r$ which ensures that \eqref{est:rmbm-expansion} also holds on $\DD_{\bar r}(0)$. 

Similarly, given $q_1\in \RR^\si(q_i^*)$ and $\mu> \bar \mu$ we can use that $\hat q_{1,\mu}=\pi_{p^*}\circ q_1(\tfrac{1}{\mu\cdot})$ is so that  
\beq
\label{est:tilde-om-expansion}
\abs{\rmbbhatmu(z)}\leqs \sum_{k=1}^{\dqs} \tfrac{\abs{\akbb}}{1+\abs{\mu z}^k} \text{ and } \abs{\na \rmbbhatmu(z)}\leqs \mu \sum_{k=1}^{\dqs} \tfrac{\abs{\akbb}}{1+\abs{\mu z}^{k+1}}. 
\eeq
The proof of all claims made in the lemma now follow from short explicit calculation 
based these two estimates \eqref{est:rmbm-expansion} and \eqref{est:tilde-om-expansion}. To begin with we note that
\beqa
\label{est:int-quadratic-u-om}
\int_{\Om_1} \abs{\rmbmhat}^2 \abs{\na \rmbbhatmu}^2+\abs{\rmbbhatmu}^2\abs{\na \rmbmhat}^2 
&\leqs 
\sum \abs{\ajbm}^2\abs{\akbb}^2 \mu^{-2j}\mu^2 \int_{\Om_1}
 \tfrac{ \abs{\mu z}^{2j}}{1+\abs{\mu z}^{2k+2}} +  \tfrac{ \abs{\mu z}^{2j-2}}{1+\abs{\mu z}^{2k}}
 \\
 &\leqs \sum  \abs{\ajbm}^2\abs{\akbb}^2\mu^{-2j}\int_0^{\sqrt{\mu}} 
 \tfrac{t^{2j-1}}{1+t^{2k}} dt \\
 &\leqs \sum  \abs{\ajbm}^2\abs{\akbb}^2
 (\mu^{-2j}+\mu^{-j}\mu^{-k}\log\mu)\leq \nutot^{2}\log\mu
 \eeqa
where here and in the following we only need to sum over $j=1,\ldots, \dps$ and $k=1,\ldots,\dqs$ and can use that the corresponding coefficients $a_j(q_i)$ are uniformly bounded as we only consider elements of $\RR^\si(q_i^*)$. 
As we furthermore have 
\beq
\label{est:int-bm-squared}
\int_{\Om_1}\abs{\hat q_0}^{2\alpha} \abs{\na\rmbmhat}^2\leqs \sum_j \abs{\ajbm}^{2(1+\alpha)}
\int_{\Om_1}\abs{z}^{2(1+\alpha)j-2} \leqs \nutot^{1+\alpha}
\text{ for any } \alpha\geq 0,
\eeq
so in particular for $\al=1$, we obtain the claimed estimate \eqref{est:I1} on $I_1$. 

Using \eqref{est:rmbm-expansion}, \eqref{est:tilde-om-expansion} and \eqref{est:base-simple} 
 we then get
\beqas
\int_{\Om_1}\abs{\rmbmhat}\abs{\peps \rmbmhat}\abs{\na \rmbbhatmu}^2+\abs{\peps \hat q_0}\abs{\hat q_{1,\mu}}\abs{\na \rmbbhatmu}\abs{\na q_{1,\mu}}
&\leqs \sum \abs{\ajbm}\abs{\akbb}^2 \mu^{-j-1}\int_0^{\sqrt{\mu}}\tfrac{t^{j+1}}{1+t^{2k+1}} dt
\\
&\leqs \nutot\mu^{-1}\log\mu.
\eeqas
As the other contributions to $I_2$ are bounded by 
\beqas
\int_{\Om_1}\abs{\rmbbhatmu}\abs{\na \rmbmhat}^2\leqs \sum \abs{\akbb}\abs{\ajbm}^2\mu^{-2j}\int_0^{\sqrt{\mu}}\tfrac{t^{2j-1}}{1+t^k}\leqs \nutot^{3/2}\leqs \mu^\mhalf \nutot
\eeqas
and 
$$\int \abs{\hat q_0} \rho_q\abs{\na \hat q_0}\leq (\nutot^2\log\mu)^\half \nutot^\half\leqs \mu^{-1}(\log\mu)^\half \nutot,$$
compare \eqref{est:int-quadratic-u-om} and \eqref{est:int-bm-squared}, we hence obtain the claimed bound \eqref{est:I2} on $I_2$. 

Combining \eqref{est:int-bm-squared}
with 
\beqas
\int_{\Om_1}\abs{\rmbmhat}\abs{\na\rmbbhatmu}^2+\abs{\na\rmbmhat} \abs{\rmbbhatmu} \abs{\na\rmbbhatmu} &\leqs \sum \abs{\ajbm}\abs{\akbb}^2\mu^{-j}\int_{0}^{ \sqrt{\mu}} \tfrac{t^{j+1}}{1+t^{2k+2}}\\
&\leqs \sum \abs{\ajbm}\abs{\akbb}^2(\mu^{-j}+\mu^{-j/2}\mu^{-k}\log\mu)\leqs \nutot
\eeqas
then shows that $I_3$ satisfies \eqref{est:I3}, while 
\eqref{est:base-simple} and \eqref{est:bubble-simple} give the claimed bound  on
\beqas
I_4\leq  \int_{\Om_1} \abs{z} (1+\tfrac{\mu^2}{1+\abs{\mu z}^4})
\leqs\mu^{-\tfrac32}+ \mu^{-1}\int_0^{\sqrt{\mu}} \tfrac{t^2}{1+t^4} dt\leqs \mu^{-1}.
\eeqas
Finally, as 
$\abs{\rmbmhat}^2+\abs{\rmbbhatmu}^2\leqs  \nutot$  on $\partial\Om_1=\partial \DD_{\mu^{-\half}}$, we get that 
$
\abs{I_{\partial\Om_1}}\leq  \abs{\partial\Om_1} \nutot \leqs \nutot\mu^{-\half}
$.
\end{proof}

In Section \ref{sect:variations} we will also use that a very similar calculation gives
\beq
\label{est:int-for-later}
\int \abs{\peps\hat q_0}^2\rho_q^2 +\abs{\hat q_0}^2 (\abs{\peps\na \hat q_0}^2+\abs{\peps \na\hat q_{1,\mu}}^2) \leqs \eta_{rat} (\mu^{-j^*}+\nutot)^2\log\mu.
\eeq 
We finally need to prove that the contribution of the error terms $\errcut=\zz-\zz_1$ to $\peps E(\zz_\eps)$ is controlled as described in Lemma \ref{lemma:errphi}. 

As  $\errcut$ is supported on $\hat A$ we first collect some estimates that are valid on this set where $\abs{z}\sim \mu^\mhalf$ and where we can hence bound
\beq
\label{est:A-cut-q}
\abs{q_0}+\abs{q_{1,\mu}}+\mu^\mhalf (\abs{\na q_0}+\abs{\na q_{1,\mu}}) \leqs \nutot^\half 
\eeq
and 
\beq
\label{est:A-cut-var-q}
\abs{\peps q_0}+\abs{\peps q_{\mu,1}}+\mu^\mhalf 
\abs{\peps \na q_0}+\abs{\peps \na q_{\mu,1}}
\leqs \mu^\mhalf.
\eeq
Since 
 $v_0-v_1=\beta(\rmmu)-\beta(0)-d\beta(0)(\rmmu)$ 
we have 
\beq
\label{est:v0-v1}
\abs{v_0-v_1}+\mu^\mhalf \abs{\na (v_0-v_1)}+\mu^{-1} \abs{\Delta(v_0-v_1)}\leqs \norm{\beta} (\abs{q_\mu }^2+\mu^{-1}\abs{\na q_\mu}^2) \leq \norm{\beta}\nutot\eeq
as well as 
\beq
\label{est:v0-v1-var}
\abs{\peps(v_0-v_1)}+\mu^\mhalf \abs{\peps\na (v_0-v_1)}+\mu^{-1} \abs{\peps\Delta(v_0-v_1)} \leq (\norm{\beta}+\norm{\peps \beta})\nutot+\eta_{rat} \norm{\beta}\mu^{-j^*}.\eeq
We furthermore note that $\gamma$ is harmonic on $\Acut$ so we can bound 
\beq
\label{est:nav-middle} 
\abs{\na v_i}^2+\abs{\Delta v_i} \leqs \rho_q^2+\abs{\na\gamma}^2 \leqs \mu (
 \nutot+\tfrac{\de^2}{(\log\mu)^2}).
 \eeq

Since the error term is given by 
\beq
\label{eq:errphi}
\errcut=\pi_N( v_1+t\varphi(v_0-v_1))-\pi_N(v_1)=
\int_0^1 d\pi_N(v_1+t\varphi(v_0-v_1))(\varphi(v_0-v_1)) dt
\eeq
for a function $\varphi$ that vanishes for $\abs{z}\leq \half \mu^\mhalf$ and that satisfies $\abs{\na\varphi}^2+\abs{\Delta \varphi}\leqs\mu$  
we have
\beqa
\label{est:errcut}
\abs{\errcut}\leqs \Ctilt\nutot \qquad \text{ and } \qquad 
 \abs{\Delta  \errcut}\leqs\Ctilt \mu \nutot.
\eeqa 
As $\abs{\peps \mu_i}\leqs \mu_i$ we have
\beq
\label{est:var-phicut}
\abs{\peps(\varphi_\mu(z))}\leqs\mu^{-\half}\abs{\peps \tfrac{\mu_1}{\mu_0}}\abs{\phi'(\mu^\half \abs{z}})\abs{z}\leqs \mathbbm{1}_{\Acut} 
\eeq
allowing us to bound
\beqa
\label{est:errcut-eps-rough}
\abs{\peps \errcut}&\leqs \abs{v_0-v_1}+\abs{\peps(v_0-v_1)}
\leqs (\Ctilt +\norm{\peps \beta})\nutot+\eta_{rat} \norm{\beta} \mu^{-j^*}.
\eeqa 
We will later also use that the above estimates also allow us to bound
\beq
\label{est:errcut-peps-Delta}
\mu^{-1}\abs{\peps \Delta\errcut}
\leqs  (\Ctilt +\norm{\peps \beta})\nutot+\eta_{rat} \norm{\beta} \mu^{-j^*}.
\eeq

Based on these estimates we can now complete the

\begin{proof}[Proof of Lemma \ref{lemma:errphi}]
Since 
$\abs{\peps \zz_1}\leqs  \mu^\mhalf +\de+\etaz$ on $\Acut$ and since $\abs{\Acut}\leqs \mu^{-1}$ we can bound 
\beqas
\abs{\int_{\Om_1}\peps \zz\Delta \zz-\peps \zz_1\Delta \zz_1}&\leq \int_{\hat A\cap \Om_1} (\abs{\peps \zz_1}+\abs{\peps \errcut})\abs{\Delta\errcut}+ \int_{\hat A\cap \Om_1}\abs{\peps \errcut \Delta \zz_1} \\
&\leqs
 \Ctilt( \mu^\mhalf +\de+\etaz)\nutot +\int_{\hat A\cap \Om_1}\abs{\peps \errcut \cdot \Delta \zz_1}.
\eeqas
To bound this last term we split $\peps \errcut \Delta \zz_1$ into the 
contributions of the tangential and the normal parts. 
Since $P_{\zz_1}=d\pi(\zz_1)$ 
 and since $\abs{v_1-\zz_1}\leq \de+\norm{\beta} \nutot^\half$ and $\abs{\peps v_1}\leqs \eta_0+\mu^\mhalf$ on $\Acut$ 
we get 
\beqa 
\label{est:errcut-eps-perp}
\abs{P_{\zz_1}^\perp(\peps \errcut)} 
&\leqs  (\abs{v_1-\zz_1}+\abs{v_0-v_1})\abs{\peps(\varphi (v_0-v_1))}+\abs{\peps v_1}\abs{\varphi(v_0-v_1)}
\\
&\leq (\de+\norm{\beta}\nutot^\half)(\norm{\beta}+\norm{\peps \beta})\mu^{-1}
+(\eta_0+\mu^{\mhalf}) \nutot\norm{\beta} \\
&\leqs 
\norm{\beta}\mu^{-3/2}
+\de(\norm{\peps \beta}+\norm{\beta})\mu^{-1} +\eta_0 \norm{\beta} \nutot
\eeqa
and can thus in particular bound $\abs{P_{\zz_1}^\perp(\peps \errcut)} \leqs \mu^{-1}(\norm{\beta}+\de)$. 

Since \eqref{est:nav-middle} ensures that $\abs{\Delta \zz_1}=\abs{\Delta(\pi_N(v_1))}\leqs \mu(\nutot+\de^2)$ on $\Acut$ and since  $\abs{\Acut}\leqs \mu^{-1}$ 
this suffices to estimate 
\beqas
\int\abs{P_{\zz_1}^\perp \peps \errcut}\abs{\Delta \zz_1}&\leq 
(\nutot+\de^2) 
\mu^{-1}\big[\norm{\beta}+\de] 
\eeqas
and hence to see that this term
is bounded by the right hand side of \eqref{claim:errcut}.

From \eqref{eq:expansion-z-Om1} we furthermore get that 
\beqas
\abs{P_{u_1}(\Delta \zz_1)}&\leqs 
\abs{\tau(u_1)}+\abs{\na \gamma}^2+\abs{\na j_1}^2+\rho_q( \abs{\na\gamma}+\abs{\na j_1})\leq \TT\nutot \mu+\de^2\mu +\de\nutot^\half \mu +
 \norm{\beta}\nutot \mu.
 \eeqas
 As  \eqref{est:errcut-eps-rough} ensures that
 $\abs{\peps \errcut}\leqs (\norm{\beta}+\norm{\peps\beta})\mu^{-1}$ and as 
 $\abs{\Acut}\leqs \mu^{-1}$
 we immediately conclude that 
also   
$\int\abs{\peps \errcut}\abs{P_{u_1}\Delta \zz_1}$ is controlled by the right hand side of \eqref{claim:errcut}, which completes the proof of the lemma. 
\end{proof}

\subsection{Dominating terms in the energy expansions for variations of the rational maps}
\label{sec:main-term} $ $\\
In this section we explain how the general energy expansion proven in the previous section can be used to identify a variation 
for which 
 Lemma \ref{lemma:energy-expansion-rational} holds true.

By symmetry it suffices to prove the claim of the lemma for $i=1$, i.e.~to identify a variation of  $q_1$ so that 
\eqref{claim:energy-expansion-case1} holds for
the resulting family $\zz_\eps$  for which we keep $U_{0.1}$, $q_1$ and $\mu$ fixed. 

For such variations we know from Lemma \ref{lemma:general-energy-expansion}  that
\beq\label{est:ddE-rat-proof}
dE(\zz)(\peps \zz)=\int_{\Om_1}j_1\peps \Delta \bb +\int_{\Om_0}j_0\peps \Delta \bm  +\err_1
\eeq
now for an error that is bounded by $\abs{\err_1}\leqs R_2$ for 
\beq
\label{est:error-rational}
R_2:=
 \nutot 
 \Ctilt \mu^\mhalf
 +\tfrac{\de^2}{\log\mu}+ 
\tfrac{\de}{f_\mu^{2}\log\mu} 
 +\TT\mu^{-1}\log\mu
\eeq 
as $\eta_0=\eta_\tau=0$. As $\peps U_0=0$ and as $\abs{\peps \hat q_\mu}\leqs \abs{\peps q_{1,\mu}}$ is small on $\Om_0$ we can easily check that 
\beq
\abs{\int_{\Om_0}j_0\peps \Delta u_0}\leqs \nutot\norm{\beta} \mu^{-\half}\leq R_2.
\eeq
To prove the lemma it hence suffices to show that we can find a variation of  $q_1$ 
for which 
\beq
\label{eq:to-show-rational}
\int_{\Om_1}j_1 \peps\Delta u_1 \geq c_1\abs{d\beta(0)}^2 \nu_0-\err_2
\eeq
for an error term that is bounded by $\abs{\err_2}\leqs R_2$. 

To this end we will use that away from $z=0$ the functions $2(\mu z)^{-j}$ are well approximated by the first 
two components  $h_{j,\mu}(z):=\tfrac{2(\mu\bar z)^j}{1+\abs{\mu \bar z}^{2j}}$
of the harmonic map $\bar h_{j,\mu}(z):=\pi((\mu \bar z)^j):\Chat\to S^2$, where we continue to identify $(x,y)\in \R^2$ with $x+i y \in \C$ whenever convenient. 
We hence write 
$u_1(z)= U_1(q_0(z)+q_{1,\mu}(z))$ as 
\beqas
u_1(z)&=\bbu(0)+
 d\bbu(0)\bigg( \thalf\sum_{j=1}^{\dqs} \ajbb h_{j,\mu}
(z)+\rmbm(z)
\bigg) +f_1(z)+f_2(z)+f_3(z)
\eeqas
for 
\beqas
f_1(z)&:= \bb(z)-(\bbu(0)+d\bbu(0))(\rmmu)= \int_0^1 (d\bbu(t q_\mu)-d\bbu(0))( q_\mu) dt\\
f_2(z)&:= d\bbu(0))(\sum_{1\leq j\leq \dps} \ajbb ((\mu z)^{-j}-\thalf h_{j,\mu}
(z))= 
\thalf d\bbu(0) \big( \sum_{1\leq j\leq \dps}\ajbb
\tfrac{(\mu\bar z)^j}{\abs{\mu z}^{2j}(1+\abs{\mu z}^{2j})}\big)\\
f_3(z)&:=
d \bbu(0) \big( \sum_{j\geq \dps+1} \ajbb (\mu z)^{-j}\big).
\eeqas 
As $\Delta j_1=0$ we get that
 $T_{B}^i:=  \int_{\Om_1}j_1\Delta \peps f_i= 
 \int_{\partial \Om_1} \modbb \pn \peps f_i-\pn j_1 \peps f_i dS$ and will use this to show that all these terms are controlled by $R_2$.

We can use that 
$\abs{\partial \Om_1}\sim \mu^\mhalf$ and that the rational functions are bounded by \eqref{est:A-cut-q} and their variations by \eqref{est:A-cut-var-q} at points with $\abs{z}\sim\mu^\mhalf$. This immediately implies that 
\beqas
\abs{T_B^1}&\leqs \abs{d\beta(0)}\int_{\partial \Om_1}\abs{q_0}(\abs{\na q_\mu}\abs{\peps q_\mu}+\abs{q_\mu}\abs{\na\peps q_\mu})+\abs{\na q_0} \abs{\peps q_\mu}\abs{q_\mu} \leqs \abs{d\beta(0)} \mu^{-\half} \nutot\leqs R_2,
\eeqas
 while 
$\abs{T_B^3}\leqs \sitilt \mu^{-\half}\nu_{\rmbmhat}\leqs R_2$ follows since 
$\abs{\peps f_3}+\mu^\mhalf \abs{\peps \na f_3}\leqs \mu^{-(\dps+1)/2}\leqs \mu^\mhalf \nutot^\half$.

We can now use that terms of the form $d\beta(0)(a_j z^j)$, $a_j\in \C$, and their radial derivatives are $L^2(\partial \DD_r)$ orthogonal to terms of the form $dU_1(0)(b_k (\bar z)^k)$, $b_k\in \C$, and their radial derivatives on any circle $\partial \DD_r=\partial \DD_r(0)$ if $j\neq k$. 
This implies that 
\beqas
\abs{T_B^2}&\leqs \abs{d\beta(0)} \sum_{j=1}^{\dps}\abs{\ajbm} \abs{a_j(\peps q_1)} \int_{\partial \Om_1} \big[ \abs{z^j} 
\abs{\pn \tfrac{(\mu\bar z)^j}{\abs{\mu z}^{2j}(1+\abs{\mu z}^{2j})}}
+\abs{\pn z^j} \abs{\tfrac{(\mu\bar z)^j}{\abs{\mu z}^{2j}(1+\abs{\mu z}^{2j})}}\big] \\
&\leqs \sitilt \sum_{j\leq \dps} \abs{\ajbm} \abs{\peps \ajbb} \mu^{-2j} \leqs  \sitilt \mu^{-1}\nutot\leqs R_2.\eeqas 
At this stage we hence know that
\beqa\label{eq:E-E-case1-step1}
\int_{\Om_1}\modbb\peps\Delta \bb
&= \thalf 
\sum_{j=1}^{\dqs} \int_{\Om_1} \modbb \cdot  dU_1(0)(\peps \ajbb\Delta h_{j,\mu})+\err_3
\eeqa
for $\abs{\err_3}\leqs R_2$ since we know that $\Delta q_0=0$.

As
$\bar h_{j,\mu}$ is a harmonic map into $S^2$ 
and as 
$$\abs{\na \bar h_{j,\mu}}^2(z)=\abs{\na \pi((\mu \bar z)^j)}^2\cdot (j\mu)^2 \abs{\mu z}^{2(j-1)} =\mu^2 \tfrac{8j^2 \abs{\mu z}^{2j-2}}{(1+\abs{\mu z}^{2j})^2}$$
we can write 
\beqs 
\Delta h_{j,\mu}=- \abs{\na \bar h_{j,\mu}}^2 h_{j,\mu}=
-2\mu^2 H_j(\mu \abs{z})\cdot (\mu\bar z)^j
\eeqs
for  $H_j(t):= \tfrac{4 j^2t^{2j-2}}{(1+t^{2j})^3}$. 
The orthogonality of the different Fourier modes on circles hence allows us to write the main terms in \eqref{eq:E-E-case1-step1} as 
\beqa \label{est:main-term-rat}
\thalf 
\int_{\Om_1} \modbb \cdot  dU_1(0)(\peps \ajbb\Delta h_{j,\mu}) &=
-
 I_j^H(\mu)\int_{S^1}d\beta(0)(a_{j}(q_0)e^{ij\theta} )\cdot dU_1(0)(a_j(\peps q_1) e^{-ij\theta})\\
&= - \abs{a_{j}(q_0)} \abs{a_{1}(\peps q_1)} I_j^H(\mu)
I_j^\theta(\al_j, U_0,U_1)
\eeqa
for $\al_j=\al_j(\peps q_1,q_0):=j^{-1}\big[\Arg(a_{j}(\peps q_1))-\Arg(a_{j}(q_0))\big]$, 
\beq
\label{def:I-theta}
I_j^\theta(\al, U_0,U_1):= \int_{S^1}d(U_0-U_1)(0)(e^{ij\theta} )dU_1(0)( e^{-ij(\theta-\alpha)}) d\theta 
\eeq 
and
\beq
\label{eq:Ij-H} I_j^H(\mu):=\mu^2\int_0^{\mu^{-\half}} r^jH_j(\mu r) (\mu r)^j r dr
=4j^2\mu^{-j}\int_0^{\sqrt{\mu}}\tfrac{t^{4j-1}}{(1+t^{2j})^3} dt\sim \mu^{-j}.
\eeq
We note that the above calculation applies for any variation 
of $q_i$ that satisfies \eqref{ass:variations} and hence shows that $\peps E(\zz)$ is always controlled by \eqref{est:var-E-always-rational}.

To find a specific variation of $q_1$ for which  Lemma \ref{lemma:energy-expansion-rational} holds we now use the following lemma, a proof of which is included below.
\begin{lemma}\label{lemma:choosing-theta}
Let $U_{0,1}^*:\Chat\to N$ be distinct harmonic maps which satisfy Assumption \ref{ass:Ui}. 
Then there exists $\al^*\in [0,2\pi)$ and $c^*>0$ so that 
$-I_j^\theta(\al^*, U_0^*, U_1^*)>2c^* $ for all $j\in \N$ 
and hence so that 
$$I_j^\theta(\al^*, U_0, U_1)\geq c^* \text{ for all } U_i\in \RR^\si(U_i^*), \,j\in\N.$$
\end{lemma}

Given a rational map $q_0$ 
we will apply this lemma for 
$j_0\in \{1,\ldots, \dps\}$
chosen so that 
\beq
\label{def:choice-J-p}
\abs{\ajbm}\mu^{-j} \La_1^{j}\leq \abs{a_{j_0}(q_0)}\mu^{-j_0} \La_1^{\jrmbm} \text{ for } j=1,\ldots, \dps
\eeq
where  $\La_1\geq 1$ is a large number that only depends on $U_{0,1}^*$ and $q_{0,1}^*$ and that is fixed below.

Writing 
 $\rmbb$ in the form $q_1=\tfrac{r_1}{r_2}$ for polynomials $r_{1,2}$ which are normalised by $r_2(0)=1$, we finally define the desired variation of $q_1$ by 
\beqs
\rmbbeps(z):=
\frac{r_1(z)+\eps (e^{i\th^*} z)^{\jrmbm}}{r_2(z)} 
\eeqs
where we choose 
$\theta^*:=\al^*+\frac{1}{j_0}\Arg(a_{\jrmbm})$ so that $\al_{j_0}(\peps q_1,q_0)$ is given by the number 
 $\al^*$ from Lemma \ref{lemma:choosing-theta}.

As  $a_j(\peps q_1)=0$ for $j< j_0$ while $\abs{ a_{j_0}(\peps q_1)}=1$ 
we hence get from \eqref{eq:E-E-case1-step1} and 
\eqref{est:main-term-rat} 
that 
\beqas
\int_{\Om_1}\modbb\peps\Delta \bb
&\geq c^*\abs{a_{j_0}(q_0)}  I_{j_0}^H(\mu)-C  \sum_{j=\jrmbm+1}^{\dps}\abs{a_{j}(q_0)}I_{j}^H(\mu) -CR_2\\
&\geq 2c \abs{a_{j_0}(q_0)} \mu^{-j_0}-C  \sum_{j=\jrmbm+1}^{\dps}\abs{a_{j}(q_0)}\mu^{-j} -CR_2\\
&\geq (2c-C\La_1^{-1}) \abs{a_{j_0}(q_0)} \mu^{-j_0}-CR_2\geq c \abs{a_{j_0}(q_0)} \mu^{-j_0}-CR_2
\eeqas
for a constant $c>0$ that only depends on $U_{i}^*$ and $q_{0,1}^*$,
where the last step holds true provided  $\La_1=\La_1(U_{i}^*,q_{i}^*)$ in \eqref{def:choice-J-p} is chosen sufficiently large.  

Since  \eqref{def:choice-J-p} ensures that $\nu_{q_0}\leqs \abs{a_{j_0}(q_0)} \mu^{-j_0}$
 this completes the proof of 
 \eqref{eq:to-show-rational} and hence of 
 Lemma \ref{lemma:energy-expansion-rational} up to the 

\begin{proof}[Proof of Lemma \ref{lemma:choosing-theta}]
We use that harmonic spheres are weakly conformal and that $dU_i^*(0)\neq 0$.

If $U_1^*(z)=U_0^*(\bar z)$ we can hence 
choose 
coordinates on the target so that
$dU_{0,1}^*(0)$ are given
 by the matrices with columns $\tfrac{1}{\sqrt{2}}\abs{dU_0^*(0)}(1,0,0_{N-2})^T$ and $\tfrac{1}{\sqrt{2}}\abs{dU_0^*(0)}(0,\pm 1,0_{N-2})^T$ when viewed as maps from $\R^2$ to $\R^N$.
Setting $\al^*:=\pi$ hence gives 
$$-I_j^\theta(\al^*, U_0^*,U_1^*)=\thalf\abs{dU_0^*(0)}^2 \int_{S^1}\sin^2(j\theta)+\cos^2(j\theta)=\pi \abs{dU_0^*(0)}^2$$ which establishes the claim in this case. 

So suppose instead that $N$ is three dimensional and that $U_{0,1}^*$ are minimal spheres whose tangent spaces at $U_0^*(0)=U_1^*(0)$ intersect transversally.  
In this case we choose the coordinates on the domain so that 
$\partial_{x_1} U_1^*(0)$ is contained in the intersection of these tangent spaces 
and then use the conformality of $U_1^*$ to choose coordinates on the target so that $\partial_{x_{1,2}} U_1^*=c_{U_1^*}e_{1,2}$ and so that $T_{U_0^*(0)}N=\R^3\times \{0\}$. Here $e_{i}$ is the standard basis of $\R^N$ and $c_{U_1^*}:= \tfrac{1}{\sqrt{2}}\abs{dU_1^*(0)}$.

The matrix $A=c_{U_0^*}^{-1} d\bmu^*(0):\R^2\to \R^N$, whose columns are orthonormal, is hence so that 
 $A_{31}=0$, $A_{32}\neq 0$ and so that $A_{ij}=0$ for $i\geq 4$. 
As 
\beqas
\int dU_1^*(0)(e^{ij\theta}) dU_1^*(0)(e^{-ij(\theta-\al^*)})
&=\int \cos(j\theta)\cos(j\theta+\theta^*)-\sin(j\theta)\sin(j\theta+\al^*)
=0,\eeqas
for any $\al^*$  we hence get that 
\beqas
I_j^\theta(\al^*, U_0^*, U_1^*)
& =c_{U_1^*} c_{U_0^*}
\int [A_{11}\cos(j\theta) +A_{12}\sin(j\theta)]\cos (j\theta+\al^*)\\
&\quad -c_{U_1^*} c_{U_0^*}\int[A_{21}\cos(j\theta)+A_{22} \sin(j\theta)]\sin(j\theta+\al^*) d\theta\\
&=
\pi c_{U_1^*} c_{U_0^*}[ (A_{11}-A_{22})\cos\al^*-(A_{12}+A_{21})\sin\al^*].
\eeqas
If $A_{11}\neq A_{22}$ we can hence choose $\al^*$ to be either $0$ or $\pi$. Conversely, if $A_{11}=A_{22}$ then we cannot have that also $A_{12}=-A_{21}$ as both columns of $A$ have unit length and as 
$A_{31}=0$ while $A_{32}\neq 0$. In this second case we can thus choose $\al^*$ as either $\frac\pi2$ or $\tfrac{3\pi}2 $.   
\end{proof}

\subsection{Dominating terms in the energy expansion for variations of the underlying maps $U_i$}\label{sect:dom-other}
$ $\\
We now want to prove that we can always find variations of the maps $U_i$ in $\HH_1(U_i^*)$ so that the induced variation in $\ZZ_0$ is as described in Lemmas \ref{lemma:energy-expansion-rotation} respectively \ref{lemma:energy-expansion-tension}. 

For such a variation we can always bound $\abs{\peps \Delta u_1}\leqs \rho_ q^2$ and hence have 
$$\abs{\int_{\Om_1}j_1\peps \Delta u_1}\leqs \abs{d\beta(0)}\int_{\Om_1}\abs{\hat q_0} \rho_q^2\leqs  \abs{d\beta(0)}I_3\leqs \norm{\beta} \nutot,
$$
compare Lemma \ref{lemma:errors-energy-expansion}, while the analogue estimate for $\abs{\int_{\Om_0}j_0\peps \Delta u_0}$ follows by symmetry.

The general formula \eqref{eq:en-exp-nice} for the variation of the energy on $\ZZ_0$ hence tells us that 
\beqa 
dE(\zz)(\peps \zz)&=\pi c_{\mu} \tfrac{d}{d\eps}\de(U_{0},U_{1})^2
+ \deg(\rmbb) \tfrac{d}{d\eps}E(\bbu) + \deg(\rmbm) \tfrac{d}{d\eps} E(\bmu) + \err_4
\eeqa 
for $\abs{ \err_4}\leqs \norm{\beta} \nutot 
+\frac{\de}{\log\mu}[\de+\fmu^{-1}+(\log\mu)^{-1}]+\TT(\log\mu)^{-1}\fmu^{-1}$.  

As $c_\mu\sim (\log\mu)^{-1}$ we immediately get that Lemma \ref{lemma:energy-expansion-tension} holds true if we choose 
a variation 
 $U_{i,\eps}$ in $\HH_1^\si(U_i^*)$ for which \eqref{est:key-HH} holds. 

 On the other hand, if we consider 
variations of $U_i$ that are induced by translations on the domain then we know that $\ddeps E(U_i)=0$ as the energy is conformally invariant. To prove Lemma 
 \ref{lemma:energy-expansion-rotation}
 it hence suffices to prove 
Lemma \ref{rmk:de}, i.e. that there always exist variations 
 of the form $U_{i,\eps}=U_i(\cdot +\eps e^{i\theta_0})$ so that 
\beq\label{claim:dd-delta} 
\tfrac{d}{d\eps}\de^2(U_{0},U_{1})\geqs \de(U_{0},U_{1}).
\eeq
 If we are in a non-degenerate setting for which $U_{0}^*=U_1^*$ then we know that 
 $U_{i}(z)=U_0^*(z + c_{i})$ for some $\abs{c_i}\leq \si$.
If $c_0=c_1$ then $\de=0$ so there is nothing to prove, while for $c_0\neq c_1$ 
we can consider $U_{1,\eps}(z)=U_{0}^*(z+c_1+\eps \tfrac{c_0-c_1}{\abs{c_0-c_1}})= U_{1}(z+\eps \tfrac{c_0-c_1}{\abs{c_0-c_1}})$ and use that $dU_0^*(0)\neq 0$ to get a variation for which \eqref{claim:dd-delta} holds. The analogue argument also applies in settings where $U_0^*(z)=U_1^*(\bar z)$ and in both these cases we could just as well have constructed a variation of $U_0$.

So suppose that the maps $U_{0,1}^*$ are instead  as in part (ii) of Assumption \ref{ass:Ui}. In this case we can exploit that  the 
 tangent spaces $T_{U_i^*(0)}U_i^*(\Chat)$, which are 2 dimensional subspaces of the same 3 dimensional space $T_{U_i^*(0)}N$, intersect transversally
 and that the length  $\de$ of the geodesic $\hat\gamma_{U_1,U_0}$ that connects the points $U_i(0)$ can be made smaller than any given constant by reducing $\si$.
 
We hence obtain that  at least one of the angles 
$\al_i\in[0,\frac{\pi}{2}]$ at which $\hat\gamma$ intersects $T_{U_i(0)}U_i(\Chat)$ must be so that $\abs{\frac{\pi}{2}-\al_i}\geq \tfrac13 \bar\al$, where 
 $\bar \al\in (0,\frac{\pi}2]$ denotes
the 
angle between the spaces $T_{U_i^*(0)}U_i^*(\Chat)$.
For this $i$ we hence know that 
$\norm{P^{T_{U_i(0)}U_i(\Chat)}(\tfrac{\hat \gamma'(i^*)}{\norm{\hat \gamma'(i^*)}})}=\cos(\al_i)\geq \sin(\bar \al/3) $ 
is bounded away from zero, where we recall that $\hat\gamma(i^*)=U_i(0)$, $0^*=1, 1^*=0$. 

As $U_i^*$ is  weakly conformal and $dU_i^*(0)\neq 0$ and as the maps $U_i$ are $C^k$ close to $U_i^*$ we hence deduce that $U_i(\cdot +\eps e^{i\theta_0})$ has the desired property \eqref{claim:dd-delta} if $\theta_0$ is chosen so that 
$\peps U_i(0)( e^{i\theta_0})$ points in the direction of $ \pm P^{T_{U_i(0)}U_i(\Chat)}(\tfrac{\hat \gamma'(i^*)}{\norm{\hat \gamma'(i^*)}})$ . 

This completes the proof of Lemma \ref{rmk:de} and hence the proof of Lemma \ref{lemma:energy-expansion-rotation}.

\section{Estimates on the first and second variation of the energy}\label{sect:variations}
We conclude this paper with the proofs of the claims on the first and second variation of the energy at points $\zz\in\ZZ_0$ made in Lemmas \ref{lemma:positive-def-2ndvar}-\ref{lemma:second-variation}. We will continue to work in stereographic coordinates which are scaled so that $\abs{a_{n_0^*}(q_0)}=1$ and hence so that $\mu_0=1$, but we now need to 
 consider variations in general directions $w\in \Gamma (\zz^*TN)$. It is hence useful to observe that for any such $w$ we can bound 
\beq
\label{est:w-on-circle}
\int_{S^1} \abs{w(r e^{i\th})} d\th  \leqs [\log(\min(r^{-1},\mu r))]^\half\norm{w}_\zz \text{ for any } r\in [r_1,r_0]=[\mu^{-1}\fmu,\fmu^{-1}],
\eeq
as $\norm{\na w}_{L^2(\DD_1)}$, $
\int_{S^1}\abs{w (e^{i\theta})}$ and $\int_{S^1}\abs{w (\mu^{-1} e^{i\theta})}$ are all controlled by $\norm{w}_\zz$. 
In the following we can thus use that  
\beq
\label{est:weighted-norm-annulus}
\int_{A}\rho_\zz^2 \abs{w}^2 dx\leqs \fmu^{-2} \log\fmu \norm{w}_\zz^2 \text{ and } \int_A\tfrac1{r^2} \abs{w}^2 dx\leq C (\log\mu)^2\norm{w}_{z}^2
\eeq
which in particular implies that 
\beq
\label{est:v1-der-norm-w}
\int [\rho_z^2+\abs{\peps \na \gamma}^2+\abs{\peps \na v_i}^2]\abs{w}^2\leqs \norm{w}_{z}^2 \,\text{ while }\int_{A}\abs{\na \gamma}^2 \abs{w}^2\leqs \de^2\norm{w}_\zz^2,
\eeq
and we note that combined with \eqref{eq:second-var} this immediately gives \eqref{est:sec-var-general}.

In the proofs of Lemmas \ref{lemma:first-variation} and \ref{lemma:second-variation} below we will furthermore use that \eqref{est:w-on-circle} ensures that
\beq
\label{est:weighted-norm-A-star}
\int_{A^*}\tfrac1{r^2} \abs{w} dx\leq C (\log\fmu)^{\half} \norm{w}_\zz \text{ and that } \norm{w}_{L^1(\Acut)}\leq C \mu^{-1}(\log\mu)^\half.
\eeq

\subsection{Uniform definiteness of the second variation orthogonal to $\ZZ$}
\label{sec:definiteness-proof}
$ $\\
In this section we want to prove that the second variation is 
uniformly definite orthogonal to our set of singularity models as claimed in Lemma \ref{lemma:positive-def-2ndvar}.

The main step in the proof of Lemma \ref{lemma:positive-def-2ndvar} is to show that there exists a constant $c_0>0$ so that 
for any 
 $w\in T_\zz^\perp \ZZ$ with $\norm{w}_\zz=1$ there is an element $v\in  T_\zz^\perp \ZZ$ with $\norm{v}_\zz=1$  so that
 \beq
\label{claim:d^2E-definite}
d^2E(\zz)(w,v)\geq c_0.\eeq
Once we have shown this we can 
define 
$\VV_\zz^{\pm}$ as the span of the eigenfunctions  to positive respectively negative eigenvalues of the corresponding Jacobi operator 
 $$\tilde L_{\zz}:T_\zz^\perp\ZZ\to T_\zz^\perp\ZZ,$$ which is characterised by 
 $$d^2E(\zz)(v,w)=\langle  \tilde L_\zz v,w\rangle_\zz \text{ for } v,w\in T_\zz^\perp\ZZ$$
and obtain  the claim of the lemma from the spectral theorem for selfadjoint Fredholm operators.

 To prove this claim \eqref{claim:d^2E-definite} we will relate elements $\zz$ of $\ZZ_0$ and vector fields $v$ and $w$ along $\zz$ to the corresponding elements $\hat \zz =U_0\circ q_0$ of $ \HH^\si(\om_0)$ and to vector fields $\hat v$ and $\hat w$ along $\hat \zz$. 
 
For such vector fields we will always work with respect to the inner product 
 \beq
 \label{def:inner-prod-hat}
 \langle \hat w,\hat v\rangle_\pi:= \int_{\R^2} \abs{\na \pi}^2\hat w\hat v+\na \hat w\na \hat v
 \eeq
 which  (upto a factor $2$) corresponds to the $H^1$ inner product on the sphere and which approximates $\langle\cdot,\rangle_\zz$ well away from the origin.

We note that 
$d^2E(\om_0)$ is uniformly definite on $T_{\om_0}^\perp\HH^\si(\om_0)$ 
since $T_{\om_0}\HH^\si(\om_0)$ coincides with the kernel of the corresponding 
 Jacobi-operator $ L_{\om_0}:\Gamma(\om_0^*TN)\to \Gamma(\om_0^*TN) $ and since the eigenvalues of this self-adjoint Fredholm operator tend to infinity. 
As $\HH^\si(\om_0)$ is a finite dimensional manifold which is contained in a small $C^k$ neighbourhood of $\om_0$ we can then use the continuity of the Jacobi-operator to deduce that the analogue statement is also true for general elements of $ \HH^\si(\om_0)$ (provided $\si>0$ is chosen sufficiently small).

I.e.~we can use that there exists $c>0$ so that for any 
 $\hat \zz\in \HH^\si(\om_0)$ and any $\hat w\in T_{\hat \zz}^\perp \HH^\si(\om_0)$ with $\norm{\hat w}_{\pi}=1$ there exists a unit element $\hat v\in T_{\hat \zz}^\perp \HH(\om_0)$ 
 so that 
 \beq
\label{est:pos-d2E-bubble}
d^2E(\hat \zz)(\hat w,\hat v)\geq c.
 \eeq
In the proof below we will furthermore use 
that 
 variations $\zz_\eps$ in $\ZZ_0$ induced by variations of only $U_0$ and $q_0$ are well approximated by the corresponding variations $\hat \zz_\eps=U_0^{(\eps)}\circ q_0^{(\eps)}$ in $\HH^\si(\om_0)$. 
 To be more precise
we can easily check that 
\beq
\label{est:z-ustar-2}
\norm{
\abs{\zz-\hat \zz}^2+\abs{\peps(\zz-\hat \zz)}^2}_{L^\infty(\R^2\setminus \DD_{ r_0^2}) }+\int_{\R^2} \abs{\na(\zz-\hat \zz)}^2+\abs{\na \peps(\zz-\hat \zz)}^2=o(1)
\eeq 
where here and in the following we write that a quantity is given by $o(1)$ if we can ensure that it is smaller than any given positive number by choosing $\bar \mu$ sufficiently large and $\si$ and $\si_1$ sufficiently small.
Here we use that $\abs{1-\phi_\mu}\leqs \tfrac{\log\fmu}{\log\mu}=o(1)$  for $\abs{z}\geq r_0^2=\fmu^{-2}$ to bound the $L^\infty$ norm of $\peps (\zz-\hat\zz)$ and we note that we furthermore have
\beq
\label{est:z-star-1}
\norm{\peps \zz}_{\zz,\DD_{r_0}}+\norm{\peps \hat \zz}_{\pi,\DD_{r_0}}=o(1)
\eeq 
where $\norm{\cdot}_{\pi,\Om}$ and $\norm{\cdot}_{\zz,\Om}$ are defined as in  \eqref{def:inner-product} and \eqref{def:inner-prod-hat} but with the integrals taken over $\Om$.

With these preparations in place we now turn to the proof of
\eqref{claim:d^2E-definite}.

So let $\zz\in \ZZ_0$ and let $w\in T_\zz^\perp \ZZ$ be so that $\norm{w}_\zz=1$. 
We first remark that \eqref{est:weighted-norm-annulus} and \eqref{est:v1-der-norm-w} ensure that $\norm{w}_{\zz,A}+\norm{\abs{\na \zz} w}_{L^2(A)}=o(1)$ and hence that 
\beqas
d^2E(\zz)(w,w)&\geq \norm{w}_\zz^2-C \norm{w}_{\zz, \R^2\setminus A}^2-o(1).
\eeqas
The claim \eqref{claim:d^2E-definite} is thus trivially true (for $v=w$) if $\norm{w}_{\zz,\R^2\setminus A}$ is small and by symmetry we can hence assume that $\norm{w}_{\zz,\R^2\setminus \DD_{r_0}}^2\geq c_4$ for a small, but fixed constant $c_4>0$.

Given such a $w$ we first want to construct a 
$\hat w\in T_{\hat \zz}^\perp \HH^\si(\om_0)$, 
 so that 
\beq
\label{claim:hat-w}
\norm{w-\hat w}_{\pi,\R^2\setminus \DD_{s_0}}=o(1) \text{ while }  \norm{\hat w}_{\pi,\DD_{2s_0}}=o(1) \text{ and } 
\norm{w}_{\DD_{2s_0}\setminus \DD_{s_0/2}}=o(1)
\eeq 
for some $s_0\in [2r_0^2,r_0]$.
To this end we initially construct a function $w_1$ which vanishes on $\Om_1=\{z:\abs{z}\leq \mu^\mhalf\}$ and for which there is a radius $s_0\in [2r_0^2,r_0]$ so that 
$$w_1\equiv w \text{ on } \R^2\setminus \DD_{s_0} \text{ while }\norm{w_1}_{\zz,\DD_{2s_0}}=o(1)
.
$$
To obtain such a function $w_1$ we use that   
$\DD_{r_0}\setminus \DD_{r_0^2}$ contains $J\gtrsim \log(r_0^{-1})=\log\fmu$ disjoint annuli of the form $\DD_{4r}\setminus \DD_{\tfrac14 r}$ to select  $s_0\in
  [2r_0^2,\half r_0]$ for which 
\beqa \label{est:choosing_s}
\int_{\DD_{2s_0}\setminus \DD_{s_0/2}}\abs{\na w}^2 dx +\int_{S^1} \abs{\partial_\theta w(s_0 e^{i\theta})}^2
\leqs (\log\fmu)^{-1}=o(1).
\eeqa
On $\DD_{s_0}\setminus \DD_{s_0/2}$ we then define $w_1$ using a standard interpolation between $w\vert_{\partial \DD_{s_0}}$ and the mean value $\bar w$ over this circle which, thanks to \eqref{est:w-on-circle}, is bounded by $\abs{\bar w}\leqs (\log\fmu)^\half$. Combined with \eqref{est:choosing_s} and $\int_A\rho_\zz^2\leqs \fmu^{-2}$ 
this ensures that $\norm{w_1}_{\zz,\DD_{2s_0}\setminus \DD_{s_0^2}}=o(1)$. On $\DD_{s_0/2}\setminus \DD_{\mu^\mhalf}$ we can then choose $w_1$ as the harmonic function which transitions between $0$ and $\bar w$ and note that this map has  energy of order $O(\tfrac{\log\fmu}{\log\mu})=o(1)$ and weighted $L^2$ norm of order $O((\log\fmu)^\half \fmu^{-1})=o(1)$.

Given such a function $w_1$ we then set $w_2:= P_{\hat \zz} w_1$
to obtain a vector field along $\hat\zz$. We note that \eqref{est:z-ustar-2} ensures that 
$\norm{w_1-w_2}_{\pi,\R^2\setminus \DD_{s_0}}\sim \norm{w_1-w_2}_{\zz,\R^2\setminus \DD_{s_0}}=o(1)$ and that we still have $\norm{w_2}_{\zz,\DD_{s_0}}=o(1)$. 

Given a variation $\hat\zz_\eps$ in $\HH^\si(\om_0)$ with $\norm{\peps \hat \zz}_\pi =1$ we can hence consider the corresponding variation  $\zz_\eps$ of $\zz$ in $\ZZ$ and use \eqref{est:z-ustar-2}  and \eqref{est:z-star-1} as well as that $w\in P_\zz^\perp \ZZ$  to see that
$$\abs{\langle w_2,\peps \hat \zz\rangle_{\pi}}= \abs{\langle w,\peps \zz\rangle _{\zz,\R^2\setminus \DD_{s_0}}} +o(1)\leq \abs{\langle w,\peps \zz\rangle_{\zz}}+\epssimu=\epssimu.
$$
We thus deduce that $\norm{P^{T_{\hat\zz} \hat \ZZ} w_2}_{\pi}=o(1)$ and hence that $\hat w:= P^{T_{\hat\zz}^\perp\hat \ZZ} w_2$ satisfies \eqref{claim:hat-w}. 

As 
$\norm{\hat w}_{\pi}^2\geq c_4-\epssimu$ is bounded away from zero we can hence apply \eqref{est:pos-d2E-bubble} to obtain a vector-field $\hat v\in T_{\hat\zz}\HH^\si(\om_0)$ with $\norm{\hat v}_{\pi}\leqs 1$ so that 
\beq
\label{est:lower-d2E-vhat} 
d^2E(\hat \zz)(\hat w,\hat v)\geq 1
\eeq
and we will obtain the required element $v$ of $T_\zz^\perp\ZZ$ by modifying this vector field $\hat v$.

To this end we repeat the argument from above to obtain a $v_1$ 
with
\beq
\label{est:v1-properties}
v_1\equiv \hat v \text{ on } \R^2\setminus \DD_{2s_0}, \quad v_1\equiv 0 \text{ on } \DD_{\mu^\mhalf}, \quad \norm{v_1}_{\zz,\DD_{2s_0}\setminus \DD_{s_0/2}}\leqs 1 \text{ and } \norm{v_1}_{\zz,\DD_{s_0/2}}=o(1),
\eeq
except that we work with a radius $\hat s_0\in [s_0,2s_0]$ with $\int_{S^1}\abs{\partial_\theta \hat v(\hat s_0 e^{i\theta})}^2\leqs 1$ rather than with $s_0$. 

As $\hat v$ is tangential to $N$ along $\hat \zz$ and supported on $\Om_0$ where $\rho_\zz\sim \abs{\na \pi}$  we can combine \eqref{est:z-ustar-2} and \eqref{est:v1-properties} to see that
\beqas
\norm{P_{\zz}v_1-v_1}_{\zz}\leqs \norm{v_1}_{\zz,\DD{s_0/2}}
+\norm{\zz-\hat\zz}_{L^\infty(\R^2\setminus \DD_{r_0^2})} \norm{v_1}_\zz+\norm{\abs{\na \zz-\na\hat \zz} v_1}_{L^2(\R^2\setminus \DD_{r_0^2})}=o(1)
\eeqas
where we use that $\abs{\na \zz-\na\hat\zz}\leqs o(1)(\rho_\zz+\abs{\na \phimu})$ on $\R^2\setminus \DD_{r_0^2}$ to deal with the last term.

As $\hat v$ is orthogonal to $T_{\hat \zz}\HH^\si(\om_0)$ we can combine this bound with \eqref{est:z-ustar-2} and \eqref{est:z-star-1} to see that
$\langle \peps \zz_\eps,P_\zz v_1\rangle_\zz=o(1)$ for all variations $\zz_\eps$ in $\ZZ$ that are induced by variations of only $U_0$ and $q_0$. Conversely, variations of $U_1$ and $q_1$ result in variations $\zz_\eps$ for which $\norm{\peps \zz_\eps}_{\zz,\Om_0}=o(1)$ and for which we hence trivially know that $\langle \peps \zz_\eps,P_\zz v_1\rangle_\zz=o(1)$. 

All in all this ensures that $v:=P^{T^\perp_\zz \ZZ}(P_\zz v_1)\in T_{\zz}^\perp \ZZ$ is so that 
\beqs
\norm{v}_{\zz, \DD_{s_0/2}}=o(1), \quad \norm{v}_{\zz, \DD_{2s_0}\setminus \DD_{s_0/2}}\leqs 1 \text{ and } 
\norm{v-\hat v}_{\zz, \R^2\setminus \DD_{2s_0}}=o(1).
\eeqs
We now write for short $II_{\zz}(w,v) =\na w\na v-A( \zz)(\na  \zz,\na \zz)A(\zz)( w,v)$ for the integrand that appears in the formula 
\eqref{eq:second-var} for $d^2E$ 
and note that $\int_{\Om}\abs{II_\zz(w,v)}\leqs \norm{w}_{\zz,\Om} \norm{v}_{\zz,\Om}+o(1)$ for every $\Om\subset \R^2$ where the $o(1)$ term comes from the contribution of $\na \gamma$ to $\na \zz$, compare \eqref{est:v1-der-norm-w}. We hence get
\beqas
d^2E(\zz)(v,w)&\geq \int_{\R^2\setminus \DD_{2s_0}} II_\zz(w,v)-C\big[\norm{v}_{\zz,\DD_{s_0/2}}+\norm{w}_{\zz,\DD_{2s_0}\setminus \DD_{s_0/2}}\big]-o(1)\\
&\geq \int_{\R^2\setminus \DD_{2s_0}} II_{\hat\zz}(\hat w,\hat v)-C\big[\norm{v-\hat v}_{\zz,\R^2\setminus \DD_{2s_0}}+\norm{w-\hat w}_{\zz,\R^2\setminus \DD_{2s_0}}+\norm{\zz-\hat\zz}_{L^\infty(\R^2\setminus \DD_{2s_0})}\big]-o(1)\\
&\geq\int_{\R^2} II_{\hat\zz}(\hat w,\hat v)-C\norm{\hat w}_{\pi,\DD_{2s_0}}-o(1)\\
&\geq d^2E(\hat \zz)(\hat w,\hat v)-o(1)\geq 1-o(1)\geq \thalf ,
\eeqas 
where the last inequality holds true provided $\si$ and $\si_1$ are sufficiently small and $\bar \mu$ is sufficently large.
As $\norm{v}_{\zz}\leqs 1$ we hence obtain that \eqref{claim:d^2E-definite} holds true for the corresponding unit element of $T_\zz^\perp \ZZ$ which completes the proof of Lemma \ref{lemma:positive-def-2ndvar}.

\subsection{Proofs of Lemmas \ref{lemma:first-variation} and \ref{lemma:second-variation}}
$ $\\
We finally turn to the proofs of the estimates on the norms of the first and second variation of $E$ claimed in Lemmas \ref{lemma:first-variation} and \ref{lemma:second-variation}.
To simplify the notation we write for short 
\beq
\label{def:R3}
R_3:= \Ctilt \nutot(\log\mu)^\half+\tfrac{\de}{\fmu\log\mu }
\eeq
and 
\beq
\label{def:R4}
R_4:= 
 (\Ctilt+\norm{\peps \beta})\nutot\log(\mu)^{\half} 
 +\eta_{rat}\norm{\beta} \mu^{-j^*}\log(\mu)^{\half}+\tfrac{(\de+\eta_0)(\log\fmu)^\half}{\log\mu}
\eeq 
and recall that to prove  Lemma \ref{lemma:first-variation} we need to show that 
\beqs
\abs{dE(\zz)(w)}=\abs{\int\tau(\zz)w}\leqs R_3+\TT \text{ for every } w \text{ with } \norm{w}_\zz=1
\eeqs
while Lemma \ref{lemma:second-variation} asserts that 
\beqs
\abs{d^2E(\zz)(\peps \zz, w)}=\abs{\int \peps (\tau(\zz)) w}\leqs R_4+\TT+\eta_\TT \text{ for every } w\in \Gamma(\zz^*TN) \text{ with } \norm{w}_\zz=1.
\eeqs
By symmetry it suffices to bound the corresponding integrals over $\Om_1$ where
$\zz=\zz_1+\errcut$ for  $\zz_1=\pi_N({v_1})$ and for $\errcut$ which is defined by \eqref{eq:errphi} and supported on $\Acut$. 

To prove both lemmas we write the tension of $\zz$ on $\Om_1$ 
as 
\beq
\label{eq:write-tau}
\tau(\zz)=P_{\zz}(\Delta \zz_1+\Delta \errcut)=\tau(\zz_1)+(P_{\zz}-P_{\zz_1})(\Delta \zz_1)+P_{\zz}(\Delta \errcut),\eeq
and use that $\Delta j_1=0$ to write 
\beqa
\label{eq:tension-formula}
\tau(\zz_1)=
P_{\zz_1}(d\pi_N({v_1})\Delta {v_1})+P_{\zz_1}(d^2\pi_N({v_1})(\na {v_1},\na {v_1}))=T_{u_1}^\tau+T_{u_1}^A+T_\gamma+T_\pi
\eeqa
where we split the contribution of $\Delta u_1$ into the terms
\beq
\label{eq:sunny-term1}
T_{u_1}^\tau:= 
P_{\zz_1}(d\pi_N({v_1}) (\tau(\bb)) \text{ and } T_{u_1}^A:= -P_{\zz_1}(d\pi_N({v_1}))(A(u_1)(\na u_1,\na u_1)))\eeq and  write for short 
\beq
\label{eq:sunny-term2}
T_\gamma:=
P_{\zz_1}(\Delta\gamma) \text{ and } 
T_\pi:= P_{\zz_1}(d^2\pi_N({v_1})(\na {v_1},\na {v_1}))
.\eeq
In the analysis of the resulting integrals we will use that 
\beq
\label{est:z1minusv1} 
\abs{\zz_1-v_1}\leq \abs{u_1-v_1}\leq \abs{\tilde\gamma_{u_1}}+\abs{j_1}\leqs \de\phimu +\norm{\beta}\abs{\rmbmhat}\eeq
which, thanks to \eqref{est:int-quadratic-u-om}, \eqref{est:int-bm-squared} and \eqref{est:I-phimut}, ensures that
\beqa
\label{est:proof-1var-1}
\norm{\abs{v_1-\zz_1}\rho_{q}}_{L^2(\Om_1)}\leq 
\norm{\abs{u_1-v_1}\rho_{q}}_{L^2(\Om_1)}
&\leqs R_3.
\eeqa
Similarly, since we can bound 
\beqa
\label{est:peps-z1minusv1} 
\abs{\peps(u_1-v_1)}&\leqs
 \abs{\peps j_1}+\abs{\peps\tgabb} +\abs{\tilde\gamma_{u_1}}+\abs{j_1}\\
 &
 \leqs (\norm{\beta}+\norm{\peps \beta}) \abs{\rmbmhat}+\norm{\beta}\abs{\peps \rmbmhat}+(\de+\eta_0)\phi_\mu +\de(\log\mu)^{-1}\chiA,
 \eeqa compare also  \eqref{est:var-gamma}, we can use these estimates \eqref{est:int-quadratic-u-om}, \eqref{est:int-bm-squared} and \eqref{est:I-phimut} together with 
 \eqref{est:int-for-later} to see that 
\beq\label{est:proof-2var-1}
\norm{ \abs{\peps(u_1-v_1)}\rho_q}_{L^2}\leqs R_4.
\eeq 
To analyse terms that involve $\gamma$ we instead write 
$v_1=\gamma+\tilde u_1+j_1$ to see that
$$\abs{v_1-\gamma}+\abs{\peps(v_1-\gamma)}\leqs \abs{\hat q_0}+\abs{\hat q_{1,\mu}}+
\abs{\peps \hat q_0}+\abs{\peps \hat q_{1,\mu}}
\leqs \abs{z}+(1+\mu^2\abs{z}^2)^{-1}$$
and hence that also 
\beq
\label{est:proof-1var-3}
\norm{\abs{v_1-\gamma} \na \gamma}_{L^2}+\norm{\abs{\peps(v_1-\gamma)}\na \gamma}_{L^2} \leq \tfrac{\de}{\fmu \log\mu}\leqs R_3.
\eeq 
Finally since $\abs{P_{\zz_1}\na v_1-\na v_1}\leqs \abs{u_1-\zz_1}\rho_q+\Ctilt \abs{\na \rmbmhat}+\abs{v_1-\gamma}\abs{\na \gamma}$ we also get that 
\beq
\label{est:proof-1var-4}
\norm{P_{\zz_1}\na v_1-\na v_1}_{L^2}\leqs
R_3.
\eeq 
With these estimates in place we can now show that all contributions to $dE(\zz)(w)$ are controlled by $R_3+\TT$ and all contributions to $d^2E(\zz)(\peps \zz,w)$ are controlled by $R_4+\TT+\eta_\TT$. 

To begin with we use that $\abs{\tau(u_1)}\leqs \TT \rho_q^2\leqs  \TT\rho_\zz^2$ to see that  
$$\int_{\Om_1} \abs{T_{u_1}^\tau w}\leqs \TT \norm{w}_\zz=\TT$$
while we can use that 
$\abs{\peps T_{u_1}^\tau}\leqs \abs{\tau(u_1)}+\abs{\peps \tau(u_1)}\leqs \TT\rho_q\rho_\zz+\eta_\tau\rho_q^2\leqs  (\TT+\eta_\tau)\rho_\zz^2$ to bound 
\beqs
\abs{\int T_{u_1}^\tau w}\leqs \TT+\eta_\tau.
\eeqs
As $A(u_1)(\na u_1,\na u_1)\in T_{u_1}^\perp N$ we can then write  
\beq
\label{eq:sunny-1}
T_{u_1}^A= 
P_{\zz_1}\big((d\pi_N(u_1)-d\pi_N({v_1}))(A(u_1)(\na u_1,\na u_1))\big)
\eeq
and hence use 
 \eqref{est:proof-1var-1} to bound
\beqas
\int_{\Om_1} \abs{T_\bb w} 
&\leqs \norm{\abs{u_1-z_1}\rho_q}_{L^2}\leqs 
 R_3.
\eeqas
From \eqref{eq:sunny-1} we see that 
 $\abs{\peps T_{u_1}^A}\leqs \big[\abs{u_1-v_1}+\abs{\peps(u_1-v_1)}\big] \rho_q \cdot  \rho_\zz$ which we can combine with \eqref{est:proof-1var-1} and \eqref{est:proof-2var-1}
to see that also 
 \beqas 
 \abs{\int \peps T_{u_1}^A w}&\leqs \norm{ \abs{u_1-v_1}\rho_q}_{L^2}+
\norm{ \abs{\peps(u_1-v_1)}\rho_q}_{L^2}\leqs R_3+R_4\leqs R_4.
\eeqas
Similarly, we can write  
\beq
\label{eq:sunny-2}
T_\gamma=P_{\zz_1}((d\pi_N(\gamma)-d\pi_N({v_1}))(A(\gamma)(\na \gamma,\na \gamma)))+P_{\zz_1}(d\pi_N({v_1})\tau(\gamma)))
\eeq and use  
\eqref{est:der-gamma}, 
\eqref{est:v1-der-norm-w}, \eqref{est:weighted-norm-A-star} 
and \eqref{est:proof-1var-3} 
to see that
\beqa 
\abs{\int T_\gamma w}&\leq \de \norm{\abs{\gamma-v_1}\abs{\na \gamma}}_{L^2}+\de\cmu\int_{A^*}\abs{w} r^{-2}\leqs \tfrac{\de^2}{f_\mu \log\mu }+\tfrac{\de (\log\fmu)^\half}{\log\mu}\leqs R_3.\eeqa
As  
\beqa\label{est:var-ten-ga}
\abs{\peps \tau(\gamma)}\leqs (\etazero+\de^2\cmu)\abs{\Delta \phimu}+\de \abs{\peps \Delta \phimu} \leqs (\etazero+\de) \cmu r^{-2}\chiAstar
\eeqa
we can furthermore use
\eqref{est:v1-der-norm-w},  \eqref{est:weighted-norm-A-star}, \eqref{est:proof-1var-3} and 
 \eqref{est:der-gamma} to see that 
\beqas
\abs{\int \peps T_\gamma w} &\leqs \norm{\abs{\gamma-v_1}\abs{\na \gamma})}_{L^2}+\norm{\abs{\peps(\gamma-v_1)}\abs{\na \gamma}}_{L^2}+\int \abs{\tau(\gamma)} \abs{w} +\abs{\peps \tau(\gamma)}\abs{w}\\
&\leqs \tfrac{\de+\eta_0}{(\log\mu)\fmu}
+
\tfrac{\de+\eta_0}{(\log\mu)}\int_{A^*}r^{-2} \abs{w}\leqs \tfrac{\de+\eta_0}{(\log\mu)} (\log\fmu)^\half\leqs R_4
\eeqas 
where the penultimate step follows from \eqref{est:weighted-norm-A-star}.

Finally we use \eqref{eq:d2pi-A}  to write 
\beq 
\label{eq:sunny-3}
T_\pi=P_{\zz_1}(d^2\pi_N({v_1})(\na {v_1},\na {v_1})-d^2\pi_N({\zz_1})(P_{\zz_1} \na {v_1},P_{\zz_1}\na {v_1})).
\eeq
From  \eqref{est:v1-der-norm-w}, \eqref{est:proof-1var-1},  \eqref{est:proof-1var-3} and  \eqref{est:proof-1var-4} we hence get that also 
$$\int \abs{T_\pi w}
\leqs \norm{\abs{v_1-\zz_1}\rho_q}_{L^2}+\norm{\abs{v_1-\zz_1}\na\gamma}_{L^2}+\norm{P_{\zz_1}\na v_1-\na v_q}_{L^2}\leqs R_3.
$$
From \eqref{eq:sunny-3} we see that 
\beqas
\abs{\int \peps  T_\pi w}&\leqs \norm{\abs{v_1-\zz_1}\na v_1}_{L^2}
+ \norm{\abs{\peps(v_1-\zz_1)}\na v_1}_{L^2}+ \norm{P_{\zz_1}\na v_1-\na v_1}_{L^2}+\int\abs{\peps(P_{\zz_1}^\perp \na v_1)}\abs{\na v_1} \abs{w}.
\eeqas
As we can write
 $P_{\zz_1}^\perp\na v_1=(P_{\zz_1}^\perp-P_{u_1}^\perp)(\na u_1)+(P_{\zz_1}^\perp-P_{\gamma}^\perp)(\na \gamma)+(P_{\zz_1}^\perp-P_{U_1(0)}^\perp)\na j_1$ we hence also get that 
 $\abs{\int \peps  T_\pi w}\leqs R_4$.
 
All in all we hence get that  
 $$\int_{\Om_1} \abs{\tau(\zz_1) w}\leqs R_3+\TT \text{ and }
 \int_{\Om_1} \abs{\peps \tau(\zz_1) w}\leqs R_4+\TT+\eta_\TT.
$$
To complete the proofs of Lemmas \ref{lemma:first-variation} and \ref{lemma:second-variation} it remains to bound 
the contributions of the last two terms in \eqref{eq:write-tau}. These are supported on $\hat A$ where $\rho_\zz\leqs 1$, $\tau(\gamma)=0$ and hence  
  \beq
  \label{est:DElta-zhat}
  \abs{\Delta\zz_1}\leqs \mu (\nutot+\de^2(\log\mu)^{-2})
  \eeq
  so in particular 
$\abs{\Delta \zz_1} \leqs 1+ \tfrac{\de^2}{(\log\mu)^2}\mu $. 
Combined with \eqref{est:errcut} and \eqref{est:weighted-norm-A-star} this shows that
$$\abs{\int (P_\zz-P_{\zz_1})(\Delta \zz_1)w}\leqs \norm{\beta}
\nutot \big[1+ \tfrac{\de^2}{(\log\mu)^2}\mu \big]\norm{w}_{L^1(\Acut)}
\leqs \norm{\beta}
\nutot \mu^{-1}(\log\mu)^{\half}+\tfrac{\de^2}{(\log\mu)^{3/2}}$$
only gives a lower order contribution, while  \eqref{est:errcut} and \eqref{est:weighted-norm-A-star} also allow us to bound
$$\int \abs{\Delta \errcut}\abs{w}\leqs\Ctilt \nutot \int_\Acut \abs{w}\leqs \Ctilt \nutot(\log\mu)^\half\leqs R_3.$$
This completes the proof of Lemma \ref{lemma:first-variation}

Combining 
\eqref{est:DElta-zhat} with
\eqref{est:errcut-eps-rough} gives a pointwise bound of $\abs{\peps \errcut}\abs{\Delta \zz_1} \leqs R_4$ while combining
 $\abs{\peps\Delta\zz_1}\leqs \mu (\nutot +\eta_{rat}  \mu^{-j^*})+\mu \tfrac{\de^2+\de\eta_0}{(\log\mu)^{2}}$
with \eqref{est:errcut} ensures that also $\abs{\errcut}\abs{\peps\Delta \zz_1}\leqs R_4$ on $\hat A$. 
Combined with \eqref{est:weighted-norm-A-star} we can hence obtain that 
\beqas
\abs{\int \peps((P_{\zz_1}-P_{\zz})\Delta \zz_1) w}
& \leqs R_4 \mu^{-1}(\log\mu)^\half\ll R_4
 \eeqas
 just gives a lower order contribution to $d^2E(\zz)(\peps \zz,w)$. 

Finally \eqref{est:errcut}, \eqref{est:errcut-peps-Delta} and \eqref{est:weighted-norm-A-star} allow us to also bound
\beqa
\int \abs{\peps(P_\zz \Delta \errcut)}\abs{w}&\leqs ( \norm{\peps \Delta \errcut}_{L^\infty(\hat A)}+ \norm{ \Delta \errcut}_{L^\infty(\hat A)} )\norm{w}_{L^1(\Acut)}\\
&\leqs  \big[(\Ctilt +\norm{\peps \beta})\nutot+\eta_{rat} \norm{\beta} \mu^{-j^*}\big] (\log\mu)^\half\leqs R_4,
\eeqa
which completes the proof of Lemma \ref{lemma:second-variation}.

M. Rupflin: Mathematical Institute, University of Oxford, Oxford OX2 6GG, UK\\
\textit{melanie.rupflin@maths.ox.ac.uk}

\end{document}